\documentclass[11pt]{article}
\usepackage{amsthm,amsmath,amssymb,amsfonts}
\usepackage{color}
\usepackage{graphicx}
\usepackage{BOONDOX-cal}
\usepackage{euscript,mathrsfs}
\usepackage{url}

\newcommand{\dx}{{\, \rm d}x}

\newcommand{\Div}{{\rm div}\,}
\newcommand{\aleq}{\stackrel{<}{\sim}}
\newtheorem{thm}{Theorem}

\newtheorem{lem}[thm]{Lemma}
\newtheorem{prop}[thm]{Proposition}
\newtheorem{cor}[thm]{Corollary}
\newtheorem{df}{Definition}
\newtheorem{rmk}{Remark}

\newcommand{\bfphi}{\boldsymbol{\varphi}}

\newcommand{\Ov}[1]{\overline{#1}}
\newcommand{\DC}{C^\infty_c}

\newcommand{\vr}{\varrho}

\newcommand{\vu}{\vc{u}}

\newcommand{\vc}[1]{{\bf #1}}
\newcommand{\vcg}[1]{{\pmb #1}}
\newcommand{\F}[1]{$\mathbb{#1}$}
\newcommand{\Grad}{\nabla}

\newcommand{\tn}[1]{\mbox {\F #1}}

\newcommand{\dt}{\, {\rm d} t }

\newcommand{\intO}[1]{\int_{\Omega} #1\dx}

\newcommand{\vv}{\vc{v}}

\newcommand{\ep}{\varepsilon}

\newcommand{\R}{\mathbb{R}}

\begin{document}

\title{Weak solutions for a bi-fluid model for a mixture of two compressible
non interacting fluids with general boundary data}
\author{Stanislav Kra\v cmar
\thanks{S.K.  has been supported by the Czech Science Foundation
(GA\v CR) project GA19-04243S within RVO:67985840.}   
\and
Young-Sam Kwon
\thanks{ { The work of Y-S.K. was
supported by the National Research Foundation of Korea (NRF2020R1F1A1A01049805)}}
\and
\v S\' arka Ne\v casov\' a \thanks{ \v S. N.  has been supported by the Czech Science Foundation
(GA\v CR) project GA19-04243S. The Institute of Mathematics, CAS is
supported by RVO:67985840. }
\and
Anton\'{i}n Novotn{\' y}
\thanks{The work of the fourth author was partially supported by the distinguished Edurad \v Cech visiting program at the
Institute of Mathematics of the Academy of Sciences of  the Czech Republic.}}
\maketitle

\medskip
\centerline{Department of Mathematics, Czech Technical University,
Faculty of Mechanical Engineering,}

\centerline{Resslova 307, 110 00 Praha, Czech Republic, kracmar@marian.fsik.cvut.cz }
\medskip

\centerline{Department of Mathematics, Dong-A University}
\centerline{Busan 49315, Republic of Korea, ykwon@dau.ac.kr}

\medskip
\centerline{Institute of Mathematics of the Academy of Sciences of the Czech Republic,}
\centerline{\v Zitn\'a 25, 115 67 Praha, Czech Republic, matus@math.cas.cz}

\medskip
\centerline{University of Toulon, IMATH, EA 2134,  BP 20139}
\centerline{839 57 La Garde, France, novotny@univ-tln.fr}
\begin{abstract}
We prove global existence of weak solutions for a version of one velocity  Baer-Nunziato system with dissipation describing a mixture of two  non interacting
viscous compressible fluids in a piecewise regular Lipschitz domain with general inflow/outflow boundary conditions. The geometrical setting is  general enough to comply
with most current domains important for applications as, for example, (curved) pipes of picewise regular and axis-dependent cross sections. 

{ As far as the existence proof is concerned, we adapt to the system the nowaday's classical Lions-Feireisl approach to the compressible Navier-Stokes equations which is combined with a generalization of the theory of renormalized solutions to the transport equations in the spirit of Vasseur-Wen-Yu. The results related to the families of transport equations presented in this paper extend/improve some of statements of the theory of renormalized solutions, and they are therefore of independent interest.}
\end{abstract}


\smallskip


\section{Introduction} \label{se1}
 One of the acceptable models for description of mixture of several
compressible fluids is the so called {\it two velocity  Baer-Nunziato model}.
The equations of the  Baer--Nunziato  model with dissipation   read
(cf. \cite{BaNu}, \cite{DA}, \cite[Section 1]{Guil}):
$$
\partial_t{\alpha}_{\pm}+\vc v_I\cdot\Grad{\alpha}_{\pm}=0,
$$
$$
\partial_t(\alpha_\pm\vr_\pm)+{\rm div}(\alpha_\pm\vr_\pm\vu_\pm)=0,
$$
$$
\partial_t(\alpha_\pm\vr_\pm\vu_\pm)+{\rm div}(\alpha_\pm\vr_\pm\vu_\pm\otimes\vu_\pm) +\Grad (\alpha_\pm P_\pm(\vr_\pm))-P_I\Grad(\alpha_\pm)
$$
$$
=
\alpha_\pm\mu_\pm(\Delta\vu_\pm)+\alpha_\pm(\mu_\pm+\lambda_\pm)\Grad{\rm div}\vu_\pm
$$
$$
0\le\alpha_\pm\le 1,\quad\alpha_+ + \alpha_-=1.
$$
In the above $(\alpha_\pm,\alpha_\pm\vr_\pm\ge 0,\vu_\pm\in \R^3)$ -concentrations, densities, velocities of the $\pm$ species - are unknown functions of time-space { $(t,x)=Q_T:=I\times \Omega$, $t\in I=(0,T)$, $T>0$,} and $x\in\Omega\subset \R^3$, $P_\pm$ are two (different) given functions defined on $[0,\infty)$ and
$P_I$, ${\vc v}_I$ are conveniently chosen  quantities - they represent pressure and velocity at the interface. In the multifluid modeling, there are many possibilities how the quantities ${\vc v}_I$, $P_I$ could be chosen, and there is no consensus about this choice.

Our goal in this paper is to prove the existence of weak solutions for the Baer-Nunziato system with dissipation on an arbitrary large time interval $(0,T)$ in a Lipschitz bounded domain $\Omega$,
under the following
simplifying assumptions:
\begin{equation}\label{1}
\mu_\pm:=\mu,\;\lambda_\pm:=\lambda,\; \vc v_I=\vu_\pm:=\vu
\end{equation}
\begin{equation}\label{2}
\alpha P_\pm(s)= {\cal P}_\pm(f_\pm(\alpha)s)\;\mbox{for all $\alpha\in (0,1)$, $s\in [0,\infty)$}
\end{equation}
with some functions ${\cal P}_\pm$ defined on $[0,\infty)$ and functions $f_\pm$ defined on $(0,1)$.

With this simplifications, the two velocity Baer-Nunziato system reduces to the following system (which we will call the {\it one velocity  Baer-Nunziato type system}):
\begin{equation}
\label{e1}
\partial_t\alpha+(\vu\cdot \nabla)\alpha=0,\quad
0\le\alpha\le 1,
\end{equation}
\begin{equation}
\label{e2}
\partial_t\vr+\mbox{div}(\vr \vu)=0,
\end{equation}
\begin{equation}
\label{e3}
\partial_tz+\mbox{div}(z\vu)=0,
\end{equation}
\begin{equation}
\label{e4}
\partial_t((\rho+z)\vu)+\mbox{div}((\rho+z)\vu\otimes \vu)
+\nabla P( f(\alpha)\vr, g(\alpha)z)={\rm div} \tn S(\Grad\vu)
\end{equation}
Here $P:[0,\infty)^{2}\mapsto [0,\infty)$ as well as $f,g:(0,1)\mapsto [0,\infty)$  are given functions, and
$$
{\tn S(\tn Z)}=\mu(\tn Z+{\tn Z}^T)+\lambda{\rm Tr}(\tn Z)\tn I.
$$
($\tn I$ is the identity tensor, {\rm Tr} denotes the trace) is the viscous stress tensor.
 The constant viscosity coefficients
satisfy standard physical assumptions, $\mu>0$, $\lambda+\frac 23\mu\ge 0$.
The system is endowed with initial conditions
\begin{equation}
\label{e5}
 \alpha|_{t=0}=\alpha_0,\; \vr|_{t=0}=\vr_0,\; z|_{t=0}=z_0,\; (\vr+z)\vu|_{t=0}=(\vr_0+z_0)\vu_0,
\end{equation}
{
We consider the  general inflow-outflow boundary conditions,
\begin{equation}
\label{e6}
\vc{u} |_{\partial \Omega} = \vu_B, \ \vr|_{\Gamma^{\rm in}} = \vr_B,\ z|_{\Gamma^{\rm in}} = z_B,\ \alpha|_{\Gamma^{\rm in}} = \alpha_B,
\end{equation}
where
\begin{equation}
\label{e6a}
\Gamma^{\rm in} = \left\{ x \in \partial \Omega \ \Big| \ {\vc{u}_B} \cdot \vc{n} < 0 \right\}.
\end{equation}
}
For further use, we also define
\begin{equation}\label{e6b}
\Gamma^{\rm out} = \left\{ x \in \partial \Omega \ \Big| \ {\vc{u}_B} \cdot \vc{n} > 0 \right\},\; 
\Gamma^{0} = {\rm int}_2\left\{ x \in \partial \Omega \ \Big| \ {\vc{u}_B} \cdot \vc{n} = 0\;\mbox{or}\;\vc n(x)\;\mbox{does not exist}\,\right\}.
\end{equation}
{ Here and in the sequel, ${\rm int}_2A$
denotes the interior of $A\subset\partial\Omega$ with respect to the trace topology of $\R^3$ on $\partial\Omega$.}

Assumption (\ref{2}) is certainly true in the classical situation of two isentropic gases when
\begin{equation}\label{gammaP}
{P}_\pm(s)=a_\pm s^{\gamma^{\pm}}, \;\gamma_\pm>0;
\end{equation}
indeed, in this case
\begin{equation}\label{Pg}
P(R,Z)=a_+ R^{\gamma^+}+a_-Z^{\gamma^-},\;f(s):=f_+(s)={s}^{\frac{1}{\gamma^+}-1},\;g(s):=f_-(s)={(1-s)}^{\frac{1}{\gamma^-}-1}.
\end{equation}
We shall however be able to treat in system (\ref{e1}--\ref{e5}) more general functions $P,f,g$ than those being given by
(\ref{Pg}).

System (\ref{e1}--\ref{e4}) belongs to the family of multi-fluid models with differential closure, cf. Ishii, Hibiki \cite{ISHI},
Drew, Passman \cite{DRPAS}. It is not without interest that it can be viewed as a barotropic counterpart of the so called five-equation
bi-fluid model derived in Allaire, Clerc, Kokh  \cite{ACK1}, \cite{ACK2} or by Guillard, Murrone \cite{GuMu} by different considerations.

{ The mathematical literature dealing with these types of models
is in a short supply and none of it deals with the general boundary data. We quote a few papers dealing with no-slip or periodic or
slip boundary conditions for related problems: \cite{BreschMF}, \cite{BRZA}, \cite{EV1}, \cite{EV2}, \cite{3MNPZ},
\cite{NoSCM}, \cite{AN-MP}, \cite{QWE}, \cite{VWY}.}
In particular, existence of weak solutions for the problem
(\ref{e1}--\ref{e5})
under quite general assumptions on constitutive functions $P$ and $f,g$ is known in the case of no-slip boundary conditions ($\vu|_{\partial\Omega}=0$), see \cite{NoSCM}. {\it The main goal and achievement of this work is to treat the general non-zero inflow-outflow problem which is more adequate than the no-slip or slip cases for most physical and engineering applications.}

Similarly as in \cite{NoSCM},  the proof will be based on the reformulation of the original problem via the change of variables
\begin{equation}\label{change}
R:=f(\alpha)\vr,\; Z=g(\alpha)z
\end{equation}
as follows:
\begin{equation}\label{eq1.1}
\begin{aligned}
\partial_t \vr +\Div(\vr \vu)  & = 0,  \\
\partial_t z + \Div (z\vu) & = 0,  \\
\partial_t R +\Div(R \vu)  & = 0,  \\
\partial_t Z  + \Div(Z \vu) &= 0, \\
\partial_t \big((\vr+z)\vu\big) + \Div((\vr+z)\vu\otimes \vu) + \nabla P(R,Z) &=
{\rm div}\tn S(\Grad\vu)
\end{aligned}
\end{equation}
with boundary and initial conditions
{

\begin{equation}
\label{eq1.2}
\begin{array}{l}
\vc{u} |_{\partial \Omega} = \vu_B, \ \vr|_{\Gamma^{\rm in}} =  \vr_B,\
z|_{\Gamma^{\rm in}}=z_B,\\
\ R|_{\Gamma^{\rm in}} = R_B:=f(\alpha_B)\vr_B,\ Z|_{\Gamma^{\rm in}} = Z_B:=g(\alpha_B)z_B,
\end{array}
\end{equation}

}

\begin{equation}\label{eq1.3}
\vr(0,x)=\vr_0(x),\;
z(0,x)=z_0(x),
\end{equation}
$$
R(0,x) =R_0(x):= f(\alpha_0)\vr_0(x), \;
Z(0,x)=Z_0(x):=g(\alpha_0)z_0(x),
$$
$$
(\vr+z)\vu(0,x)=(\vr_0+z_0)\vu_0
$$
for unknown quintet $(\vr,z,R,Z,\vu)$ of functions defined on the space-time cylinder $Q_T=I\times\Omega$. In this paper, we will call it an {\it academic bi-fluid system}.

 It is to be noticed that, in the family of equations (\ref{eq1.1}), the transport equation for $\alpha$ is tacitly hidden in the first four
continuity equations: We anticipate here the fact that (formally), $\alpha=f^{-1}(R/\vr)$  and $\tilde\alpha=g^{-1}(Z/z)$ verify transport equation with the same initial condition $\alpha_0$ and the same boundary condition $\alpha_B$. To pass from the academic system (\ref{eq1.1}
--\ref{eq1.3}) to the original system (\ref{e1}--\ref{e6a}),  we shall need to identify $\alpha$ and $\tilde\alpha$. This will be done through the observation that the {\em pure transport equation} 
enjoys the {\em almost uniqueness property} (without condition ${\rm div}\vu\in L^1(0,T;L^\infty(\Omega))$ and without a slightly weaker
condition of Bianchini, Bonnicatto \cite{BIBO}), see Corollary \ref{teu}. These  notions and { further  properties
related to families of transport equations} which in a sense generalize
and complete the results of the seminal paper by DiPerna-Lions \cite {DL} will  be specified and put on rigorous grounds in Section \ref{TE} which is of independent interest.

The statement about the existence of weak solutions for the academic system (\ref{eq1.1}--\ref{eq1.3}) is formulated { in Section \ref{se2}} in Theorem \ref{theorem1}
and similar statement about the existence of weak solutions to the one velocity Baer-Nunziato type system (\ref{e1}--\ref{e5}) is available in Theorem \ref{theorem2}.

This article is inspired by two main sources: 1) The paper \cite{NoSCM}, where the author investigates weak solutions for the bi-fluid systems (both "academic"
and "realistic" under the no-slip boundary conditions). 2) Papers \cite{ChJNo}, \cite{ChNoYa}, \cite{KwNo} and monograph \cite{PloSo}, where the authors construct weak solutions for
the "mono-fluid" compressible Navier-Stokes equations with general inflow-outflow boundary data. We wish to concentrate on the effects of the non-homogenous boundary conditions. Therefore, in contrast with \cite{NoSCM}, we renounce at accommodation of the most general pressure law, in order to avoid the
unessential technical difficulties. Still, the pressure law considered in this paper covers most of classical situations including the mixture of two isentropic gases, cf. example (\ref{Pg}).

The main steps in  our approach are the following:

\begin{enumerate}
 \item { In Section \ref{TE} we develop the theory of renormalized
 solutions to families of transport equations with non-homogenous boundary data, which is one of the building blocks of the proofs. {\it These results are new and of independent interest.} It includes in this context:
\begin{enumerate}
\item Passage (via renormalization) from the (two) continuity equations to a pure transport equation.
\item Passage (via renormalization) from the (two) transport equations and a continuity equation to a continuity equation
\end{enumerate}
and two consequences of these results
\begin{enumerate}
\item {\em Almost compactness of the ratio of two solutions of continuity equations.}
\item {\em Almost uniqueness of the solutions to the pure transport equation.}
\end{enumerate}
As mentioned already above, these are delicate issues which are somehow connected (and as far as the {\it uniqueness} is concerned, somehow generalize) the seminal works
of DiPerna-Lions \cite{DL}, Ambrosio, Crippa, \cite{Ambr},   Bianchini, Bonicatto \cite{BIBO}, employing only the Eulerian approach. {{\it Almost compactness} generalizes to the uniform in time convergence and to the non homogenous boundary data the original results of Vasseur, Wen, Yu from their seminal paper
\cite{VWY}.}
We refer also to papers by Boyer \cite{Boyer} and Crippa et al \cite{CDL}   for results related to transport equations with non homogenous boundary data.
}

\item { We approximate the "academic" system similarly as in \cite{ChJNo}:  The momentum equation with added artificial pressure term (small parameter $\delta>0$)}
is approximated by the Galerkin approximation (of dimension $n$) while each continuity
equation is approximated by a specific parabolic boundary value problem (small parameter $\ep$) with the Robin type boundary conditions (for a while,
we shall call the solutions of these parabolic problems  "densities"). Having in view
applications in numerical analysis, in contrast with \cite{ChJNo}--where one solves the parabolic problems by using the maximal parabolic 
regularity and thus needs at least $C^2$ boundary-- we solve the parabolic problem on Lipschitz domains following Crippa et al. \cite[Lemma 3.2]{Crippa}.

It is known since \cite{3MNPZ} that the property of domination of one density by another one perpetuates for all times if it is in force initially. We shall
prove by using the maximum principle that the parabolic equations in consideration enjoy this property. Following \cite{NoSCM} and \cite{KwNo}, we easily derive for the approximate system the energy inequality and uniform bounds. 

The approximation on Lipschitz domains has an evident practical advantage: In \cite{ChNoYa}, the passage from $C^2$-domain required in \cite{ChJNo} to a Lipschitz piecewise $C^2$ domains
(which can be considered as an reasonable geometry for the inflow-outflow problems) is effectuated by a laborious approximation of domains, while
in the present approach, this step is for free (and requires slightly less of the domain than it is required in \cite[Theorem 2.4]{ChNoYa}) 
-- only at cost of more work at the level of the Galerkin approximation. This part of the proof is treated in Sections \ref{APPR}--\ref{Sn}, see Lemma \ref{EL1} and Proposition \ref{EP1}.

\item { In contrast to \cite{ChJNo}, we derive the renormalized equation for the parabolic problem for densities at the level of its weak formulation, see Section \ref{SE4.3}. In particular, we do not need its satisfaction almost everywhere in the time cylinder. This allows to simplify the approximation of the momentum equation:
compared to \cite{ChJNo} or \cite{NoSCM}, there is no need to consider the $\ep$-dependent power law dissipation.  }
\item For the remaining limit passages { $\ep\to 0$} and then $\delta\to 0$, we need to improve the estimates of pressure.  With the domination principle at hand, this is done via the Bogovskii operator exactly in the same way as for the simple mono-fluid case, cf. \cite{FNP} completed with \cite{AN-MP} or \cite{NoSCM}. { However, due to the non homogenous data, these estimates are available
only on compact subsets of $\Omega$.}
\item The main difficulty in both passages { $\ep\to 0$} and then $\delta\to 0$ is to pass to the limit in the non-linear pressure term $P(R,Z)$. It was observed for the first time by Vasseur et al. \cite{VWY} (and later improved in \cite[Proposition 7]{AN-MP}) that the quantity $Z/R$ is {\em "almost compact"} provided
$R$ dominates $Z$ and both quantities $R$ and $Z$ satisfy continuity equation with the same transporting velocity. The {\em almost compactness}  of $Z/R$
 is the property of the continuity equations and it is completely decoupled from the remaining equations in the system.
With this observation at hand, it is enough to prove the compactness for the quantity $\Pi(R_\ep,t,x):=P(R_\ep,R_\ep\theta)$  (resp. $\Pi(R_\delta,t,x):=
P(R_\delta,R_\delta\theta)$, where $\theta(t,x)$ is a given function (a ratio of weak limits of
the sequences $Z_\ep$ and $R_\ep$ resp. of the weak limits of sequences $Z_\delta$ and $R_\delta$). The task thus reduces practically to the task to prove the strong convergence of density in the "mono-fluid" case (with the pressure dependent
on only ``dominating'' density and time-space $(t,x)$). This process is nowadays well understood, cf. Lions \cite{L4}, Feireisl et al \cite{FNP}.\footnote{Since 2018, there exists an alternative approach to \cite{L4} due to Bresch, Jabin \cite{BrJab}, which is however not exploited in the present paper.} It passes through:
\begin{enumerate}
\item Derivation of the effective viscous flux identity.
\item Eliminating oscillations in the sequence of densities by using the theory of renormalized solutions due to DiPerna-Lions \cite{DL} which must be modified to accommodate the non
homogenous boundary conditions and renormalizing functions of several variables.
\end{enumerate}
The first point is very similar to the "mono-fluid" case. It is briefly explained in Section \ref{Sep} for $\ep\to 0$ and 
in Section \ref{Sdel} for $\delta\to 0.$ The second point is more delicate since the theory of renormalized solutions to the transport equation
is not available for the problems with the non-homogenous boundary conditions. Some elements of it are developed in Plotnikov, Sokolowski \cite{PloSo} and in \cite{ChJNo} but this
is not enough for our purpose. We treat this part in Section \ref{Sep} ($\ep\to 0$) and Section \ref{Sdel} ($\delta\to 0$) referring abundantly to Section \ref{TE}.
\item In  Section \ref{Appx} we gather all necessary specific results from functional analysis  needed throughout the proofs. 
\end{enumerate}

Theorem \ref{theorem2} is  the first rigorous result on existence of weak solutions for
a {\em version of the Baer-Nunziato type bi-fluid model with non zero inflow-outflow boundary conditions.}

The Di-Perna, Lions  transport theory { in conjonction
with the absence of improved estimates of pressure up to the boundary} imposes limitations on adiabatic coefficients $\gamma^\pm$ in formula (\ref{gammaP}) - or an equivalent limitation
on growth conditions of $P$ (see the next Section): at least one of them has to be greater or equal than {$2$}. In view of the existing mono-fluid theory, existence of weak solutions could be possibly hoped to be achieved if the  adiabatic coefficients of constituents were greater than $3/2$. This remains however a very interesting open problem.

In what follows, the scalar-valued functions will be printed with the usual font, the vector-valued functions will be printed in bold, and the tensor-valued functions with a special font, i.e. $\vr$ stands for the density, $\vu$ for the velocity field and $\tn{S}$ for the stress tensor.
We use standard notation for the Lebesgue and Sobolev spaces equipped by the standard norms $\|\cdot\|_{L^p(\Omega)}$ and $\|\cdot\|_{W^{k,p}(\Omega)}$, respectively. We will sometimes distinguish the scalar-, the vector- and the tensor-valued functions in the notation, i.e. we use $L^p(\Omega)$ for scalar quantities, $L^p(\Omega;\R^3)$ for vectors and $L^p(\Omega;\R^{3\times 3})$ for tensors. The indication of the $R$ or tensor character of the fields (here $;\R^3$ or $;\R^{3\times3}$) may be omitted, when there is no lack of confusion. The Bochner spaces of integrable functions on $I$ with values in a Banach space $X$ will be denoted $L^p(I;X)$;
likewise the spaces of continuous functions on $\overline I$ with values in $X$ will be denoted $C(\overline I;X)$. The norms in the Bochner spaces will be denoted $\|\cdot\|_{L^p(I;X)}$ and $\|\cdot\|_{C(\overline I;X)}$, respectively. In most cases, the Banach space $X$ will be either the Lebesgue or the Sobolev space.
Finally, we use vector spaces $C_{\rm weak}(\overline I; X)$  which is a subspace of $L^{\infty} (I; X)$  of continuous  functions in $\overline I$  with respect to weak topology of $X$ (meaning
that $f\in C_{\rm weak}(\overline I; X)$ iff $t\mapsto {\cal F}(f(t))$ belongs for any ${\cal F}\in X^*$ to $C( \overline I)$).

 The generic constants will be denoted by $c$, $\underline c$, $\overline c$, $C$,
$\underline C$, $\overline C$ and their value may change even in the same formula or in the same line. Sometimes, for two quantities $a$, $b$, we shall write
$$
a\aleq b\;\mbox{if $a\le c b$},\; c>0\;\mbox{a constant}, \; a\approx b\;\mbox{if}\, a\aleq b\;\mbox{and}\, b\aleq a.
$$
Here "constant" typically means a generic quantity independent on the approximating parameters of the problem (as number of Galerkin modes $n$,
artificial diffusion parameter $\ep$ or artificial pressure parameter $\delta$).

\section{Main results}\label{se2}

\subsection{Definition of weak solutions}

We first explain the notion of the weak solution to  problem (\ref{eq1.1}--\ref{eq1.3}) and to problem (\ref{e1}--\ref{e6}). Before starting the definition, we must underline, that, without  loss of generality, the boundary data (\ref{eq1.2}) (resp. (\ref{e6}){ )} are considered as a restriction to the boundary of functions defined  on the whole $\R^3$ (their regularity, as well as the regularity of the initial data will be specified later, in Section \ref{SEass}).

\begin{df}\label{d1}
A quintet $(\vr,z,R,Z,\vv=\vu-\vu_B)$ is a bounded energy weak solution to problem
(\ref{eq1.1}--\ref{eq1.3}), if the following holds:
\begin{enumerate}
\item The quintet belongs to the  functional spaces
$\vr,z,R, Z\geq 0$ a.e. in $I\times \Omega$, $(\vr,z,R,Z) \in  C_{\rm weak}(\overline I;$ $L^\gamma(\Omega))\cap L^{\gamma}(I; L^\gamma(
\partial\Omega;|\vu_B\cdot\vc n|{\rm d }S_x))$ 
with some $\gamma>1$,
$\vv \in L^2(I;W^{1,2}_0(\Omega;\R^3))$, $(\vr+z)|\vv|^2 \in L^\infty(I;L^1(\Omega))$, $P(R,Z) \in L^1(I\times \Omega)$, $(\vr+z)\vu\in C_{\rm weak}
(\overline I;L^q(\Omega))$ with some $q>1$.
\item Continuity equations
%
{
\begin{equation} \label{eq2.7}
\begin{array}{l}
\intO{r(\tau,\cdot)\varphi(\tau,\cdot)} - \intO{r_0(\cdot)\varphi(0,\cdot)} 
+  \int_0^\tau\int_{\Gamma^{\rm out}} r \vu_B \cdot \vc{n} \varphi
\ {\rm d}S_x{\rm d} t
\\ \\
=
 \int_0^\tau \int_\Omega \big(r \partial_t \varphi + r \vu \cdot \Grad \varphi\big) \dx \dt  
- \int_0^\tau\int_{\Gamma^{\rm in}} r_B \vu_B \cdot \vc{n} \varphi
\ {\rm d}S_x{\rm d} t 
\end{array}
\end{equation}
}
are satisfied
{ for any $\tau \in [0,T] $ and with any $\varphi\in C^1_c([0,T] \times\overline\Omega)$}, { where $r$ stands for $\vr$, $z$, $R$, $Z$.} 
\item Momentum equation
\begin{equation} \label{eq2.8}
\begin{aligned}
\int_\Omega (\vr+z)\vu \cdot \vcg{\varphi}(\tau,\cdot) \dx
 - \int_\Omega (\vr_0+z_0)\vu_0 \cdot \vcg{\varphi}(0,\cdot) \dx
=\int_0^\tau \int_\Omega \Big((\vr+z)\vu \cdot \partial_t \vcg{\varphi} \\
+ (\vr+z) \vu\otimes \vu: \Grad \vcg{\varphi} + P(R,Z) \Div \vcg{\varphi}-\tn S(\Grad\vu):\Grad\vcg{\varphi} \Big) \dx \dt 
\end{aligned}
\end{equation}
holds with  any $\tau\in [0,T]$ and $\vcg{\varphi} \in C^1_c([0,T) \times \Omega;\R^3)$.
\item Finally, the energy inequality
{
\begin{equation} \label{eq2.9}
\begin{aligned}
&\int_\Omega \Big(\frac 12 (\vr+z)|\vv|^2 + H(R,Z)\Big)(\tau,\cdot) \dx  
+\int_0^\tau \int_\Omega \tn S(\Grad\vu):\Grad\vu \dx\dt\\
&+ 
\int_0^\tau\int_{\Gamma^{\rm out}}H(R,Z)\vu_B\cdot \vc{n}{\rm d}S_x{\rm d}t
\\
\leq & \int_\Omega \Big(\frac{1}{2}(\vr_0+z_0)\vv_0^2 + H(R_0,Z_0)\Big) \dx
-\int_0^\tau\int_{\Gamma^{\rm in}}H(R_B,Z_B)\vu_B\cdot \vc{n}{\rm d}S_x{\rm d}t 
\\
+& \int_0^\tau\int _\Omega \Big(-P(R,Z){\rm div}\vu_B-\rho\vu\cdot \nabla \vu_{B}\cdot\vv + \tn S(\Grad\vu):\Grad\vu_B\Big)\dx
\end{aligned}
\end{equation}
}
is satisfied for a.a. $\tau \in (0,T)$, where $\vv_0=\vu_0-\vu_B$ and
{
\begin{equation}\label{H}
H(R,Z)=R\int_1^R\frac{P(s, s\frac Z R)}{s^2}{\rm d}s,\;\mbox{if $R>0$},\; H(0,Z)=0.
\end{equation}
}
\end{enumerate}

\end{df}

\begin{df}\label{d2}
A quartet $(\alpha,\vr,z,\vu)$ is a bounded energy weak solution to problem
(\ref{e1}--\ref{e6}),  if the following holds:
\begin{enumerate}
\item
 $\vr, z\geq 0$ a.e. in $I\times \Omega$, $(\vr,z) \in C_{\rm weak}(\overline I;L^\gamma(\Omega))
\cap L^\gamma(I;L^\gamma(\partial\Omega;|\vu_B\cdot\vc n|{\rm d}S_x))$
with some $\gamma>1$, $\alpha\in L^\infty(Q_T)\cap C_{\rm weak}(\overline I;L^\gamma(\Omega)){\cap L^\infty(I\times\partial\Omega)}$, $0\le\alpha\le 1$,
${ \vv=\vu-\vu_{\infty}}
 \in L^2(I;W^{1,2}_0(\Omega;\R^3))$, $(\vr+z)|\vu|^2 \in L^\infty(I;L^1(\Omega))$, $P(f(\alpha)\vr,g(\alpha) z) \in L^1(I\times \Omega)$,
$(\vr+z)\vu\in   C_{\rm weak}(\overline I;L^q(\Omega))$ with some $q>1$.
\item Continuity equations
{
\begin{equation} \label{eq2.7-}
\begin{array}{l}
\intO{r(\tau,\cdot)\varphi(\tau,\cdot)} - \intO{r_0(\cdot)\varphi(0,\cdot)} 
+  \int_0^\tau\int_{\Gamma^{\rm out}} r \vu_B \cdot \vc{n} \varphi
\ {\rm d}S_x{\rm d} t
\\ 
= 
\int_0^\tau \int_\Omega \big(r \partial_t \varphi + r \vu \cdot \Grad \varphi\big) \dx \dt   
- \int_0^\tau\int_{\Gamma^{\rm in}} r_B \vu_B \cdot \vc{n} \varphi
 {\rm d}S_x{\rm d} t 
\end{array}
\end{equation}
}
are satisfied for all $\tau\in [0,T]$ 
and with any { $\varphi\in C^1_c([0,T] \times\overline\Omega)$}, where $r$ stands for $\vr$, $z$.
\item Pure transport equation
\begin{equation}\label{eq2.7+}
\begin{array}{l}
\intO{\alpha(\tau,\cdot)\varphi(\tau,\cdot)} - \intO{\alpha_0(\cdot)\varphi(0,\cdot)} { + \int_0^\tau\int_{\Gamma{\rm out}} 
\alpha \vu_B \cdot \vc{n} \varphi
 {\rm d}S_x{\rm d} t} 
\\
=
\int_0^{ \tau} \int_\Omega \big(\alpha \partial_t \varphi + \alpha \vu \cdot \Grad \varphi{ +}\varphi\alpha{\rm div}\vu\big) \dx \dt  -\int_0^\tau\int_{\Gamma^{\rm in}} 
\alpha_B \vu_B \cdot \vc{n} \varphi
 {\rm d}S_x{\rm d} t 
\end{array}
\end{equation}
holds for all $\tau\in [0,T]$ with { any $\varphi\in C^1_c([0,T] \times\overline\Omega)$.} 
\item Momentum equation
\begin{equation} \label{eq2.8-}
\begin{aligned}
\int_\Omega (\vr+z)\vu \cdot \vcg{\varphi}(\tau,\cdot) \dx
- \int_\Omega (\vr_0+z_0)\vu_0 \cdot \vcg{\varphi}(0,\cdot) \dx
=\int_0^\tau \int_\Omega \big((\vr+z)\vu \cdot \partial_t \vcg{\varphi} \\+ (\vr+z) \vu\otimes \vu: \Grad \vcg{\varphi} +
P(f(\alpha)\vr,g(\alpha)z) \Div \vcg{\varphi}-\tn S(\Grad\vu):\Grad\vcg{\varphi}\big) \dx \dt 
\end{aligned}
\end{equation}
holds for all $\tau\in [0,T]$ with  any $\vcg{\varphi} \in C^1_c([0,T) \times \Omega;\R^3)$.
\item The energy inequality holds
\begin{equation} \label{eq2.9-}
\begin{aligned}
\int_\Omega \Big(\frac 12 (\vr+z)|\vu|^2 + H(f(\alpha)\vr, g(\alpha)z)\Big)(\tau,\cdot) \dx  
+ \int_0^\tau \int_\Omega \tn S(\Grad\vu):\Grad\vu \dx\dt\\ +
\int_0^\tau\int_{\Gamma^{\rm out}}H(f(\alpha)\vr, g(\alpha)z)\vu_B\cdot \vc{n}{\rm d}S_x{\rm d}t
\leq  \int_\Omega \Big(\frac{1}{2}(\vr_0 +z_0)\vu_0^2 + H(f(\alpha_0)\vr_0,g(\alpha_0)z_0)\Big) \dx\\
-\int_0^\tau\int_{\Gamma _{\rm in}}H(f(\alpha_B)\vr_B,g(\alpha_B)z_B)\vu_B\cdot \vc{n}{\rm d}S_x{\rm d}t 
\\
+\int_0^\tau\int _\Omega \Big(-P(f(\alpha)\vr,g(\alpha)z){\rm div}\vu_B-\rho\vu\cdot \nabla \vu_{B}\cdot\vv + \tn S(\Grad\vu):\Grad\vu_B\Big)\dx
{\rm d}t
\end{aligned}
\end{equation}
for a.a. $\tau \in (0,T)$, where $H$ is the same as in (\ref{eq2.9}).
\end{enumerate}
\end{df}

\subsection{Assumptions}\label{SEass}

Motivated by \cite[Section 2]{AN-MP} and \cite[Section 2.2]{NoSCM} we shall gather the hypotheses for Theorems
\ref{theorem1} and \ref{theorem2}. In this paper, we however concentrate to the phenomenons due to the
the effects of the non-zero inflow-outflow, and we do not insist on the most general hypotheses concerning constitutive law
for pressure, which would introduce to the problem { further} unessential technical difficulties. { We refer the reader to the Remark \ref{RemH}
for the possible relaxation of this part of hypotheses.}

To start, we define an {\em admissible inflow-outflow boundary { related}  to ${\vu}_B$} as follows: 
{
\begin{df}\label{admb}
We say that $\Omega$ is a domain with {\em admissible inflow-outflow boundary {related } to ${\vu}_B$} iff:
\begin{enumerate}
\item $\Omega$ is a bounded Lipschitz domain. 
\item
\begin{equation}\label{Gin0}
\Gamma^{\mathfrak{a}}=\emptyset\ \mbox{or}\
\Gamma^{\mathfrak{a}}=\cup_{k_{\mathfrak{a}}=1}^{\overline{ 
k_{\overline{\mathfrak{a}}}}}\Gamma^{\mathfrak{a}}_{k_{\mathfrak{a}}},
\ \mbox{ $\mathfrak{a}$ stands for ``in'', ``out''}
\end{equation}
and $\Gamma^{\rm in}_{k_{\rm in}}$, $\Gamma^{\rm out}_{k_{\rm out}}$ are
(open) parametrized $C^2$-surfaces or $C^2$ compact manifolds, cf. Section \ref{DG}.\footnote{Condition that each ``$\Gamma=\Gamma^{\rm in/out}_{\rm \cdot}$ is a $C^2$-surface or a $C^2$ compact manifold'' can be relaxed. It is a sufficient contition to guarantee existence of a projection $P$ to $\Gamma$ on a neighborghood $U$ of $\Gamma$, which is continuous on $U$.
  Indeed, this is the only condition from conclusion of Lemma \ref{LDG} which is used in the proof. Also,
conditions on $\mathfrak{g}^{\rm bd}$  could be relaxed. Indeed, the only thing 
we need in the proofs is that $\partial\Gamma^{\rm in}\cup\partial\Gamma^{\rm out}$ satisfies (\ref{Tubes}). These are the least conditions needed in the proofs and they enter into the game only through the Proposition \ref{LP2}.}

\item 
\begin{equation}\label{gbd}
 \partial\Gamma^{\rm in}\cup\partial\Gamma^{\rm out}=\mathfrak{g}^{\rm bd},
\end{equation}
where
$$
{\mathfrak{g}}^{\rm bd}=\Big(\cup_{k_{\rm bd}}^{\overline{k_{\rm bd}}}{{\mathfrak{g}}^{\rm bd}_{k_{\rm bd}}}\Big) \cup
\Big(\cup_{j_{\rm bd}}^{\overline{j_{\rm bd}}}{{\mathfrak{p}}^{\rm bd}_{j_{\rm bd}}}\Big)
$$
with ${\mathfrak{g}}^{\rm bd}_{k_{\rm bd}}$ being  bounded parametrized  (open) $C^1$-curves and ${\mathfrak{p}}^{\rm bd}_{j_{\rm bd}}$
points in $\R^3$.
\item The sets $\Gamma^{\rm in}_{k_{\rm in}}$, $\Gamma^{\rm out}_{k_{\rm out}}$,
${\mathfrak{g}}^{\rm bd}_{k_{\rm bd}}$,  ${\mathfrak{p}}^{\rm bd}_{j_{\rm bd}}$ have two by two empty intersection.
\end{enumerate}
\end{df}
}

We are now at the point to summarize the hypotheses for Theorem \ref{theorem1}.

\begin{enumerate}
\item{\it Boundary and initial conditions:}\footnote{The strict inequalities $r_B>0$ and $R_0>0$ are here for the sake of simplicity. They could be { relaxed} up to
$r_B\ge 0$ and $R_0\ge 0$.}
\begin{equation}\label{ruB}
0<r_B\in C_c(\R^3),\;{ \vu_B\in C_c^1(\R^3)},
\; r\;\mbox{stands for $\vr$, $z$, $R$, $Z$}.
\end{equation}
\begin{equation} \label{eq2.6}
0<R_0 \in L^\gamma(\Omega),\,{ \gamma\ge 2}, \; Z_0 \in L^\beta(\Omega) \; \text{ if } \beta > \gamma,\; (\vr_0+z_0)|\vu_0|^2\in L^1(\Omega).
\end{equation}
\begin{equation} \label{eq2.1}
(R_0, Z_0)(x)\in \overline{\cal O}, \; \underline F R_0(x)\le \vr_0(x)\le\overline F R_0(x),\; \underline G Z_0(x)\le z_0(x)\le\overline  G Z_0(x),
\end{equation}
$$
(R_B, Z_B)(x)\in \overline{\cal O}, \underline F R_B(x)\le \vr_B(x)\le\overline F R_B(x),\; \underline G Z_0(x)\le z_B(x)\le\overline  G Z_B(x),
$$
In the above $0<\underline F<\overline F$, $0<\underline G<\overline G$  and
 \begin{equation}\label{calO}
{\cal O}:=
(R,Z)\in R^2\,|\, \underline a R<Z<\overline a R\}\;\mbox{with some $0\le\underline a<\overline a$}.
\end{equation}

\item{\it Domain}
\begin{equation}\label{Om}
\mbox{$\Omega$ is a bounded domain with admissible inflow-outflow boundary { related} to ${{\vu}}_B$}.
\end{equation}

\item{\it Regularity and growth of the pressure function $P$:}
\begin{equation}\label{regP}
P\in C^1(\overline{\cal O})\cap C^2({\cal O}),\; P(0,0)=0.
\end{equation}
{
\begin{equation}\label{grP}
R^\gamma+ Z^{\beta}-1\aleq P(R,Z)\aleq R^\gamma+Z^\beta + 1
\;\mbox{in ${\cal O}$},
\end{equation}
\begin{equation}\label{dzP}
0\le\partial_Z P(R,Z)\aleq R^{\underline\gamma-1} + R^{\overline\gamma-1}\;\mbox{in ${\cal O}$ with some}\;\underline\gamma\in (0,1],\;
1\le\overline\gamma<\gamma +\gamma_{\rm Bog}
\end{equation}
where
$$
{ \gamma\ge 2},\;\beta >0,\;\gamma_{\rm Bog}=\min\{\frac 23\gamma-1,\frac \gamma 2\},
$$
\begin{equation}\label{drP}
 R^{\gamma-1}\aleq \partial_R P(R,Z) 
 \;\mbox{in ${\cal O}$}.
\end{equation}
Finally, 
\begin{equation}\label{convH}
H \;\mbox{is convex on}\; {\cal O}.
\end{equation}
}
\end{enumerate}

At this stage a few remarks impose.

{
\begin{rmk}\label{RemH}
\begin{enumerate}
\item The pressure function introduced in (\ref{Pg})
with $1\le\gamma^-\le \gamma^+$, { $\gamma^+\ge 2$} is an example of a constitutive law which satisfies all conditions (\ref{regP}--\ref{convH}).  
\item We introduce 
\begin{equation}\label{Ps}
\forall s\in  L^\infty(Q_T),\,\mbox{ such that}\ s\in [\underline a,\overline a],\;\Pi(R, t,x)=P(R,Rs(t,x)).
\end{equation} 
Due to (\ref{dzP}--\ref{drP}), there exists $d>0$ such that
\begin{equation}\label{Ps1}
 \mbox{for a.a. $(t,x)\in Q_T$, $\Pi(R,t,x)=dR^\gamma+\pi(R,t,x)$}
\end{equation}
where  for a.a. $(t,x)\in Q_T$, $R\mapsto\pi(R,t,x)$ is an non-decreasing function on $(0,\infty)$. This observation is an important point in the
proof of the strong convergence of the dominating density sequence.
\item The function $H$ defined in (\ref{H}) is called Helmholtz function. We easy verify that
\begin{equation}\label{regH}
H\in C^1({\cal O})\cap { C(\overline{\cal O})},\;H(1,Z)=0,\,\forall Z\in [\underline a,\overline a],\; H(R,Z)\ge \underline H>-\infty\;\mbox{for all $(R,Z)\in \overline{\cal O}$}.
\end{equation} 
provided
$P\in C^2({\cal O})\cap { C^1(\overline{\cal O})}$, $P(0,0)=0$  and it is a solution
of the first order partial differential equation
\begin{equation}\label{H*}
R\partial_R H(R,Z)+Z\partial_Z H(R,Z)-H(R,Z)=P(R,Z).
\end{equation}
\item Due to formula (\ref{H}), the function $H$ inherits the growth conditions of $P$, in particular, 
\begin{equation}\label{grH}
R^\gamma+Z^\beta-1\aleq H(R,Z)\aleq R^\gamma + Z^\beta+1 \;\mbox{for all $(R,Z)\in {\cal O}$}.
\end{equation}
\item
The conditions (\ref{regP}--\ref{convH}) are not the most general ones to guarantee the { existence} of weak solutions, cf. \cite{NoSCM}. Nevertheless, they provide a reasonable compromise between a presentable proof and overhelming technical complexity. Indeed: 

Condition (\ref{dzP}) could be replaced by a weaker one,
\begin{equation}\label{dzP+}
-1\aleq\partial_Z P(R,Z)\aleq R^{\underline\gamma-1} + R^{\overline\gamma-1}\;\mbox{in ${\cal O}$}
\end{equation}
with the range of $\underline\gamma$ and $\overline\gamma$
as in (\ref{dzP}).

Condition (\ref{drP}) could be replaced by a weaker one,
\begin{equation}\label{drP+}
 R^{\gamma-1}-1\aleq \partial_R P(R,Z) 
 \;\mbox{in ${\cal O}$ }.
\end{equation}

{ The most restrictive hypotheses (\ref{convH}) could be replaced by
\begin{equation}\label{convH+}
\forall (R,Z)\in {\cal O},\;|H(R,Z)-H_0(R,Z)|\aleq R^{\tilde\gamma}+1,\ 0<\tilde\gamma<\gamma,
\end{equation}
with some $H_0\in C^1({\cal O})\cap C(\overline{\cal O})$ convex on ${\cal O}$.}
\end{enumerate}
\end{rmk}

\subsection{Main results}\label{SEMR}

The first main result of the paper deals with the academic system (\ref{eq1.1}--\ref{eq1.3}) and  reads
\begin{thm} \label{theorem1}
 Under  Hypotheses (\ref{ruB}--\ref{convH}), problem (\ref{eq1.1}--\ref{eq1.3}) admits at least one { bounded energy} weak solution in the sense of Definition \ref{d1}.
Moreover, for all $t\in\overline I$, $(R(t,x),Z(t,x))\in \overline{\cal O}$, $\underline F R(t,x)\le \vr(t,x)\le\overline F R(t,x)$ and $\underline G Z(t,x)\le z(t,x)\le\overline G Z(t,x)$  for a.a. $x\in \Omega$, and further { for a.a. $(t,x)\in I\times\partial\Omega$, $(R(t,x),Z(t,x))\in \overline{\cal O}$, $\underline F R(t,x)\le \vr(t,x)\le\overline F R(t,x)$ and
$\underline G Z(t,x)\le z(t,x)\le\overline G Z(t,x)$.  
}
Finally,
$\vr,z,R,Z \in C(\overline I;L^1(\Omega))$,
$(\vr+z)\vu
\in C_{weak}([0,T);$ $L^{\frac{2\gamma}{\gamma+1}} (\Omega;\R^3))$  and $P(R,Z) \in L^q(I\times\Omega)$ for some $q>1$ and $Z \in  C_{weak}([0,T); L^\beta (\Omega)){ \cap L^\beta(I,L^\beta(\partial\Omega;|\vu_B\cdot \vc n|{\rm d}S_x))} $ if $\beta>\gamma$.
%
\end{thm}

The second main result of the paper deals with the one velocity Baer-Nunziato type system (\ref{e1}--\ref{e6}) and  reads:
\begin{thm} \label{theorem2}
Suppose that $f,g\in { C^1(0,1)}$
{ are} two
strictly monotone and { strictly positive} functions on interval $(0,1)$ and that the boundary conditions $\vr_B$, $z_B$, $\vu_B$ satisfy conditions (\ref{ruB}). 
Let { $\gamma\ge 2$}, $\beta >0$ and, in addition,
$$
\alpha_B\in C(\overline\Omega),\;0<\underline\alpha\le\alpha_B\le\overline\alpha<1,\; (f(\alpha_B)\vr_B, g(\alpha_B)z_B)(x)\in \overline{\cal O},
$$
$$
\alpha_0\in L^\infty(\Omega), \;0<\underline\alpha\le\alpha_0\le\overline\alpha<1,\; (f(\alpha_0)\vr_0, g(\alpha_0)z_0)(x)\in \overline{\cal O}
$$
\begin{equation} \label{eq2.6!}
{ 0<\vr_0} \in L^\gamma(\Omega), \; z_0 \in L^\beta(\Omega) \; \text{ if } \beta > \gamma,
\;
(\vr_0+z_0)|\vu_0|^2\in L^1(\Omega).
\end{equation}
Suppose that the domain $\Omega$ is a bounded Lipschitz domain with the admissible inflow-outflow boundary with respect to $\vu_B$, cf. (\ref{Om}). 
Finally suppose that the pressure $P$ and its Helmholtz function $H$ verify hypotheses (\ref{regP}--\ref{convH}).
Then the problem (\ref{e1}--\ref{e6}) admits at least one { bounded energy} weak solution in the sense of Definition \ref{d2}.
Moreover, for all $t\in\overline I$, $(f(\alpha)\vr(t,x),g(\alpha) z(t,x))\in \overline{\cal O}$ and $\underline\alpha\le\alpha(t,x)\le\overline\alpha$ for a.a. $x\in\Omega$, and further for a.a. $(t,x)\in I\times\partial\Omega$,
$(f(\alpha)\vr(t,x),g(\alpha) z(t,x))\in \overline{\cal O}$ and $\underline\alpha\le\alpha(t,x)\le\overline\alpha$. Finally,
$\alpha,\vr,z\in  C(\overline I;L^1(\Omega))$,  
$z \in C_{weak}(\overline I;$ $ L^{{\beta}} (\Omega)){ \cap L^\beta(I,L^\beta(\partial\Omega;|\vu_B\cdot \vc n|{\rm d}S_x))} $ if $\beta>\gamma$, $(\vr+z)\vu
\in C_{weak}([0,T);$ $L^{\frac{2\gamma}{\gamma+1}} (\Omega;\R^3))$
 and $P(f(\alpha)\vr,g(\alpha)z) \in L^q(I\times\Omega)$ with some $q>1$.
\end{thm}

{
\begin{rmk}\label{rth1+2}
\begin{enumerate}
\item
 As already mentioned in Remark \ref{RemH}, Theorems \ref{theorem1}, \ref{theorem2} remain still valid if we replace
the hypotheses (\ref{dzP}--\ref{convH}) for the constitutive law for pressure by weaker hypotheses (\ref{dzP+}--\ref{convH+}).
\item  The  growth corresponding to $\beta\in (0,1]$ is covered by Theorems \ref{theorem1}--\ref{theorem2} only provided $0<\underline a$ and 
provided one considers the version of the theorems with weaker hypotheses (\ref{dzP+}--\ref{convH+}). Indeed, $\beta\in (0,1]$ is in contradiction with (\ref{convH}) whatever is the value of $\underline a$ and it is in contradiction with  (\ref{dzP+}) if $\underline a =0$.
\end{enumerate}
\end{rmk}
}

Theorem \ref{theorem1} will be proved through Sections \ref{APPR}--\ref{Sdel}. Theorem \ref{theorem2} is proved in Section \ref{Realistic}. { Without loss of generality, we shall concentrate on the case when $\Gamma^{\rm in}$ and $\Gamma^{\rm out}$ are not empty. For the sake of simplicity, we shall also discard the situation, when $\Gamma^{\rm in}$ or $\Gamma^{\rm out}$  contain a compact manifold. Finally, also for the sake of simplicity, we shall limit ourselves to the case $\beta\le\gamma$.}

\section{Transport equations with non homogenous boudary data}\label{TE}

The properties of solutions to the continuity equation and an interplay between the solutions of continuity and transport equations  play essential role
in the proofs in this paper. These results are well known in the case, when the transporting velocity $\vc u$ is zero at the boundary,
essentially due to seminal paper of Di-Perna, Lions \cite{DL}, see Vasseur et al. \cite{VWY}, and \cite[Section 3]{AN-MP}, \cite[Section 3]{NoSCM} for further extension, still with the condition $\vu=0$ at $\partial\Omega$. The purpose of this section is to extend them to non-homogenous boundary data. This section is therefore of independent interest.

\subsection{Some elements of differential geometry}\label{DG}

Let $k=1,2,\ldots$. A set $\Gamma\subset \R^3$ (resp. $\mathfrak{g}
\subset \R^3$) is a $C^k$-parametrized bounded surface (resp. $C^k$-parametrized bounded curve) 
iff there exists a bounded domain ${\cal O}\subset \R^2$ (resp. ${\cal O}\subset \R$) a  bijection ${\cal G}\in C^k(\overline{\cal O};\overline{\Gamma})$
(resp. ${\cal G}\in C^k(\overline{\cal O};\overline{ \mathfrak{g}})$) such that ${\cal G}$ is an $C^1$-differomorphism from ${\cal O}$ onto $\Gamma$ (resp. $\mathfrak{g}$).\footnote{ Meaning that ${\cal G}$  is a bijection of regularity $C^1({\cal O})$ and the differential $\Grad {\cal G} (x)$ is a bijection of $R^2$ to $R^2$ (resp.
$R$ to $R$) for any $x\in {\cal O}$.}
For a set $A\subset \R^3$ we denote $d_A(x)={\rm dist}(x;A):=\inf_{y\in A}|x-y|$ the distance function to $A$ and $\vc n=\frac{\partial_1G\times\partial_2G}{|{\partial_1G\times\partial_2G}|}$ the (outer) normal vector to the surface $\Gamma\subset\partial\Omega$ with respect to $\Omega$.\\

{ We define
$$
T(\Gamma;\ep):=\{x\in \R^3\,|\,x=x_\Gamma+s\vc n(x_\Gamma),\,x_\Gamma\in\Gamma,\,s\in { (-\ep,\ep)}\},
$$
\begin{equation}\label{Tset}
{ T^\pm(\Gamma;\ep)}:=\{x\in \R^3\,|\,x=x_\Gamma+s\vc n(x_\Gamma),\,x_\Gamma\in\Gamma,\,\pm s\in { (0,\ep)}\},
\end{equation}
$$
T(\mathfrak{g};\ep):= B(\mathfrak{g};\ep)\setminus \overline{B(\partial\mathfrak{g};\ep)},
$$
where, for a subset $A$ of $\R^3$,
$$
B(A;\ep):=\{x\in \R^3|{\rm dist}(A;x)<\ep\}.
$$
It is well known that $d_A$ is 1-Lipschitz function on $\R^3$ and if $A$ is closed, for almost all $x\notin A$, there exists a unique point $P_A(x)\in A$ nearest to $A$ such that 
\begin{equation}\label{dLipschitz}
\Grad d_A(x)= \frac{x-P_A(x)}{d_A(x)},
\end{equation}
cf. Ziemer \cite[Exercice 1.15]{Zi}. We also recall that if $k\ge1$,
\begin{equation}\label{Tubes}
|T(\Gamma;\ep)|\aleq \ep,\quad |T(\mathfrak{g};\ep)|\aleq \ep^2,\
|B(\partial\mathfrak{g};\ep)|\aleq \ep^3
\end{equation}
cf. Gray \cite{Gray}.

The following Lemma resumes the properties of distance and projection
to $\Gamma$ which can be deduced from Theorems 1,2 in Foote \cite{Foote}. 
\begin{lem}\label{LDG}
Let $k\ge 2$. Then there exists $\overline\ep>0$ such that for all $0<\ep\le\overline\ep$, we have the following:
\begin{enumerate}
\item $T_\ep:=T(\Gamma;\ep)$, { $T^\pm_\ep:=T^\pm(\Gamma;\ep)$} are open sets in $\R^3$.
\item $\forall x\in T_{\ep},\;\exists! x_{\Gamma}\in\Gamma,\;|x_{\Gamma}-x|=d_\Gamma(x)$. We denote $P_\Gamma(x):=x_{\Gamma}$. Then $P_\Gamma\in C^{k-1}(T_{\ep})$. 
\item $
d_\Gamma\in C^{k}( {T_{\ep}^+\cup\Gamma})\cup C^{k}(T_{\ep}^-\cup\Gamma)$ 
 and for all $x\in T^\pm_{\ep}$, $\Grad d_\Gamma(x)=\frac{x-P_\Gamma(x)}{|x-P_\Gamma(x)|}=\pm\vc n(P_\Gamma(x))$.
 \end{enumerate}
\end{lem}
}

\subsection{Renormalized solutions to families of transport equations}\label{DR}

We consider the general transport equations on the time-space cylinder $Q_T=I\times\Omega$, $\Omega$ a bounded Lipschitz domain in $\R^3$  
and $I=(0,T)$, $T>0$ a time interval. The equations read:
\begin{equation}\label{co1}
{\partial_t r+{\rm div}\,(r\vu) + rv=0\;\mbox{in $(0,T)\times\Omega$}}
\end{equation}
with initial and boundary conditions
$$
r(0,\cdot)=r_0(\cdot)\;\mbox{in $\Omega$},\; r|_{I\times\Gamma_{\rm in}}=r_B.
$$
Equation (\ref{co1}) is called {\em continuity equation} if $v=0$ and {\em pure transport equation} if $v=-\Div\vu$.

We shall consider the following regularity of transporting coefficients 
\begin{equation}\label{tvu}
\vu=\vv+\vu_B,\; \vv\in L^2(I; W_0^{1,2}(\Omega; \R^3)),\;v\in L^2(Q_T).
\end{equation}
A function  { $0\le r\in L^2(Q)\cap L^\infty(I,L^\gamma(\Omega))\cap
L^\gamma(I;L^\gamma(\partial\Omega;|\vu_B\cdot\vc n|{\rm d}S_x))$} is a   weak solution of the transport equation (\ref{co1}) with boundary data $r_B$ and initial data $r_0$ 
iff
\begin{equation}\label{cow}
\int_0^T\int_{\Gamma^{\rm in}}r_B\vu_B\cdot\vc n{\rm d}S{\rm d}t +
\int_0^T\int_{\Gamma^{\rm out}}r\vu_B\cdot\vc n{\rm d}S{\rm d}t
-\intO{r_0\varphi(0,\cdot)}
\end{equation}
$$
=\int_0^T\intO{\Big(r\partial_t\varphi
+r\vu\cdot\Grad\varphi
-r v\varphi\Big)}{\rm d}t
$$
holds
with any { $\varphi\in C^1_c([0,T)\times \overline\Omega)$.}

We shall need the following result on the renormalized solutions to the transport equation, which is of independent interest.
It is formulated in the specific functional setting needed for the purpose of this paper and its formulation could be easily generalized to the $L^p-L^q$ setting.

\begin{prop}{\rm [Renormalized solutions to families of transport equations]}\label{LP2} 
Let $\Omega\subset \R^3$ be a bounded Lipschitz domain and $\gamma>1$.
Let $\vu_B$ belong to
class (\ref{ruB}), let each component of ${\mathfrak{ r}}_B=(r_B^{(1)},\ldots,r_B^{(M)})$, $\mathfrak{ s}_B=(s_B^{(1)},\ldots,s_B^{(N)})$ be non negative function in $C(\overline\Omega)$  and let each component of ${\mathfrak{r}}_0=(r_0^{(1)},\ldots,r_0^{(M)})$, ${\mathfrak{ s}}_0=(s_0^{(1)},\ldots,s_0^{(N)})$ be non negative and belong
to  $L^\gamma(\Omega)$ {(for the components of $\mathfrak{r}$) 
resp. to $L^\infty(\Omega)$ (for the components of $\mathfrak{s}$)}. Suppose that $\Omega$ has an {\em admissible inflow-outflow boundary relative to ${\vu}_B$}
in the sense of Definition \ref{admb}.\footnote{ As we already mentioned in Definition \ref{admb}, condition that each ``$\Gamma=\Gamma^{\rm in/out}_{\rm \cdot}$ is a $C^2$-surface or a $C^2$ compact manifold'' can be relaxed. It is a sufficient contition to guarantee existence of a projection $P$ to $\Gamma$ on a neighborghood $U$ of $\Gamma$, which is continuous on $U$.   Indeed, this is the only condition from conclusion of Lemma \ref{LDG} which is used in the proof. Also,
conditions on $\mathfrak{g}^{\rm bd}$ in Definition \ref{admb} could be relaxed. The only thing 
we need in the proofs is that $\partial\Gamma^{\rm in}\cup\partial\Gamma^{\rm out}$ satisfies (\ref{Tubes}).}  Assume further that $\vu$ belongs to the class (\ref{tvu}) and that 
$$
0\le\mathfrak{r}=[(r^{(1)},\ldots,r^{(M)}]\in L^2(I;L^2(\Omega))\cap L^\infty(I,L^\gamma(\Omega))\cap { L^\gamma(I;L^\gamma(\Gamma^{\rm out};|\vu_B\cdot\vc n|{\rm d}S_x)),}
$$
$$
0\le\mathfrak{s}=[s^{(1)},\ldots,s^{(N)})\in L^\infty((0,T)\times\Omega)
{ \cap L^\infty(I;L^\infty({\Gamma^{\rm out}};|\vu_B\cdot\vc n|{\rm d}S_x))}
$$ 
are such that each component of $\mathfrak{r}$ is a weak solution of the continuity equation (\ref{co1}) (i.e., it satisfies (\ref{cow}) with
$\vu$ and $v=0$), while each component of $\mathfrak{s}$ is a weak solution the pure transport equation (\ref{co1}) (i.e., it satisfies (\ref{cow}) with
$\vu$ and $v=-\Div \vu$). In the above the sign "$\le$" means that each component of the vector is a non negative number.

Then there holds:
\begin{enumerate}
\item The quantities $\mathfrak{r}, B(\mathfrak{r})\in C_{\rm weak}(\overline I; L^\gamma(\Omega))\cap C(\overline I; L^p(\Omega))$, $1\le p<\gamma$  and we have
{
\begin{equation}\label{P3}
\intO{ B(\mathfrak{r}) \varphi}\Big|_0^\tau +  \int_0^\tau\int_{\Gamma^{\rm out}} B(\mathfrak{r}) 
{\vu}_B \cdot \vc{n} \varphi \ {\rm d}S_x {\rm d} t +\int_0^\tau\int_{\Gamma^{\rm in}} B(\mathfrak{r}_B) 
{\vu}_B \cdot \vc{n} \varphi \ {\rm d}S_x {\rm d} t 
\end{equation}
$$
=
\int_0^\tau\intO{ \Big({ B({\mathfrak{ r}})\partial_t\varphi}+B(\mathfrak{r}) {\vu} \cdot \Grad \varphi -\varphi\left( B'(\mathfrak{r}) \mathfrak{r} - B(\mathfrak{r}) \right) \Div{\vu} \Big)}  {\rm d} t,\ B(\mathfrak{r}(0))=B(\mathfrak{r}_0)  
$$
}
with any $\tau\in\overline I$, 
for any { $\varphi \in  C_c^1([0,T]\times  
\overline\Omega$)} and 
any
\begin{equation}\label{B}
B\in C([0,\infty)^M),\;\nabla_{\mathfrak{r}}B\in L^\infty((0,\infty)^M).
\end{equation}

\item
The quantities $\mathfrak{s}, B(\mathfrak{s})\in C([0,T]; L^p(\Omega))$, $1\le p<\infty$ and we have
{
\begin{equation}\label{P3s}
\intO{ B(\mathfrak{s}) \varphi}\Big|_0^\tau + \int_0^\tau\int_{\Gamma^{\rm out}} B(\mathfrak{s}){\vu}_B \cdot \vc{n} \varphi \ {\rm d}S_x {\rm d} t  +  \int_0^\tau\int_{\Gamma^{\rm in}} B(\mathfrak{s}_B) 
{\vu}_B \cdot \vc{n} \varphi \ {\rm d}S_x {\rm d} t 
\end{equation}
$$
=
\int_0^\tau\intO{ \Big({ B({\mathfrak{ s}})\partial_t\varphi}+B(\mathfrak{s}) {\vu} \cdot \Grad \varphi +\varphi
B(\mathfrak{s})  \Div {\vu} \Big)}  {\rm d} t,\ B(\mathfrak{s}(0))=B(\mathfrak{s}_0)  
$$
with any $\tau\in \overline I$, 
for any  $\varphi \in  C_c^1(\overline I\times\overline\Omega)$
}
and any $B$ in class
\begin{equation}\label{Bs}
B\in C([0,\infty)^N)\cap C^1((0,\infty)^N).
\end{equation}

\item Let now $r$ be one component of $\mathfrak{r}$. Then
$rB(\mathfrak{s})\in C_{\rm weak}([0,T];L^\gamma(\Omega))\cap C([0,T];L^p(\Omega))$, $1\le p<\gamma$ and we have:
{
\begin{equation}\label{P3rs}
\intO{ rB(\mathfrak{s}) \varphi}\Big|_0^\tau + \int_0^\tau\int_{\Gamma^{\rm out}} r B(\mathfrak{s}) 
{\vu}_B \cdot \vc{n} \varphi \ {\rm d}S_x {\rm d} t
+ \int_0^\tau\int_{\Gamma^{\rm in}} r_B B(\mathfrak{s}_B) 
{\vu}_B \cdot \vc{n} \varphi \ {\rm d}S_x {\rm d} t 
\end{equation}
$$
=
\int_0^\tau\intO{ \Big({ r  B({\mathfrak{ s}})\partial_t\varphi}+ rB(\mathfrak{s}) {\vu} \cdot \Grad \varphi\Big)}, \ \mathfrak{r}(0)
B(\mathfrak{s}(0))= \mathfrak{r}_0
B(\mathfrak{s}_0)
$$
with any $\tau\in \overline I$, 
for any {$\varphi \in  C_c^1(\overline I\times \overline\Omega))$,} and 
any $B$ in the class (\ref{Bs}).
}
\end{enumerate}
\end{prop}

\begin{rmk}\label{rinflow}
\begin{enumerate}
\item The condition (\ref{B}) on the renormalizing function $B$ in Lemma \ref{LP2} can be relaxed by using the
Lebesgue dominated convergence theorem: one can take, e.g., 
$B\in C([0,\infty)^M)$ $\cap$ $C^1((0,\infty)^M)$, $\mathfrak{r}\cdot\nabla_{{\mathfrak{r}}}B-B\in C[0,\infty)$ and 
$|B(\mathfrak{r})|\le c(1+r)^{q_1}$, $ |\mathfrak{r}\cdot\nabla_{{\mathfrak{r}}}B (\mathfrak{r})-B(\mathfrak{r})|\aleq c(1+|\mathfrak{r}|)^{q_2}$, $0\le q_1\le\frac 56\gamma$, $0\le q_2\le
\min\{1,\gamma/2\}$. 
In this case, { relation (\ref{P3}) remains} still valid and $B(\mathfrak{r})\in C_{\rm weak}(\overline I;L^{\gamma/q_1}(\Omega))$ $\cap$
 $C(\overline I;L^p(\Omega))$, $1\le p<\gamma/q_1$.
 \item Equation (\ref{P3})  implies
 $$
-\int_0^T\intO{B({\mathfrak{ r}})\eta\partial_t\varphi}{\rm d}t=<{\cal F},\varphi>,\;\forall \varphi \in C^1_c((0,T)\times\overline\Omega), \;\forall \eta \in C^1_c(\Omega).
$$
where 
{
$$
<{\cal F},\varphi>=
\intO{B(\mathfrak{r})\eta {\vu} \cdot \Grad \varphi
+B(\mathfrak{r})\varphi{\vu} \cdot \Grad \eta-\varphi\left( B'(\mathfrak{r}) \mathfrak{r} - B(\mathfrak{r}) \right) \Div{\vu} \Big)}  {\rm d} t 
$$
$$
- \int_0^\tau\int_{\Gamma^{\rm out}} B(\mathfrak{r}) 
{\vu}_B \cdot \vc{n} \varphi \ {\rm d}S_x {\rm d} t
- \int_0^\tau\int_{\Gamma^{\rm in}} B(\mathfrak{r}_B) 
{\vu}_B \cdot \vc{n} \varphi \ {\rm d}S_x {\rm d} t. 
$$
}
We verify by using the H\"older inequality and the trace theorem that
$$
\partial_t[B(\mathfrak{r})\eta]\in L^2(I;[W^{1,p}(\Omega)]^*)\,\mbox{with 
$\frac 1p+\frac q\gamma=\frac 56$, $q=\max\{q_1,q_2\}$}.
$$
\end{enumerate}
\end{rmk}
\noindent
{\bf Proof of Proposition \ref{LP2}}\\ 
We shall limit ourselves to show Item 1. (i.e. (\ref{P3})) in the "simple situation"$M=1$. To this end we set $\mathfrak{r}=r^{(1)}$, $\mathfrak{r}_0=r^{(1)}_0$, $\mathfrak{r}_B=r^{(1)}_B$ and consider $B\in C([0,\infty))$, $B'\in L^\infty(0,\infty)$. The proof is divided into seven steps. 
Extension from $M=1$ to $M>1$ is briefly described in {\it Step 8}.
The general case and the proof of Items 2.-3. follow the same strategy. 

{ In {\it Step 1} we construct convenient
outer neighborhoods ${\cal U}^+$ of $\Gamma^0$ and of each component
$\Gamma\subset\Gamma^{\rm out/in}$. In particular, the projection operator$P$ (cf. Lemma \ref{LDG}) must be sufficiently regular on the outer { neighborhoods} of  components of $\Gamma^{\rm out/in}$. In {\it Step 2}, we construct an { outer} { neighborhood} ${\cal V}^+\subset{\cal U}^+$ of any component $\Gamma\subset\Gamma^{\rm in}$ to which the density field can be extended via the characteristics of the vector field $-\vu_B$ in such a
way that the couple $(\mbox{density},\vu_B)$ satisfies the continuity equation on $(0,T)\times{\cal V}^+$--see (\ref{P2dod-}). In {\it Step 3},
we do the same for any component $\Gamma\subset\Gamma^{\rm out}$-- see
(\ref{extout}), and, in {\it Step 4.}, for $\Gamma^0$. The global extension is then defined in {\it Step 5.}; it satisfies the continuity
equation in the sense of distributions on a domain $\tilde\Omega$ created as union of $\Omega$ with $\Gamma$'s and their outer neighboghoods. It provides enough space to apply the DiPerna-Lions regularization procedure
(cf. \cite{DL}) in
{\it Step 6.}.  The result of Steps 1-6, is the renormalized continuity equation satisfied with test functions up to $\overline\Omega$ but with compact support in $I$. The extension to its time integrated form is
discussed in {\it Step 7}.
\\ \\
{\it Step 1: Construction of a particular outer neighborhoods of $\Gamma^{\rm in}$, $\Gamma^{\rm out}$ and $\Gamma^0$--cf. (\ref{e6a}), (\ref{e6b}) and Definition \ref{admb}.}
We denote, 
\begin{equation}\label{ball+}
B^+(\xi,\ep)= B(\xi,\ep)\cap \R^3\setminus\overline\Omega,\; B^-(\xi,\ep)= B(\xi,\ep)\cap \Omega,
\end{equation} 
\begin{enumerate}
 \item Let $\Gamma=\Gamma^0$. Then for any $\xi\in\Gamma$ there exists
 $\epsilon(\xi)$ such that $B^+(\xi,\epsilon)\cap\partial\Omega\subset\Gamma$. We then set 
 \begin{equation}\label{calU++}
  {\cal U}^+(\Gamma)=\cup_{\xi\in \Gamma}B^+(\xi;\epsilon).
 \end{equation}
\item Let $\Gamma$ be  any parametrized surface in the decomposition (\ref{Gin0}). 
We deduce that
\begin{equation}\label{ballep}
\forall\xi\in\Gamma,\;\exists\epsilon=\epsilon(\xi)\in (0,\ep),\; B(\xi,\epsilon){ \subset { T}}(\Gamma;\ep),
\end{equation}
where $T(\Gamma;\ep)$ is defined in Lemma \ref{LDG}. { We define open sets}
\begin{equation}\label{calU+}
{\cal U}={\cal U}_\ep(\Gamma):= \cup_{\xi\in \Gamma} B(\xi;\epsilon(\xi)),\quad
{\cal U}^\pm={\cal U}_\ep^\pm(\Gamma):=\cup_{\xi\in \Gamma} B^\pm (\xi;\epsilon(\xi)).
\end{equation}
Taking into account Lemma \ref{LDG}, we deduce that for any
$\ep>0$ sufficiently small, we have, in particular,
$$
{\cal U}={\cal U}^+\cup {\cal U}^-\cup \Gamma,\ {\cal U}^+\subset \R^3\setminus\overline\Omega,\ {\cal U}^-\subset\Omega,
$$
and
\begin{equation}\label{pd}
P_\Gamma\in { C^1({ \overline{\cal U}})},\; d_\Gamma\in { C^2 ( \overline{{ \cal U}^+})\cap  C^2 ( \overline{{ \cal U}^-}),\  \forall x\in{\cal U}^\pm,\ \Grad d_\Gamma(x)=\pm\vc n(P(x))
}
\end{equation}
\item
We realize that construction (\ref{calU++}), (\ref{calU+})
can be done in such a way that
\begin{equation}\label{UtildeU}
{\cal U}^+(\Gamma)\cap{\cal U}^+(\tilde\Gamma)=\emptyset,\; \mbox{for any couple $\Gamma\neq\tilde\Gamma$ in the decomposition (\ref{Gin0}) }
\end{equation}
$$
\mbox{or for $\Gamma$ in the decomposition (\ref{Gin0}) and $\tilde\Gamma=\Gamma^0$}.
$$
\end{enumerate}

The goal now is to extend the density and velocity fields $(\mathfrak{r},\vu)$ from $\Omega$ to an outer neighborhood of $\Gamma$ (which is a subset of ${\cal U}^+$ and which we will denote for a moment ${\cal V}^+(\Gamma)$) in such a way that the extended fields will satisfy the continuity equation 
in the sense of distributions on the open set  $\Omega\cup {{\cal V}^+(\Gamma)}\cup \Gamma)$. The construction will depend on the fact whether $\Gamma\subset\Gamma^{\rm in}$ or $\Gamma\subset \Gamma^{\rm out}$
or $\Gamma=\Gamma^0$. 
\\ \\
{\it Step 2: Extension of the density beyond the inflow boundary:}
\begin{enumerate}
 \item {\it ${\cal V}^ +$ in the case $\Gamma=\Gamma^{\rm in}$.}
\begin{enumerate}
\item {\it Flux of $-\vu_B$.} We denote by $\mathfrak{X}$ the flux of the vector field $-\vu_B$, i.e. solution of the following family of Cauchy problems  for ODE,
\begin{equation}\label{EDO}
\frac{\rm d}{{\rm d}s}\mathfrak{X}(s;x)=-\vu_B(\mathfrak{X}),\;\mathfrak{X}(0;x)=x,\; s\in \R,\;x\in \R^3.
\end{equation}
It is well known, cf. e.g. \cite[Chapter XI]{Demail}, that,
$$
\mathfrak{X}\in C^1(\R\times \R^3),\;
\mathfrak{X}(t,\cdot)\;\mbox{is $C^1$ diffeomorphism of $\R^3$ onto $\R^3$},
$$
in particular,
$$
\forall (s,x)\in \R\times \R^3,\;\mathfrak{X}(-s,\mathfrak{X}(s;x))=x,\;{\rm det}\Big[\Grad\mathfrak{X}(s
;x)\Big]=
{\rm exp}\Big(-\int_0^s\Div\vu_B(z,\mathfrak{X}(z,x)){\rm d}z\Big)>0.
$$
\item {\it Construction of ${\cal V}$, ${\cal V}^\pm$}.
\begin{enumerate}
\item 
Let $K\subset \Gamma$ be a compact set (with respect to the trace topology of $\R^3$ on $\partial\Omega$) where
$\Gamma\subset\Gamma^{\rm in}$ is any component in decomposition (\ref{Gin0}).
We want to prove that  $\forall\ep>0$, $\exists\delta=\delta_K>0,$
\begin{equation}\label{IN}
\forall (s,\xi)\in (-\delta,\delta)\times K,
\;
\mathfrak{X}(s,\xi)\subset {\cal U}_\ep(\Gamma)\;\mbox{and}\;
\forall (s,\xi)\in (0,\delta)\times K,
\;
\mathfrak{X}(s,\xi)\subset {\cal U}^+_\ep(\Gamma).
\end{equation}
Indeed:
\item By the uniform continuity of $\mathfrak{X}$ on compacts of $\R\times\R^3$ we easily get
$$
\exists \delta>0,\;\mathfrak{X}((-\delta,\delta);K)\subset {\cal U}_\ep(\Gamma).
$$
\item Moreover,
 due to (\ref{EDO}), 
$$
\forall x_B\in \Gamma^{\rm in},\;\exists\delta>0,\;\forall s\in (0,\delta),\;
(\mathfrak{X}(s,x_B)-P(\mathfrak{X}(0,x_B)))\cdot\vc n(x_B)>0.
$$  
By the uniform  continuity of $\vc n$ on 
compacts of $\Gamma$, $P$ on { compacts of} ${ T}(\Gamma)$ and
$\mathfrak{X}$ on compacts of $\R\times \R^3$, we deduce, in particular, that 
$$
\exists \delta>0,\;\forall (s,\xi)\in (0,\delta)\times K,\; \Big(\mathfrak{X}(s,\xi)-P(\mathfrak{X}(s,\xi))\Big)\cdot\vc n(P(
\mathfrak{X}(s,\xi))>0.
$$
{\it
This means that $\mathfrak{X}((0,\delta);K)\subset {\cal U}^+_\ep(\Gamma)$ which 
finishes the proof of (\ref{IN})}
\item Recalling Definition \ref{admb} and the definition of
the parametrized surface in Section \ref{DG}, we know that $\overline\Gamma= {\cal G}(\overline{\cal O})$ where
${\cal G}\in C^2(\overline{\cal O};\R^3)$ is a $C^2$-diffeomorhism of a domain ${\cal O}\in \R^2$ onto $\Gamma$. Let $L_n$ be an exhaustive sequence of compacts of
${\cal O}$,
\begin{equation}\label{Ln}
L_n\subset {\rm int}_{2}L_{n+1},\;\cup_{n=1}^\infty L_n={\cal O}
\end{equation}
so that $K_n:={\cal G}(L_n)$ is an exhaustive sequence of compacts in $\Gamma$ 
(one can take $L_{n}=\{x\in {\cal O}\,|\,{\rm dist}(x,R^2\setminus\overline{{\cal O}})\ge 1/n\}$).

We define  
\begin{equation}\label{calV+}
{\cal W}={\cal W}_\ep={\cal W}_\ep(\Gamma), \;
{\cal W}^+={\cal W}^+_\ep={\cal W}_\ep^+(\Gamma)=\cup_{n\in N}{\cal W}^+_n
\end{equation}
$$
{\cal V}={\cal V}_\ep={\cal V}_\ep(\Gamma)=\cup_{n\in N}{\cal V}_n,\;
{\cal V}^+={\cal V}^+_\ep={\cal V}_\ep^+(\Gamma)=\cup_{n\in N}{\cal V}^+_n ,
$$
where
$$
{\cal W}_n=(-\delta_{K_n},\delta_{K_n})\times{\rm int}_2L_n,\;
{\cal W}_n^+=(0,\delta_{K_n})\times{\rm int}_2L_n,
$$
$$
{\cal V}_n:=\mathfrak{X}((-\delta_{K_n},\delta_{K_n});{\cal G}({\rm int}_2L_n))\subset 
{\cal U}_\ep(\Gamma),\;
{\cal V}^+_n:=\mathfrak{X}((0,\delta_{K_n});{\cal G}({\rm int}_2L_n))\subset 
{\cal U}_\ep^+(\Gamma).
$$
\end{enumerate}
\item {\it Construction of a local diffeomorphism}
\begin{enumerate}
 \item
Now, we define a map,
\begin{equation}\label{diff}
\Phi: (-\infty,\infty)\times{\cal O}\ni (s,\zeta)\mapsto \mathfrak{X}(s;{\cal G}(\zeta))\in \R^3.
\end{equation}
{Clearly, $\Phi$ is contiunous and $\Phi(0,\cdot)$ is a bijection from $L_n$ to $K_n$.}
{We shall prove that $\Phi|_{{\cal W}_n}$
is a
$C^1$ local diffeomorphisms of ${\cal W}_n$ onto
${\cal V}_n$. Likewise,  $\Phi|_{{\cal W}_n^+}$
is a
$C^1$ local diffeomorphisms of ${\cal W}_n^+$ onto
${\cal V}_n^+$. In particular, ${\cal V}$, ${\cal V}^+$ and ${\cal V}^+\cup\Gamma\cup\Omega$
are open.}
\item Indeed,
in view of the theorem of local inversion, it is enough to show that
$$
\forall \zeta\in {\cal O},\; s\in R,\;
{\rm det}\Big[\partial_s\Phi,\nabla_\zeta\Phi\Big](s,\zeta)\neq 0.
$$ 
Seeing that, $\mathfrak{X}(s;\mathfrak{X}(-s;\xi))=\xi$, we infer
$$
\forall\xi\in \R^3,\; s\in \R, \quad
\partial_s\mathfrak{X}(s; \mathfrak{X}(-s;\xi))+\vu_B(\mathfrak{X}(-s;\xi))
\cdot\nabla_x\mathfrak{X}(s; \mathfrak{X}(-s;\xi))=0,
$$
i.e., equivalently,
$$
\partial_s\Phi(s,\zeta)= -\vu_B({\cal G}(\zeta))
\cdot\nabla_\xi\mathfrak{X}(s;{\cal G}(\zeta)),\;\mbox{in particular, for all $s>0$, $\zeta\in{\cal O}$,}
$$
we easily find that
$$
\Big[\partial_s\Phi,\partial_{\zeta_1}\Phi, \partial_{\zeta_2}\Phi\Big](s,\zeta)= -
\begin{bmatrix}
[\vu_B({\cal G}(\zeta))]^T\\
[\partial_{\zeta_1}{\cal G}(\zeta)]^T\\
[\partial_{\zeta_1}{\cal G}(\zeta)]^T\\
\end{bmatrix} 
\;
\begin{bmatrix}
\Grad\mathfrak{X}_1(s;{\cal G}(\zeta)), \Grad\mathfrak{X}_2(s;{\cal G}(\zeta)), \Grad\mathfrak{X}_3(s;{\cal G}(\zeta))
\end{bmatrix}, 
$$
where vectors and $\Grad$ are columns.
Whence,
$$
{\rm det}\Big[\partial_s\Phi,\nabla_\zeta\Phi\Big](s,\zeta)=
-\vu_B\cdot\vc n({\cal G}(\zeta))
{\rm exp}\Big(-\int_0^s\Div\vu_B(z,\mathfrak{X}(z;{\cal G}(\zeta))){\rm d}z\Big) >0
$$
for all $\zeta\in {\cal O}$ and $s\in R$.
\end{enumerate}
\item{\it A diffeomorphism induced by $\Phi$.}
Finally, we observe (employing the uniform continuity and the fact that $\Phi$ is a local diffeomorphism), that $\delta_{K_n}$ can be chosen so small that
$$
\forall s_1,s_2\in (-\delta_{K_n},\delta_{K_n}),\ \Phi(s_1,L_n)\cap\Phi(s_2, L_n)=\emptyset.
$$
Consequently, in particular,  
\begin{equation}\label{diff+}
\forall \xi \in {\cal V}^+,\;\exists ! (s,x_B)\in {\mathfrak{W}}^+\; \mbox{such that}\
\xi=\mathfrak{X}(s,x_B),
\end{equation}
where $\mathfrak{W}^+=\cup_{n\in N}(0,\delta_{K_n})\times {\cal G}({\cal O}_n)$. 
\end{enumerate}
\item {\it Extension of the density beyond the inflow boundary.}
We may therefore extend the boundary data to 
${\cal V}^+$ by setting 
\begin{equation}\label{ecext}
\tilde{\mathfrak{r}} (t,\xi) = \mathfrak{r}_B({x}_B){\rm exp}\Big(\int_0^s{\rm div}\vu_B(\mathfrak{ X}(z;{x}_B)){\rm d}z\Big),\ \mbox{where $\xi=\mathfrak{X}(s,x_B)$}.
\end{equation}
Clearly,  $\tilde {\mathfrak{r}}\in C^{1}(\overline I\times\overline{{\cal V^+}})$ 
and
\begin{equation}\label{P2dod-}
\partial_t\tilde{\mathfrak{r}}+\Div (\tilde{\mathfrak{r}}\vu_B) = 0 \ \mbox{in}\ (0,T)\times {\cal V}^+,
\end{equation}
and
\begin{equation} \label{P2dod}
 \partial_tB(\tilde{\mathfrak{r}})+
 {\rm div}(B(\tilde{\mathfrak{r}})\vu_B)+[\tilde{{\mathfrak{r}}}
 B'(\tilde{\mathfrak{r}})-B(\tilde{\mathfrak{r}})]{\rm div}\vu_B=0\ \mbox{in
 $(0,T)\times\mathcal{V}^+$.}
 \end{equation}
\end{enumerate}

\noindent
{\it Step 3: Extension of the density beyond the outflow boundary}
\begin{enumerate}
 \item { We construct the open set $\mathcal{V}$, $\mathcal{V}^\pm$ 
 as in the Step 3.2 using the flow determined by EDO (\ref{EDO}),
 where we replace $-\vu_B$ by $\vu_B$. We denote this flow again by
 $\mathfrak{X}$ and the corresponding diffeomorphisme  by $\Phi$.} In particular, 
 for any $\xi\in\mathcal{V}^+$ there exists unique $(s,x_B)\in \mathfrak{W}^+$
 such that $\mathfrak{X}(s,x_B)=\xi$.
 \item We take a sequence 
 \begin{equation}\label{cvg1}
 C_c^1(\Gamma)\ni \mathfrak{r}^n\to \mathfrak{r}\ \mbox{in $L^\gamma(I;L^\gamma(\Gamma;|\vu_B\cdot\vc n|{\rm d}S_x))$}.
 \end{equation}
 \item We extend $\mathfrak{r}^n$ from $\Gamma$ to ${\cal V}^+$: $\forall (t,\xi)\in [0,T]\times\mathcal{V}^+$,
 \begin{equation}\label{extout}
 \tilde{\mathfrak{r}^n}(t,\xi)=\left\{
 \begin{array}{c}
  {\mathfrak{r}}^n(t-s,x_B){\rm exp}\Big(-\int_0^s{\rm div}\vu_B(\mathfrak{X}(\zeta,x_B)){\rm d}\zeta\Big),\; \xi=\mathfrak{X}(s,x_B)\;\mbox{if $t\ge s$}\\
  {\mathfrak{r}^n}(0,x_B){\rm exp}\Big(-\int_0^s{\rm div}\vu_B(\mathfrak{X}(\zeta,x_B)){\rm d}\zeta\Big),\; \xi=\mathfrak{X}(s,x_B)\;\mbox{if $0\le t<s$}
 \end{array}
 \right\}
 \end{equation}
 and easily verify that
 that the couple $(\tilde{{\mathfrak{r}}^n},\vu_B)$ verifies equations (\ref{P2dod-})--(\ref{P2dod}).
 \item 
 Seeing (\ref{cvg1}) an using change of variables $\xi \to \Phi(s,x_B)$ we infer
 \begin{equation}\label{conv2}
  \tilde{\mathfrak{r}^n}\to\tilde{\mathfrak{r}}\ \mbox{in
  $L^\gamma(I;L^\gamma({\cal V}^+_{\mathfrak{h}}))$},\ {\cal V}^+_{\mathfrak{ h}}=({\cal V}^+\cup\Gamma)\setminus B(\partial\Gamma;\mathfrak{h})
 \end{equation}
 with any $0<\mathfrak{h}$ sufficiently small. Consequently, we deduce from
 (\ref{extout}),
 \begin{equation}\label{extout-}
  \int_0^T\int_{{\cal V}^+}{\Big(\tilde{\mathfrak{r}}\partial_t\varphi+\tilde{\mathfrak{r}}\vu\cdot\Grad\vu_B\Big)}- \int_0^T\int_{\Gamma}
  \tilde{\mathfrak{r}}\cdot\vc n\varphi{\rm d}S_x{\rm d}t=0,\ \varphi\in C^1_c((0,T)\times({\cal V}^+\cup\Gamma))
 \end{equation}
 and
 \begin{equation}\label{extout+}
  \int_0^T\int_{{\cal V}^+}\Big(B(\tilde{\mathfrak{r}})\partial_t\varphi
  + B(\tilde{\mathfrak{r}})\cdot\vu_B\cdot\Grad\varphi -[
  \tilde{\mathfrak{r}} B'(\tilde{\mathfrak{r}})-B(\mathfrak{r})]{\rm div}\vu\Big){\rm d}x{\rm d}t
  \end{equation}
  $$
  -\int_0^T\int_{\Gamma}
  B(\tilde{\mathfrak{r}})\cdot\vc n\varphi{\rm d}S_x{\rm d}t=0,\ \varphi\in C^1_c((0,T)\times({\cal V}^+\cup\Gamma)),
 $$
 where, here, $\vc n$ is the outer normal to $\mathcal{V}^+$ (at $\Gamma$).
 \end{enumerate}
 {\it Step 4: Extension of the density field beyond the slip boundary.}\\
Let now $\Gamma =\Gamma^0$.  We take ${\cal V}^+={\cal V}_\ep^+(\Gamma)={\cal U}^+(\Gamma)$, cf. (\ref{calU++}), and we set in this case simply
\begin{equation}\label{0ext}
{ \tilde{\mathfrak{r}}}(t, x)=0,\;(t,x)\in (0,T)\times {\cal V}^+.
\end{equation}
Clearly, equations (\ref{P2dod-}) and (\ref{P2dod}) hold in this case.
\\ \\
 \noindent
{\it Step 5: Continuity equation extended.} 
\\
Referring to the decomposition (\ref{Gin0}), we construct ${\cal V}^+(\Gamma^{\rm in}_{{\mathfrak{k}}_{\rm in}})$ according
to (\ref{calV+}), ${\cal V}^+(\Gamma^{\rm out}_{{\mathfrak{k}}_{\rm out}})$ according to Item 1 in Step 3, and  ${\cal V}^+(\Gamma^0)$ according to (\ref{calU++}), cf. (\ref{0ext}). These open sets are mutually disjoint by virtue of 
(\ref{UtildeU}). Finally, we set,
\begin{equation}\label{Vglobal}
\tilde{\cal V}^+:=\Big[\cup_{{\mathfrak{k}}_{\rm in}}^{\overline{{\mathfrak{k}}_{\rm in}}}
\Big({\cal V}^+(\Gamma^{\rm in}_{{\mathfrak{k}}_{\rm in}})
\cup \Gamma^{\rm in}_{{\mathfrak{k}}_{\rm in}}\Big)\Big]
\cup
\Big[\cup_{{\mathfrak{k}}_{\rm in}}^{\overline{{\mathfrak{k}}_{\rm in}}}
\Big({\cal V}^+(\Gamma^{\rm out}_{{\mathfrak{k}}_{\rm in}})
\cup \Gamma^{\rm out}_{{\mathfrak{k}}_{\rm in}}\Big)\Big]
\cup \Big[
{\cal V}^+(\Gamma^0)\cup \Gamma^0\Big],\quad \tilde\Omega=\tilde{\cal V}^+\cup\Omega
\end{equation}
and extend $[\mathfrak{r},\vu]$ from $(0,T)\times\Omega$ to $(0,T)\times\tilde\Omega$ as follows
\begin{equation}\label{newru}
(\mathfrak{r},\vu)(t,x)=
\left\{
\begin{array}{c}
(\mathfrak{ r},\vu)(t,x)\;\mbox{if $(t,x)\in (0,T)\times\Omega$},\\
(\tilde{\mathfrak{r}}(t,x),\vu_B(x))\;\mbox{if $(t,x)\in (0,T)\times\tilde{\cal V}^+$}
\end{array}
\right\},
\end{equation} 
where $\tilde{\mathfrak{r}}$ in $\tilde{\cal V}^+$ is defined through (\ref{ecext}) or (\ref{extout})--(\ref{conv2}) or (\ref{0ext}), according to the case.

By virtue of (\ref{cow})$_{v=0}$ and (\ref{P2dod-}), (\ref{0ext}), (\ref{extout-})
we easily deduce, that the new couple $[\mathfrak{r}, \vu]$ satisfies continuity equation (\ref{co1})$_{v=0}$ in the sense of distributions
on ${\cal D}((0,T)\times\tilde\Omega)$.
\\ \\
{\it Step 6: Application of the DiPerna-Lions regularization, proof of equation (\ref{P3}).}
Next, we use the regularization procedure due to DiPerna and Lions \cite{DL} applying convolution with a family of regularizing kernels
obtaining for the regularized function $[\mathfrak{r}]_{\mathfrak{e}}$,
\begin{equation} \label{P4}
\partial_t[\mathfrak{r}]_{\mathfrak{e}}+\Div ([\mathfrak{r}]_{\mathfrak{e}} \vu ) = R_{\mathfrak{e}} \ \mbox{a.e. in} \ (0,T)\times \tilde\Omega_{{\mathfrak{e}}},
\end{equation}
where
$$
\tilde\Omega_{{\mathfrak{e}}} = \left\{ x \in \tilde\Omega\ \Big| \ {\rm dist}(x, \partial \tilde\Omega ) > {\mathfrak{e}} \right\},
R_{\mathfrak{e}}:=\Div ([\mathfrak{r}]_{\mathfrak{e}} \vu )-\Div ([\mathfrak{r} \vu]_{\mathfrak{e}} ) \to 0 \ \mbox{in} \ L_{\rm loc}^1((0,T)\times\tilde\Omega) \ \mbox{as}\ {\mathfrak{e}} \to 0.
$$
The convergence of $R_{\mathfrak{e}}$ evoked above results from the application of the refined version of the Friedrichs lemma on commutators, see e.g. \cite{DL}
or \cite[Lemma 10.12 and Corollary 10.3]{FeNoB}.

Multiplying equation (\ref{P4}) on $B'([\mathfrak{r}]_{\mathfrak{e}})$, we get
\begin{equation}\label{P4*}
\partial_tB([\mathfrak{r}]_{\mathfrak{e}})+ \Div (B([\mathfrak{r}]_{\mathfrak{e}}) \vu ) + \left( B'([\mathfrak{r}]_{\mathfrak{e}}) [\mathfrak{r}]_{\mathfrak{e}} - B([\mathfrak{r}]_{\mathfrak{e}} ) \right) \Div \vu = B'([\mathfrak{r}]_{\mathfrak{e}}) R_{\mathfrak{e}}
\end{equation}
or
\[
\int_0^T\int_{\tilde\Omega} \Big(B([\mathfrak{r}]_{\mathfrak{e}})\partial_t\varphi+B([\mathfrak{r}]_{\mathfrak{e}}) \vu \cdot \Grad \varphi-\varphi \left( B'([\mathfrak{r}]_{\mathfrak{e}} ) [\mathfrak{r}]_{\mathfrak{e}} - B([\mathfrak{r}]_{\mathfrak{e}}) \right) \Div \vu \Big)\ \dx{\rm d}t
$$
$$
 - \int_0^T
\int_{\tilde\Omega}{ \varphi B'([\mathfrak{r}]_{\mathfrak{e}}) R_{\mathfrak{e}} }{\rm d} x{\ d }t=0
\]
for any $\varphi \in C^1_c (I\times\tilde\Omega)$, $0<{\mathfrak{e}}< {\rm dist}({\rm supp}(\varphi),\partial\tilde\Omega)$. 
Since the last term at the right hand side converges to $0$ by virtue of the Friedrichs commutator lemma (cf. Lemma \ref{Friedrichs}), letting ${\mathfrak{e}} \to 0$, we get
\begin{equation}\label{dod1**}
\int_0^T\int_{\Omega} \Big(B(\mathfrak{r})\partial_t\varphi+B(\mathfrak{r}) \vu \cdot \Grad \varphi
-\varphi \left( B'(\mathfrak{r} ) \mathfrak{r} - B(\mathfrak{r}) \right) \Div \vu \Big)\ \dx{\rm d}t
\end{equation}
$$
+
\int_0^T\int_{\tilde{\cal V}^+} \Big(B(\mathfrak{r})\partial_t\varphi+B(\mathfrak{r}) \vu \cdot \Grad \varphi
-\varphi \left( B'(\mathfrak{r} ) \mathfrak{r} - B(\mathfrak{r}) \right) \Div \vu \Big)\ \dx{\rm d}t=0
$$
for any $\varphi \in C^1_c (I\times\tilde\Omega)$,
where, by virtue of (\ref{P2dod}), (\ref{extout+}), (\ref{0ext}),
$$
\int_0^T\int_{\tilde{\cal V}^+} \Big(B(\mathfrak{r})\partial_t\varphi+B(\mathfrak{r}) \vu \cdot \Grad \varphi
-\varphi \left( B'(\mathfrak{r} ) \mathfrak{r} - B(\mathfrak{r}) \right) \Div \vu \Big)\ \dx{\rm d}t
$$
$$
=-\int_0^T\int_{\partial\Omega}B(\mathfrak{r})\vu_B\cdot \vc n\varphi{\rm d}S_x{\rm d}t,\ \mathfrak{r}=\mathfrak{r}_B\;\mbox{on $\Gamma^{\rm in}$}.
$$

Using in (\ref{dod1**}) test function 
$$
\varphi(t,x)=\phi(1-\chi_{\mathfrak{h}}),\ \phi\in C^1_c(I\times\overline\Omega),
\ \chi_{\mathfrak{h}}(x)=\chi\Big(\frac {d_{{\mathfrak{g}}^{\rm bd}}(x)}{\mathfrak{h}}\Big)\Big),
$$
where
$$
\chi\in C^1[0,\infty),\;|\chi'(x)|\le 3,\;
\chi(x)
\left\{\begin{array}{c}
\in [0,1]\\
=1\;\mbox{if $x\in [0,1/2]$}\\
=0\;\mbox{if $x>1$}\\
\end{array}
\right\},
$$
and $\mathfrak{h}$ is a positive sufficiently small number,
we get
$$
\int_0^T\int_{\Omega} \Big(B(\mathfrak{r})\partial_t\phi(1-\chi_{\mathfrak h})+B(\mathfrak{r}) \vu \cdot \Grad \phi (1-\chi_{\mathfrak{h}})
-\phi (1-\chi_{\mathfrak{h}}) \left( B'(\mathfrak{r} ) \mathfrak{r} - B(\mathfrak{r}) \right) \Div \vu \Big)\chi_{\mathfrak{h}}\ \dx{\rm d}t
$$
\begin{equation}\label{diff++}
+\int_0^T\int_{\Gamma^{\rm out}}B(\mathfrak{r})\vu_B\cdot \vc n\varphi\chi_{\mathfrak{h}}{\rm d}S_x{\rm d}t  +\int_0^T\int_{\Gamma^{\rm in}}B(\mathfrak{r})\vu_B\cdot \vc n\phi\chi_{\mathfrak{h}}{\rm d}S_x{\rm d}t 
\end{equation}
$$
-\int_0^T\int_{\Omega\cap B({\mathfrak{g}}^{\rm bd};\mathfrak{h})}
\phi B(\mathfrak{r})\vu\cdot\Grad\chi_{\mathfrak{h}}\dx\dt=0,
$$ 
where the last term at the left hand side can be written as
$$
\int_0^T\int_{\Omega\cap B({\mathfrak{g}}^{\rm bd};\mathfrak{h})}
\phi B(\mathfrak{r})(\vu-\vu_B)\cdot\Grad\chi_{\mathfrak{h}}\dx\dt
+ \int_0^T\int_{\Omega\cap B({\mathfrak{g}}^{\rm bd};\mathfrak{h})}
\phi B(\mathfrak{r})\vu_B\cdot\Grad\chi_{\mathfrak{h}}\dx\dt.
$$ 
The first term in the latter expression tends to $0$ as $\mathfrak{h}\to 0$ by virtue of the
Hardy inequlity and the second one tends to $0$ due to the H\"older inequality and (\ref{ruB}), cf. (\ref{dLipschitz}) and (\ref{Tubes}). 

Thus letting $\mathfrak{h}\to 0$ we obtain the desired result,
namely
$$
\int_0^T\int_{\Omega} \Big(B(\mathfrak{r})\partial_t\phi+B(\mathfrak{r}) \vu \cdot \Grad \phi
-\phi \left( B'(\mathfrak{r} ) \mathfrak{r} - B(\mathfrak{r}) \right) \Div \vu \Big)\ \dx{\rm d}t
$$
$$
-\int_0^T\int_{\Gamma^{\rm out}}B(\mathfrak{r})\vu_B\cdot \vc n\phi{\rm d}S_x{\rm d}t  -\int_0^T\int_{\Gamma^{\rm in}}B(\mathfrak{r})\vu_B\cdot \vc n\phi{\rm d}S_x{\rm d}t=0,\ \phi\in C^1_c(I\times\overline\Omega). 
$$
{\it Step 7: Time integrated renormalized continuity equation}\\
Since the couple $(\mathfrak{r},\vu)$ satisfies (\ref{cow})$_{v=0}$, and since
$\mathfrak{r}\in L^\infty(I; L^\gamma(\Omega))$,
it is standard to see that ${\mathfrak{r}}\in C_{\rm weak}(\overline I;L^\gamma(\Omega))$. Moreover, since the renormalized
equation holds in the sense of distributions, one can infer that ${\mathfrak{r}}\in C(\overline I;L^1(\Omega))$
and  the renormalized time integrated equation (\ref{P3}) holds.
This can be deduced from Di-Perna, Lions \cite{DL}, see e.g. \cite[Theorems 3, 5]{AN-MP+} for more details.
\\ \\
{\it Step 8: The case $M>1$:}\\
In this case, we obtain instead of (\ref{P4}) $M$ identities,
$$
\partial_t[\mathfrak{r}^i]_{\mathfrak{e}})
+ \Div ([{r}^i]_{\mathfrak{e}}) \vu ) =R^i_{\mathfrak{e}}:=
{\rm div}([r^i]_{\mathfrak{e}}\vu)- {\rm div}[r^i\vu]_{\mathfrak{e}},\;
i=1,\ldots, M
$$
We obtain the required result by multiplying $i$-th equation by
$\partial_i B([\mathfrak{r}^1]_\ep,\ldots,[\mathfrak{r}^M]_\ep)$,
summing the resulting equations and then proceeding
in the same way as in Step 6.
\\ \\
This finishes the proof of Proposition \ref{LP2}.
\\ \\



{
Proposition \ref{LP2} gives rise to several useful corollaries. Before stating them  we define 
\begin{equation}\label{renofrac+}
\forall \tau\in \overline I,\ \mbox{for a.a. $x\in \Omega$},\quad
\mathfrak{s}_{\mathfrak{d}}(\tau,x)= [Z/_{\mathfrak{d}}R](\tau,x):=
\left\{\begin{array}{c}
\frac {Z(\tau,x)}{R(\tau,x}\;\mbox{if $R(\tau,x)\neq 0$},\\
\mathfrak{d}\;\mbox{if $R(\tau,x)= 0$}
\end{array}\right\},\ \mathfrak{d}\in \R,
\end{equation}
$$
\mbox{for a.a. $(\tau,x)\in I\times\partial\Omega$} ,\quad
\mathfrak{s}_{\mathfrak{d}}(\tau,x)= [Z/_{\mathfrak{d}}R](\tau,x):=
\left\{\begin{array}{c}
\frac {Z(\tau,x)}{R(\tau,x}\;\mbox{if $R(\tau,x)\neq 0$},\\
\mathfrak{d}\;\mbox{if $R(\tau,x)= 0$}
\end{array}\right\}.
$$

The first of the corollaries is the following:

\begin{cor}{\rm [From continuity to pure transport equation]}\label{cont-tr}
Let $\gamma$, $\Omega$, $\vu_B$ and $\vu$ be the same as in Proposition
\ref{LP2}. Suppose that
$$
0\le Z\le\overline a R,\; R\in L^2(I,L^2(\Omega))\cap L^\infty(I;L^\gamma(\Omega))\cap L^\gamma(I; L^\gamma(\Gamma^{\rm out};
|\vu_B\cdot\vc n|{\rm d} S_x))
$$
and that both $Z$ and $R$ are weak solutions of the continuity equation (\ref{co1}) with initial conditions
$$
0\le Z_0\le\overline a R_0,\; R_0\in L^\gamma(\Omega)
$$
and boundary conditions
$$
0\le Z_B\le \overline a R_B, \; R_B\in C(\overline\Omega).
$$
Then we have: The quantities
$$
Z,R\in C_{\rm weak}(\overline I;L^\gamma(\Omega))\cap C(\overline I;L^p(\Omega)),\;1\le p<\gamma,
$$
and for any $\mathfrak{d}\in R$ the quantity $\mathfrak{s}_{\mathfrak{d}}:=Z/_{\mathfrak{d}} R$ (cf. (\ref{renofrac+}) 
 belongs to $C(\overline I;L^p(\Omega))\cap L^\infty(I,L^\infty(\Gamma^{\rm out})) $, and it satisfies
the pure transport equation
$$
\intO{{\mathfrak{s}}_{\mathfrak{d}}\varphi(\tau)} - \intO{{\mathfrak{s}}_{0,\mathfrak{d}}\varphi(0)}+\int_0^\tau\int_{\Gamma^{\rm out}}
{\mathfrak{s}}_{\mathfrak{d}}\vu_B\cdot\vc n{\rm d}S_x{\rm d}t +\int_0^\tau\int_{\Gamma^{\rm in}}
{\mathfrak{s}}_{B,\mathfrak{d}}\vu_B\cdot\vc n{\rm d}S_x{\rm d}t 
$$
\begin{equation}\label{trs}
=\int_0^\tau\intO{\Big({\mathfrak{s}}_{\mathfrak{d}}\partial_t\varphi +
{\mathfrak{s}}_{\mathfrak{d}}\vu\cdot\Grad\varphi + \varphi{\mathfrak{s}}_{\mathfrak{d}}
\Div\vu\Big)}{\rm d}t
 \end{equation}
with any $\tau\in\overline I $ and any $\varphi\in C^1_c([0,T]\times \overline\Omega).$
In the above, $\mathfrak{s}_{0,\mathfrak{d}}=Z_0/_\mathfrak{d} R_0$, $\mathfrak{s}_{B,\mathfrak{d}}=Z_B/_\mathfrak{d} R_B$.
\end{cor}

Indeed, the identity (\ref{trs}) can be obtained from (\ref{P3}) with $M=2$, $r^{(1)}=Z$, $r^{(2)}=R$, $B(R,Z)= \frac {Z}{R+\mathfrak{a}}$
after letting $\mathfrak a\to 0$ with help of the Lebesgue dominated convergence theorem. 
\\

The last item of Proposition \ref{LP2}, namely identity (\ref{P3rs}) yields readily the following corollary about the "almost uniqueness" for the transport equation.

\begin{cor}{\rm [Almost uniqueness to the pure transport equation]}\label{teu}
Let $\Omega$, $\vu$, $\vu_B$ be the same as in Proposition \ref{LP2}. Let $0\le s^{(i)}\in L^\infty(Q_T)\cap C(\overline I,L^1(\Omega)\cap L^\infty(I;
L^\infty(\Gamma^{\rm out}))$, 
$i=1,2$ be two weak solutions of 
the pure transport equation (\ref{co1}) 
(i.e. they satisfy (\ref{cow}) with $v=-{\rm div}\vc u$).  

If  $s^{(1)}(0,\cdot)=s^{(2)}(0,\cdot)$, $s_B^{(1)}=s_B^{(2)}$ then
\begin{equation}\label{uset}
\mbox{for all $\tau\in\overline I$}\;s^{(1)}(\tau,\cdot)=s^{(2)}(\tau,\cdot)\;\mbox{for a.a. $x\in\{\vr(\tau,\cdot)>0\}$},
\end{equation}
$$
s^{(1)}=s^{(2)}\ \mbox{ a.e. in
$\{(t,x)\in I\times\Gamma^{\rm out}|\vr >0\}$},
$$
where   $\vr$ is {\it any} weak solution to the continuity equation (\ref{co1}) (i.e. satisfying (\ref{cow}) with
$v=0$) in the class $0\le\vr\in C(\overline I,L^1(\Omega))\cap L^2(I; L^2(\Omega))\cap L^\infty(I;L^p(\Omega))\cap L^\gamma(I; L^p(\Gamma^{\rm out}(|\vu_B\cdot\vc n|{\rm d}S_x))$, $p>1$.
\end{cor}

To prove Corollary \ref{teu}, it is enough to take in formula (\ref{P3rs}) in Proposition \ref{LP2} $N=2$, $r=\vr$, $B(s^{(1)}, s^{(2)}) =
(s^{(1)}- s^{(2)})^2$.

This result generalizes \cite[Proposition 5]{NoSCM} from the case of zero transporting velocity to the  case of general boundary data. It also generalizes
the uniqueness results from seminal paper of DiPerna-Lions \cite[Theorem II.2]{DL} and its improvement which can be deduced from Bianchini-Bonicatto 
\cite{BIBO}.\\

%

The next corollary is one of the crucial point of the compactness argument in the existence proof. The case with zero velocity at the boundary has been treated in Vasseur  et al. \cite{VWY}, and improved in \cite[Proposition 7]{AN-MP}. The generalization of \cite[Proposition 7]{AN-MP} to the
general boundary data reads as follows.

\begin{cor}\label{L2}
We suppose that $\Omega$, $\vu_B$, $R_B,Z_B$,  $R_0,Z_0$ satisfy assumptions of Corollary \ref{cont-tr}.
Let 
$$
\vu_n\in L^2(I,W^{1,2}(\Omega;\R^3)),\;\vu_n|_{I\times\partial\Omega}=\vu_B,\
0\le Z^n\le \overline a R^n,
$$
$$
(R_n, Z_n)\in 
L^\infty(I;L^\gamma(\Omega))
\cap L^2(I;L^2(\Omega))\cap  L^\gamma(I; L^\gamma(\Gamma^{\rm out};
|\vu_B\cdot\vc n|{\rm d} S_x)),
$$
where $\gamma>1$.
Suppose that
\begin{equation}
\sup_{n\in N}\Big(\|R_n\|_{L^2(Q_T)}+
\|R_n\|_{L^\infty(I;L^\gamma(\Omega))} + \|\vu_n\|_{L^2(I;W^{1,2}(\Omega))}\Big)
<\infty,
\end{equation}
and that both couples $(R_n,\vu_n)$, $(Z_n,\vu_n)$ satisfy continuity equation in the weak sense
(i.e. (\ref{cow}) with $v=0$, $\vu=\vu_n$ holds). Then:
\begin{enumerate}
\item 
$R_n, B(R_n)$, $Z_n, B(Z_n)$ $\in C_{\rm weak}(\overline I;L^\gamma(\Omega))\cap C(\overline I;L^p(\Omega))$,
$1\le p<\gamma$ and each of $R_n$, $Z_n$ is a renormalized solution of the continuity equation, i.e.
it satisfies equation (\ref{P3})$_{\vu=\vu_n, M=1}$ with $B$ specified in (\ref{B}), see
also Remark \ref{rinflow}.
\item
Up to a subsequence (not relabeled)
\begin{equation}\label{RnZn}
\begin{aligned}
(R_n, Z_n)&\to (R,Z)\;\mbox{in $C_{\rm weak}(\overline I;L^\gamma(\Omega))$ and  weakly in  $L^\gamma(I; L^\gamma(\Gamma^{\rm out};
|\vu_B\cdot\vc n|{\rm d} S_x))$},
\\
\vu_n&\rightharpoonup\vu\;\mbox{in $L^2(I;W^{1,2}(\Omega;\R^3))$,}
\end{aligned}
\end{equation}
where $(R,Z)$, $0\le Z\le R$ belongs to spaces
$$
L^2(I; L^2(\Omega))\cap L^\infty(I,L^\gamma(I,\Omega))\cap C(\overline I;L^p(\Omega)),\; 1\le p<\gamma
$$
and $(R,\vu)$ as well as { $(Z,\vu)$} verify continuity equation in the renormalized sense (i.e., integral identity (\ref{P3})$_{\vu, M=1}$ with $B$ specified in (\ref{B}), see
also Remark \ref{rinflow}), is satisfied.
\item Let $\mathfrak{d}\in [0,\overline a]$. We define in agreement with convention (\ref{renofrac+}) for all $t\in \overline I$, 
\begin{equation}\label{sn}
s_B(x)={Z_B(x)/_{\mathfrak{d}}R_B(x)},\;
s_n(t,x)={Z_n(t,x)/_{\mathfrak{d}}R_n(t,x)},\; s(t,x)= {Z(t,x)}/_{\mathfrak{d}}{R(t,x)}.
\end{equation}
Then $s_n, s\in C(\overline I;L^q(\Omega))$, $1\le q<\infty$ and  for all $t\in \overline I$, 
$0\le s_n(t,x)\le\overline a$, $0\le s(t,x)\le \overline a$ for a.a. $x\in \Omega$, for a.a. $(t,x)\in I\times\partial\Omega$, $0\le s_n(t,x)\le\overline a$, $0\le s(t,x)\le \overline a$. Moreover, both
$(s_n,\vu_n)$ and $(s,\vu)$ satisfy the pure
transport equation in the renormalized sense (i.e., integral identity (\ref{P3s})$_{\vu, N=1}$ with $B$ specified in (\ref{B}), see also Remark \ref{rinflow}), is satisfied. 
\item
Finally,
\begin{equation}\label{cvs}
\int_{\Omega}(R_n|s_n-s|^2)(\tau,\cdot)\,{\rm d} x+
\int_0^\tau\int_{\Gamma^{\rm out}} R_n|s_n-s|^2{\rm d}S_x{\rm d}t \to 0\;\mbox{for all $\tau\in \overline I$.}
\end{equation}

\end{enumerate}
\end{cor} 
\noindent

The first statement is statement 1 of Proposition \ref{LP2}. The second statement is nowadays a mathematical folklore (see (\ref{conv0}--\ref{ruinfty})
for the reasoning). The third statement follows from statement 3 in Proposition \ref{LP2}.

We prove (\ref{cvs}). We realize, employing (\ref{RnZn}) and (\ref{sn}),  
\begin{equation}\label{d++}
\forall \tau\in [0,T],\;
\lim_{n \to \infty}\Big(\intO{R_n({s}_{n}-{s})^2(\tau)} 
+ \int_0^\tau\int_{\Gamma^{\rm out }}R_n{s_n}^2
|\vu_B\cdot\vc n |{\rm d} S_x{\rm d}t\Big)
\end{equation}
$$
=\lim_{n\to \infty}\Big(
\intO{R_n\mathfrak{s}_{n}^2(\tau)} + \int_0^\tau\int_{\Gamma^{\rm out }}R_n{s_n}^2
|\vu_B\cdot\vc n |{\rm d} S_x{\rm d}t\Big)
$$
$$
-\intO{R{s}^2(\tau)} - \int_0^\tau\int_{\Gamma^{\rm out }}R{s}^2
|\vu_B\cdot\vc n |{\rm d} S_x{\rm d}t,
$$
where $R_ns_n^2$ satisfies continuity equation, in particular, with test function $\varphi=1$,
$$ 
\intO{R_n{s}^2_{n}(\tau,x)}+ \int_0^\tau\int_{\Gamma^{\rm out }}R_n{s_n}^2
|\vu_B\cdot\vc n| {\rm d} S_x{\rm d}t
$$
$$
=\intO{R_0(x){s}^2(0,x)}-\int_0^\tau\int_{\Gamma^{\rm in}}R_B{s}^2_B
\vu_B\cdot\vc n {\rm d} S_x{\rm d}t,
$$
while $Rs^2$ satisfies the same continuity equation with $\vu_n$ replaced by $\vu$. Inserting both latter identities into the right hand side of (\ref{d++})
yields the statement.
}
}

\section{Approximations}\label{APPR}

Starting from now, we shall suppose, without loss of generality,
\begin{equation}\label{initr}
0<r_0\in C^1(\overline\Omega),\; \vu_0=\vv_0+\vu_B,\;\vv_0\in  C^1(\overline\Omega).
\end{equation}

The proof of Theorem \ref{theorem1} is based on a multilevel approximation scheme that shares certain common features with the approximation of the 
compressible Navier--Stokes in \cite{FNP}, see also monographs \cite{EF70} or \cite{NoSt}. First, we introduce a sequence of finite--dimensional spaces 
$X_n \subset L^2(\Omega; \R^3)$,
\[
X_n = {\rm span} \left\{ \vc{w}_i\ \Big|\ \vc{w}_i \in \DC(\Omega; \R^3),\ i = 1,\dots, n \right\}.
\]   
Without loss of generality, we may assume that $\vc{w}_i$ are orthonormal with respect to the standard scalar product in $L^2 (\Omega)$.

Following \cite{ChJNo}, \cite{KwNo} and \cite{AN-MP}, \cite{NoSCM} we use the following parabolic approximation of the continuity equations,
\begin{equation} \label{E1}
\partial_t r + \Div (r \vu ) = \ep \Delta r \ \mbox{in}\ (0,T) \times \Omega,\ \ep > 0,
\end{equation}
supplemented with the boundary conditions 
\begin{equation} \label{E2}
\ep \Grad r \cdot \vc{n} + (r_B - r) [\vu_B \cdot \vc{n}]^- = 0 \ \mbox{in}\ [0,T] \times \partial \Omega,
\end{equation}
and the initial condition
\begin{equation} \label{E3}
r(0, \cdot) = r_0.
\end{equation}
Here, $\vu = \vv + \vu_B$, with $\vv \in C([0,T]; X_n)$, in particular, $\vu|_{\partial \Omega} = \vu_B$ and $r$
stands for $\vr,z,R,Z$, according to the case.  Note that for given $\vu$, $r_B$, $\vu_B$, this is a linear parabolic problem with the Robin boundary conditions for the unknown $r$. 

Following \cite{KwNo}, \cite{NoSCM}, we use the Galerkin approximation of the momentum equation: we look for the approximate velocity field in the form 
\[
\vu = \vv + \vu_B, \ \vv \in C([0,T]; X_n). 
\]
Accordingly, the
approximate momentum balance reads
{
\begin{equation} \label{E5}
  \intO{ (\vr+z) \vu \cdot \bfphi } \Big|_{t=0}^{t = \tau} = 
\int_0^\tau \intO{ \Big[ (\vr+z) \vu \cdot \partial_t \bfphi + (\vr+z) \vu \otimes \vu : \Grad \bfphi 
+ P_\delta(R,Z) \Div \bfphi 
\end{equation}
$$
- \mathbb{S}(\nabla\vu): \Grad \bfphi -\ep\Grad (\vr+z) \cdot \Grad \vu \cdot \bfphi \Big] }\dt
$$
}
%
for any $\bfphi \in C^1([0,T]; X_n)$, with the initial condition
\begin{equation} \label{E6}
(\vr+z) \vu(0, \cdot) = (\vr_0+z_0) \vu_0, \ \vu_0 = \vv_0 + \vu_B, \ \vv_0 \in X_n,
\end{equation}
where 
\begin{equation}\label{Pdelta}
P_\delta(R,Z)=
P(R,Z)+\delta(R^{\mathfrak{ c}}+Z^{\mathfrak{ c}} ),\;\mathfrak{c}>\max\{\frac 92,\beta,\gamma\}.
\end{equation}
For fixed parameters $n$, $\delta > 0$, $\ep > 0$, the first level approximation is a solution $[\vr, z,  R, Z,  \vu]$
\footnote{Here in the sequel, we skip the indexes $n$, $\ep$, $\delta$ and write e.g. $\xi$ instead of
$\xi_{n,\ep,\delta}$, etc. and will use eventually only one of them in the situations when it will be useful to underline the corresponding limit passage.}  
of the parabolic problem \eqref{E1}--\eqref{E3}, and the Galerkin approximation \eqref{E5}, \eqref{E6}.

\subsection{Parabolic problem (\ref{E1}--\ref{E3})}

In contrast with \cite{ChJNo}, \cite{KwNo}, we do not want to use the maximal regularity theory of parabolic equations
(which requires at least $C^2$ boundary, see Denk, H\"uber, Pr\"uss \cite{DeHuPr}), but, we shall employ rather
the theory from Crippa, Donadello, Spignolo \cite{Crippa} which allows merely the Lipschitz boundaries.

For Lipschitz domains the usual parabolic estimates fail at the level of the spatial derivatives and we 
are forced to use the weak formulation: 
\begin{equation} \label{E4}
\begin{split}
\left[ \intO{ r \varphi } \right]_{t = 0}^{t = \tau}&= 
\int_0^\tau \intO{ \left[ r \partial_t \varphi + r \vu \cdot \Grad \varphi - \ep \Grad r \cdot \Grad \varphi \right] }
\dt \\ 
&- \int_0^\tau \int_{\partial \Omega} \varphi r \vu_B \cdot \vc{n} \ {\rm d} S_x \dt + 
\int_0^\tau \int_{\partial \Omega} \varphi (r - r_B) [\vu_B \cdot \vc{n}]^{-}  \ {\rm d}S_x \dt,\ 
r(0, \cdot) = r_0, 
\end{split}
\end{equation}
for any test function
\[
\varphi \in L^2(0,T; W^{1,2}(\Omega)),\ \partial_t \varphi \in L^1(0,T; L^2(\Omega)).
\]

\begin{lem} \label{EL1}
Let $\Omega \subset \R^3$ be a bounded Lipschitz domain and $\vu = \vv + \vu_B$, $\vv \in C(\overline I; X_n)$. Suppose that
$(r_B,\vu_B)$ belongs to the class (\ref{ruB}) while $(r_0,\vu_0)$ belongs to the class (\ref{initr}). Then we have:
\begin{enumerate}
\item
The initial--boundary value problem (\ref{E1}--\ref{E3}) 
admits a weak solution $r$ {specified in (\ref{E4})}, unique in the class 
$$
r \in L^2(I; W^{1,2}(\Omega)) \cap C(\overline I; L^2(\Omega)).
$$
The norm in the aforementioned spaces is bounded only in terms of the data $r_B$, $r_0$, $\vu_B$ and 
$\|\vv, {\rm div} \vv\|_{L^\infty(I;L^\infty(\Omega))}$.
\item Moreover, $\partial_t r\in L^2(I\times\Omega)$ and $\sqrt\ep\nabla r\in L^\infty(I;L^2(\Omega))$ are bounded in terms of the data $r_B$, $r_0$, $\vu_B$ and 
$\|\vv, {\rm div} \vv\|_{L^\infty(I;L^\infty(\Omega))}$ and $\nabla^2 r\in L^2(I; L^2_{\rm loc}(\Omega)$ is bounded in the same way on any compact set $K$ of $\Omega$
with the constant dependent in addition on $K$.
\item Strong maximum principle: The solution satisfies,
\begin{equation}\label{max}
\begin{array}{c}
{
\forall \tau\in\overline I,\;\| r(\tau) \|_{L^\infty( \Omega)}
}
 \leq M
 \exp \left( T \| \Div \vu \|_{L^\infty((0,\tau) \times \Omega)} \right),
 \\ \\
\mbox{for a.a. $\tau\in I$},\; r(\tau,x)\leq M\exp \left( T \| \Div \vu \|_{L^\infty((0,\tau) \times \Omega)} \right)\;\mbox{for a.a. $x\in \partial\Omega$},
\end{array}
\end{equation}
where
$$
M=\max \left\{\max_\Omega r_0,\max_{\Gamma^{\rm in}} r_B, 
\| \vu_B \|_{L^\infty((0,T) \times \Omega)} \right\}.
$$
\item Renormalization: For any $B\in C^2(R)$,
\begin{equation}\label{reno}
\intO{B(r)}\Big|_0^\tau+\ep\int_0^\tau\intO{|\Grad r|^2 B''(r)}{\rm d}t+ \int_0^\tau\intO{\Big(rB'(r)-B(r)\Big)\Div\vu}{\rm d}t
\end{equation}
$$
-\int_0^\tau\int_{\partial\Omega}[\vu_B\cdot\vc n]^- E_B(r_B|r){\rm d}S_x{\rm d}t
+ \int_0^\tau\int_{\partial\Omega}[\vu_B\cdot\vc n]^+ B(r){\rm d}S_x{\rm d}t= 
-\int_0^\tau\int_{\partial\Omega}[\vu_B\cdot\vc n]^- B(r_B){\rm d}S_x{\rm d}t,
$$
where $\tau\in \overline I$ and 
$$ 
E_B(r|\tilde r)=B(r)- B'(\tilde r)(r-\tilde r)-B'(\tilde r).
$$
\item Strong minimum principle: The solution satisfies,
\begin{equation}\label{min}
\begin{array}{c}
 \forall \tau\in\overline I,\;{\rm ess} \inf_{x\in \Omega} r(\tau,x) \geq  
m
\exp \left( -T \| \Div \vu \|_{L^\infty((0,T) \times \Omega)} \right),
\\ \\
\mbox{for a.a. $\tau\in I$},\; r(\tau,x)\ge m\exp \left(- T \| \Div \vu \|_{L^\infty((0,\tau) \times \Omega)} \right)\;\mbox{for a.a. $x\in \partial\Omega$},
\end{array}
\end{equation}
where
$$
m=\min \left\{ \min_\Omega r_0 , \min_{\Gamma^{\rm in}} r_B \right\}. 
$$
\end{enumerate}
\end{lem}

The first item is a particular case of Lemma 3.2 in Crippa et al. \cite{Crippa}, see also \cite[Lemma 3.1]{AbFeNo}. The latter reference contains also the proof of the second item. 
The third item, maximum principle, is proved  Crippa et al. \cite[Lemma 3.4]{Crippa}, see also \cite[Lemma 3.2]{AbFeNo}. Renormalization is proved in \cite[Lemma 3.3]{AbFeNo}. 
Finally, the minimum principle is shown in the latter reference in Corollary 3.4.\footnote{{ Inequalities (\ref{max}) and (\ref{min}) are
proved in \cite{AbFeNo} with ${\rm ess} \sup_{t,x} r(t,x)$ and
${\rm ess} \inf_{t,x} r(t,x)$, respectively. They however hold for all
 $\tau\in \overline I$ provided, in addition, $r\in C_{\rm weak}(\overline I; L^\gamma(\Omega))$, $\gamma\ge 1$. Indeed, suppose for example for the quantity $r$ in addition to the latter regularity, $r(t,x)\ge 0$  for a.a. $(t,x)\in I\times\Omega$. Then for any $M\subset\Omega$, $|M|>0$, any $\tau_0\in I$ and $\delta>0$ "small", $\int_{\tau_0-\delta}^{\tau_0+\delta}
\int_M r(t,x){\rm d}x{\rm d}t\ge 0$; whence $\int_M r(\tau_0,x){\rm d}x\ge 0$, and
$r(\tau_0,\cdot)\ge 0$ a. a. in $\Omega$ by the theorem on Lebesgue points. Since $r\in L^2(I;W^{1,2}(\Omega))$, the second inequality in (\ref{max}) resp. in (\ref{min}) follows from the first one and the trace theorem.}
}

\subsection{Existence of approximations at level I}

The \emph{existence} of the approximate solutions at the level of 
the parabolic problem (\ref{E1}--\ref{E3}) coupled with the Galerkin approximation (\ref{E5}--\ref{E6})
can be proved in the same way as in \cite[Section 4]{ChJNo} (mono-fluid case with non zero inflow-outflow) combined with \cite[Section 3]{NoSCM}, 
eventually with \cite[Section 4]{AN-MP} (multi-fluid with zero boundary conditions).  Specifically, 
for $\vu = \vu_B + \vv$, $\vv \in C([0,T]; X_n)$, we identify the unique solutions $r = 
r [\vu]$ of (\ref{E1}--\ref{E3}), where $r$ stands for $\vr, z, R, Z$ and plug them as $\vr$, $z$, $R$, $Z$ in \eqref{E5}. 
The unique solution $\vu = \vu[\vr,z,R,Z]$ of \eqref{E5} defines a mapping 
\[
\mathcal{T}: \vv \in C([0,T]; X_n) \mapsto   \mathcal{T}[\vc{v}] = (\vu[\vr,z, R,Z] - \vu_B)  
\in C([0,T]; X_n).
\]
The first level approximate solutions $r = r_{\delta, \ep, n}$, $\vu = \vu_{\delta, \ep, n}$ -- here, $r$ stands for $\vr$, $z$, $R$, $Z$--are obtained via a fixed point 
through the mapping $\mathcal{T}$. This  procedure is detailed in \cite{ChJNo} and in { \cite{KwNo}} for the mono-fluid case
with the non zero inflow-outflow and in \cite{NoSCM} for the multi-fluid case with the no-slip boundary conditions. Combinnig
\cite[Section 4]{KwNo} with \cite[Section 4]{NoSCM}, 
we easily deduce the following result.\footnote{ The energy inequality (\ref{E7}) in \cite[Lemma 4.2]{KwNo} and in \cite[Section 4]{NoSCM} 
is derived under assumption $\Omega\in C^2$. This assumption is needed due to the treatment of the parabolic problem (\ref{E1}--\ref{E3}) via
the classical maximal regularity methods. With Lemma \ref{EL1} at hand, the same proof can be carried out
without modifications also in Lipschitz domains.}

\begin{prop}{\rm [Approximate solutions, level I]} \label{EP1}
Let $\Omega \subset \R^3$ be a bounded Lipschitz domain.  
Let the data $(\vr_B,z_B,R_B,Z_B,\vu_B)$, $(\vr_0,z_0,R_0,Z_0,\vu_0)$  belong to the class { (\ref{ruB}--\ref{calO})},
(\ref{initr}). Suppose that assumptions (\ref{regP}--\ref{convH}) are satisfied.

Then for each fixed $\delta > 0$, $\ep > 0$, $n > 0$, there exists a solution 
$(\vr_n,z_n,R_n,Z_n,\vu_n=\vv_n+\vu_B)$ of the approximate problem (\ref{E4}) and \eqref{E5}, \eqref{E6}. 
Moreover, the following holds:
\begin{enumerate}
\item Lower and upper bounds of "densities":
{
\begin{equation}\label{eq5.1} 
\begin{array}{c}
\forall t\in\overline I,\;
R_n(t,x)\geq \underline c(\delta, n)>0,\;Z_n(t,x)\geq \underline c(\delta, n)>0, \; \underline a R_n(t,x) \leq Z_n(t,x) \leq \overline a R_n(t,x),\\
\underline F R_n(t,x)\le
\vr_n(t,x)\le\overline F R_n(t,x),\; \underline G R_n(t,x)\le
z_n(t,x)\le \overline F Z_n(t,x)\;\mbox{ for a.a. $x\in\Omega$,}
\\ \\
\mbox{for a.a. $t\in I$},\;
R_n(t,x)\geq \underline c(\delta, n)>0,\;Z_n(t,x)\geq \underline c(\delta, n)>0, \; \underline a R_n(t,x) \leq Z_n(t,x) \leq \overline a R_n(t,x),\\
\underline F R_n(t,x)\le
\vr_n(t,x)\le\overline F R_n(t,x),\; \underline G R_n(t,x)\le
z_n(t,x)\le \overline F Z_n(t,x)\;\mbox{ for a.a. $x\in\partial\Omega$.}
\end{array}
\end{equation}
}
\item The approximate energy inequality 
\begin{equation} \label{E7}
\begin{split}
&\left[ \intO{\left[ \frac{1}{2} (\vr_n+z_n)|\vv_n|^2 + { H}_\delta(R_n,Z_n)  \right] } \right]_{t = 0}^{ t = \tau} + 
\int_0^\tau \intO{\mathbb{S}(\Grad\vu_n):\Grad\vu_n } \dt 
\\ 
&+\int_0^\tau \int_{\Gamma^{\rm out}} { H}_\delta(R_n,Z_n)  \vu_B \cdot \vc{n} \ {\rm d} S_x \dt
- \int_0^\tau \int_{\Gamma^{\rm in}} E_{{ H}_\delta}(R_B,Z_B|R_n,Z_n) \vu_B \cdot \vc{n} 
\ {\rm d}S_x \ \dt \\
&+
\ep \int_0^\tau \intO{ { \nabla_{R,Z}^2 {H}_\delta(R_n,Z_n)} [\Grad R_n,\Grad Z_n] } \dt \\
&\leq
- 
\int_0^\tau \intO{ \left[ (\vr_n+z_n) \vu_n \otimes \vu_n + P_\delta(R_n,Z_n) \mathbb{I} \right]  :  \Grad \vu_B } \dt \\
&+ \int_0^\tau { \intO{ (\vr_n+z_n) \vu_n  \cdot\Grad \vu_B \cdot  \vu_B  } }
\dt 
+ \int_0^\tau \intO{ \mathbb{S}(\Grad\vu_n) : \Grad \vu_B } \dt 
\\
&- \int_0^\tau \int_{\Gamma^{\rm in}} { H}_\delta(R_B,Z_B)  \vu_B \cdot \vc{n} \ {\rm d} S_x \dt  
\end{split}
\end{equation}
holds for any $0 \leq \tau \leq T$, where 
\begin{equation}\label{Hdelta}
{H}_\delta(R,Z)= H(R,Z)
+\frac \delta{\mathfrak{c}-1}\Big(R^{\mathfrak{c}}+ Z^{\mathfrak{c}} \Big),
\end{equation}
and 
$$
\nabla_{R,Z}^2 {H}_\delta(R,Z) [\Grad R,\Grad Z]=
\partial^2_RH_\delta(R,Z)|\nabla R|^2+2\partial_{R}\partial_ZH_\delta\nabla R\cdot\nabla Z
+\partial^2_ZH_\delta(R,Z)|\nabla Z|^2.
$$
\item Renormalized identity
\begin{equation}\label{eq5.2}
\|r_n\|^2_{L^2(\Omega)}\Big|_0^\tau + \varepsilon \int_0^\tau\|\Grad r_n\|_{L^2(\Omega)}^2{\rm d}t  
-\int_0^\tau\int_{\Gamma^{\rm in}}\vu_B\cdot\vc n (r_n-r_B)^2{\rm d}S_x{\rm d}t
\end{equation}
$$
+ \int_0^\tau\int_{\Gamma^{\rm out}}r_n^2\vu_B\cdot\vc n {\rm d}S_x{\rm d}t= 
-\int_0^\tau\int_{\Gamma^{\rm in}}r_B^2\vu_B\cdot\vc n {\rm d}S_x{\rm d}t
-\int_0^\tau\int_\Omega r_n^2\Div \vu_n\dx{\rm d}t
$$
holds, where $r_n$ stands for $\vr_n,z_n,R_n,Z_n$.
\end{enumerate}
\end{prop}

This is level I of approximations (with three parameters $n$, $\ep$, $\delta$). We shall pass first to the limit $n\to\infty$ in order to obtain level II of approximations (with two parameters $\ep$, $\delta$). Then we obtain level III of approximations (with one parameter $\delta$) by letting
$\ep\to 0+$. Finally, we effectuate limit $\delta$ to $0+$ in order to obtain a weak solution of the (academic) problem (\ref{eq1.1}--\ref{eq1.3}).

\section{Limit passage from level I to level II (limit $n\to\infty$)}\label{Sn}

The goal of this section is to pass to the limit $n\to\infty$ in Proposition \ref{EP1}. The result
is formulated in Proposition \ref{EP2} at the end of the section.

\subsection{Limit in the parabolic equations and in the momentum equation (start)}

In view of (\ref{eq2.1}--\ref{convH}), (\ref{ruB}), (\ref{initr}), (\ref{regH}--\ref{grH}) and (\ref{Pdelta}), (\ref{Hdelta}) 
we deduce  from (\ref{eq5.1}--\ref{eq5.2}) the following bounds:
{
\begin{equation} \label{eq5.1+}
\begin{array}{c}
\mbox{$\forall t \in \overline I$},\; 
R_n(t,x)\geq 0, \; Z_n(t,x) \geq 0,\; \underline a R_n(t,x)\le
Z_n(t,x)\le\overline a R_n(t,x),\\
\underline F R_n(t,x)\le
\vr_n(t,x)\le\overline F R_n(t,x),\; \underline G R_n(t,x)\le
z_n(t,x)\le \overline F Z_n(t,x)\;\mbox{ for a.a. $x\in\Omega$,}
\\ \\
\mbox{for a.a. $(t,x)\in I\times\partial\Omega$},\;
R_n(t,x)\geq 0,\;Z_n(t,x)\geq 0, \; \underline a R_n(t,x) \leq Z_n(t,x) \leq \overline a R_n(t,x),\\
\underline F R_n(t,x)\le
\vr_n(t,x)\le\overline F R_n(t,x),\; \underline G R_n(t,x)\le
z_n(t,x)\le \overline F Z_n(t,x).
\end{array}
\end{equation}
and
\begin{align}
&\|(\vr_n+z_n)|\vv_n|^2\|_{L^\infty(I,L^1(\Omega ))}\aleq
 C(\delta),\label{ts1?}\\
&
\|\vv_n\|_{L^2(I,W^{1,2}(\Omega ))}\le C(\delta),\label{ts2?}\\
&
 \|r_n\|_{L^\infty(I,L^{\mathfrak{c}}(\Omega))}\aleq
C(\delta),\label{ts2?+}\\
&
\ep\|\nabla{r_n}\|^2_{L^2(Q_T)}  \aleq
C(\delta),\label{ts4?}\\
&
\|r_n|\vu_B\cdot\vc n|^{1/{\mathfrak{c}}} \|_{L^{\mathfrak{c}}((0,T)\times\partial\Omega)}\aleq  C(\delta),
\label{nova+}
\end{align}
where $r_n$ stands for $\vr_n$, $z_n$, $R_n$, $Z_n$, and
\begin{equation}\label{nova++}
 \ep\|\nabla(r_n^{{\mathfrak {c}/2}})\|^2_{L^2(Q_T)}\aleq C(\delta),\;\mbox{where $r_n=Z_n$ or $r_n=R_n$.}
\end{equation}

{
By  (\ref{ts2?}), 
\begin{equation}\label{ts8?}
\vv_n\rightharpoonup\vv\;\mbox{in $L^2(I;W_0^{1,2}(\Omega))$}.
\end{equation}
By virtue of (\ref{ts2?+}),} Arzela-Ascoli theorem combined with equation (\ref{E4}), estimates (to verify equi-continuity) and density of $C^\infty_c(\Omega)$ in $L^{{\mathfrak{c}}'}(\Omega)$,
$$
r_n\to r\;\mbox{in $C_{\rm weak}(\overline I;L^{\mathfrak{c}}(\Omega))$, $r$ states for $\vr,z,R,Z$.}
$$
Consequently in particular- cf. (\ref{ts1?}),
$$
(\vr_n+z_n)\vu_n\rightharpoonup_* (\vr+z) \vu\;\mbox{in $L^\infty(I;L^{\frac{2\mathfrak{c}}{\mathfrak{c }+1}}(\Omega))$},\; \vu=\vv+\vu_B.
$$
By the similar Arzela-Ascoli argument as above, but now with the momentum equation (\ref{E5}) one gets
first,
$$
(\vr_n+z_n)\vu_n\to (\vr+z)\vu\;\mbox{in $C_{\rm weak}( I;L^{\frac{2\mathfrak{c}}{\mathfrak{c }+1}}(\Omega))$}
$$
and then
$$
(\vr_n+z_n)\vu_n\otimes\vu_n\to(\vr+z) \vu\otimes\vu\;\mbox{in $L^2(\overline I;L^{\frac{6\mathfrak{c}}{4\mathfrak{c }+3}}(\Omega))$}.
$$

{
By  (\ref{ts4?}),}
$$
\Grad r_n\rightharpoonup\Grad r\;\mbox{ in $L^2(I\times\Omega)$, $r=\vr$ or $z$ or $R$ or $Z$.} 
$$
and by (\ref{nova+}), for the traces of the same $r$'s,
$$
r_n\rightharpoonup r\;\mbox{in $L^{\mathfrak{c}}(I;L^{\mathfrak{c}}(\partial\Omega;|\vu_B\cdot{\vc n}|{\rm d}S_x))$}
$$

We realize that (\ref{E4}) rewrites in the form
\begin{equation}\label{A0}
\partial_t r_n={\cal F}(r_n,\vu_n)\;\mbox{in $L^2(I;[W^{1,2}(\Omega)]^*))$},
\end{equation}
where
$$
<{\cal F}(r,\vu),\varphi>=
\int_0^T \intO{ \left[ r \vu\cdot\Grad\varphi  - \ep \Grad r \cdot \Grad \varphi \right] }
\dt
$$
$$
- \int_0^\tau \int_{\partial \Omega} \varphi r \vu_B \cdot \vc{n} \ {\rm d} S_x \dt 
+ 
\int_0^\tau \int_{\partial \Omega} \varphi (r - r_B) [\vu_B \cdot \vc{n}]^{-}  \ {\rm d}S_x \dt.
$$
Consequently, we deduce from the estimates, that
\begin{equation}\label{nova!}
\|\partial_t r_n\|_{L^{2}(I;[W^{1,2}(\Omega)]^*)}\aleq C(\ep,\delta).
\end{equation}
In view of (\ref{A0}) and (\ref{nova!}), we find in the limit
\begin{equation}\label{E4!!}
\partial_t r={\cal F}(r,\vu)\;\mbox{in $L^{2}(I;[W^{1,2}(\Omega)]^*))$}.
\end{equation}

By virtue of (\ref{ts4?}), (\ref{nova!}) and Lions-Aubin lemma, for $r=R$ or $Z$,
$$
r_n\to r\;\mbox{a.e. in $I\times\Omega$};\;\mbox{whence}\;r_n\to r\in L^q(I\times\Omega)\;\mbox{with some $q>\mathfrak{c}$},
$$
where we have used (\ref{ts2?+}) and (\ref{nova++}) and interpolation. Consequently,
$$
P_\delta(R_n,Z_n)\to P_\delta(R,Z)\;\mbox{in $L^q(I\times\Omega)$ with some $q>1$}. 
$$

With what we derived so far in this section, we are at the point to be able to pass to the limit $n\to\infty$ in the parabolic equations (\ref{E1}--\ref{E3}).
In order to pass to the limit in the momentum equation (\ref{E5}), it remains to prove strong convergence of $\Grad(\vr_n+z_n)$
and consequently,
\begin{equation}\label{epS}
\Grad(\vr_n+z_n)\cdot\Grad\vu_n\rightharpoonup \Grad(\vr+z)\cdot\Grad\vu\;\mbox{in ${\cal D}'(I\times \Omega)$}.
\end{equation}
We shall postpone this point to Section \ref{SE4.3}. Indeed, to do this, we shall need to derive first the renormalized equation (\ref{eq5.2!!}) 
for the limiting couple $(r,\vu)$.\footnote{{ In \cite{ChJNo}, the authors obtained the strong convergence of $\Grad\Sigma_n$, $\Sigma_n=\vr_n+z_n$ directly by deriving the $L^{4/3}$-bound for the  second derivatives 
of $\Sigma_n$ via the maximal parabolic regularity -- at cost of adding the additional dissipation $\ep{\rm div}(|\Grad\vc v|^2\Grad \vc v)$ to the approximated momentum equation (\ref{E4}). This bound was used also for the derivation of the renormalized identities of type (\ref{renoep}). In addition to the non homogenous boundary data, the present approach generalizes
\cite{ChJNo} in several directions: 1) There is no need of the additional dissipation in the approximations of the momentum equation. 2)There is no need
of $C^2$ boundary in order to derive the renormalized equations.}}
}

\subsection{Limit in the energy inequality}\label{SEi}
We have,
$$
\frac 1\delta\int_\tau^{\tau+\delta}\intO{(\vr_n+z_n)\vv_n^2}{\rm d}t\to\frac 1\delta\int_\tau^{\tau+\delta}\intO{(\vr+z)\vv^2}{\rm d}t
$$
$$
\mbox{with
$\lim_{\delta\to 0+}\frac 1\delta\int_\tau^{\tau+\delta}\intO{(\vr+z)\vv^2}{\rm d}t=\intO{(\vr+z)\vv^2(\tau)}$ for a.a. $\tau\in (0,T)$}
$$
where the latter conclusion is a consequence of  the theorem on Lebesgue points.

We further use lower weak semi-continuity of convex functionals, as follows:
$$
\forall\tau\in\overline I,\;
\intO{H_\delta(R,Z)(\tau)}\le\liminf_{n\to\infty} \intO{H_\delta(R_n,Z_n)(\tau)},
$$
$$
 \int_0^\tau\int_{\Gamma^{\rm out}}H_\delta(R,Z)\vu_B\cdot\vc n
{\rm d} S_x{\rm d}t
\le\liminf_{n\to\infty} \int_0^\tau\int_{\Gamma^{\rm out}}H_\delta(R_n,Z_n) \vu_B\cdot\vc n
{\rm d} S_x{\rm d}t,
$$
{ 
$$
\int_0^\tau\intO{|\Grad\vv|^2}{\rm d}t\le \liminf_{n\to\infty} \int_0^\tau\intO{|\Grad\vv_n|^2}{\rm d}t.
$$
}

At this point, we can deduce from the energy inequality (\ref{E7}) the energy inequality (\ref{E7!}) in Proposition \ref{EP2}.

{ \subsection{Renormalization}\label{SE4.3}}
In the sequel, we shall work with equation (\ref{E4!!}). To begin, we take $B\in C^2_c((0,\infty))$.
As $\partial_t r \in L^{2}(0,T; [W^{1,2}(\Omega)]^*)$,  $B'(r) \in L^2(0,T; W^{1,2}(\Omega))$ 
we deduce from this formulation 
\begin{equation}\label{renc}
\begin{split}
\int_{\tau_1}^{\tau_2} &\left< \partial_t r, B'(r)\varphi \right>_{[W^{1,2}]^*; W^{1,2}} \dt\\ 
&=  \int_{\tau_1}^{\tau_2} \intO{ r \vu\cdot\Grad( B'(r)\varphi) } \dt - 
\ep \int_{\tau_1}^{\tau_2} \intO{\Big[ |\Grad r|^2  B''(r)\varphi+B'(r)\Grad r\cdot\Grad\varphi\Big] } \dt  
\\ & 
-\int_{\tau_1}^{\tau_2} \int_{\partial \Omega}\varphi r B'(r)\vu_B\cdot\vc n {\rm d} S_x \ \dt
+ \int_{\tau_1}^{\tau_2} \int_{\partial \Omega} \varphi B'(r) (r - r_B) [\vu_B \cdot \vc{n}]^- \ {\rm d} S_x \ \dt,
\end{split}
\end{equation}
$0<\tau_1<\tau_2<T$, where $\varphi\in C^1(\overline\Omega)$.

Using the standard temporal regularization via a family of $t-$dependent convolution kernels, we find a sequence of functions 
{
$$
r_n \in C^1([\tau_1, \tau_2]; W^{1,2}(\Omega)), \
r_n \to r \ \mbox{in}\ L^2(\tau_1, \tau_2; W^{1,2}(\Omega)),\ 
 \partial_t r_n \to \partial_t r  
\ \mbox{in}\ L^2(\tau_1, \tau_2; [W^{1,2}(\Omega)]^*) 
$$
}
for any $0 < \tau_1 < \tau_2$, and as $r \in C([0,T]; L^p(\Omega))$, $1\le p<\mathfrak{c}$, 
$$
B(r_n)(\tau) \to B(r(\tau)) \ \mbox{in $L^p(\Omega)$ for any}\ \tau \in (0,T).
$$
Thus, we obtain, 
$$
\intO{ B(r)\varphi } \Big|_{\tau_1}^{\tau_2} = 
\lim_{n \to \infty} \left[ \intO{ B(r_n)\varphi } \right]_{t = \tau_1}^{t = \tau_2}
= \lim_{n \to \infty} \int_{\tau_1}^{\tau_2} \intO{ B'(r_n) \partial_t r_n \varphi} \dt 
= \int_{\tau_1}^{\tau_2} \left< { \partial_t r}, B'(r)\varphi \right> \dt
$$
and
$$
\int_{\tau_1}^{\tau_2} \intO{r\vu\cdot\Grad(B'(r)\varphi)}\dt=
\lim_{n\to 0} \int_{\tau_1}^{\tau_2} \intO{r_n\vu\cdot\Grad(B'(r_n)\varphi)}\dt
$$
$$
{ =
-\lim_{n\to 0} \int_{\tau_1}^{\tau_2} \intO{\Big[\varphi\Div(B(r_n)\vu)+(r_nB'(r_n)-B(r_n))\Div\vu\Big]}\dt +\int_{\tau_1}^{\tau_2} \int_{\partial\Omega}\varphi r_nB'(r_n)\vu_B\cdot\vc n
{\rm d}S_x\dt
}
$$
$$
=
{ \lim_{n\to 0}\Big( \int_{\tau_1}^{\tau_2} \intO{\Big[B(r_n)\vu\cdot\Grad\varphi-\varphi(r_nB'(r_n)-B(r_n))\Div\vu\Big]}\dt }
$$
{ 
$$
+\int_{\tau_1}^{\tau_2} \int_{\partial\Omega}\varphi\Big[-B(r_n)+r_nB'(r_n)\Big]\vu_B\cdot\vc n
{\rm d}S_x\dt
$$
$$
=
 \int_{\tau_1}^{\tau_2} \intO{\Big[B(r)\vu\cdot\Grad\varphi-\varphi (rB'(r)-B(r))\Div\vu\Big]}\dt 
+\int_{\tau_1}^{\tau_2} \int_{\partial\Omega}\varphi\Big[-B(r)+rB'(r)\Big]\vu_B\cdot\vc n
{\rm d}S_x\dt
$$
}
for any $0 < \tau_1 < \tau_2 < T$. 

Inserting the last two formulas to (\ref{renc}) yields,
$$
{ \intO{B(r)\varphi}\Big|_0^\tau=}
\int_0^\tau\intO{\Big[\Big(B(r)\vu-\ep B'(r)\Grad r\Big)\cdot
\Grad\varphi -\varphi\Big(\ep B''(r)|\Grad r|^2+(r B'(r)-B(r)){\rm div}\vu\Big)\Big]}{\rm d}t
$$
\begin{equation}\label{renoep!}
{ +}\int_0^\tau\int_{\partial\Omega}\varphi[\vu_B\cdot\vc n]^-
\Big(B(r_B)-B'(r)(r_B-r)- B(r)\Big){\rm d}S_x{\rm d}t 
\end{equation}
$$
-\int_0^\tau\int_{\partial\Omega}
\varphi[\vu_B\cdot\vc n]^- B(r_B){\rm d}S_x{\rm d}t-\int_0^\tau\int_{\partial\Omega}
\varphi[\vu_B\cdot\vc n]^+ B(r){\rm d}S_x{\rm d}t.
$$
Identity (\ref{renoep!}) can be extended to $B\in C^2([0,\infty))$ with growth 
\begin{equation}\label{grB}
|B(r)|\aleq (1+r)^{\frac 12\mathfrak{c}},\;
|B'(r)|\aleq (1+r)^{\frac 12\mathfrak{c}-1},\;B''(r)\aleq 1.
\end{equation}
Under these conditions, and if moreover $B$ is convex on $[0,\infty)$,
we deduce
\begin{equation}\label{renoepc}
\intO{B(r)\varphi}\Big|_0^\tau
+ \int_0^\tau\int_{\partial\Omega}
\varphi[\vu_B\cdot\vc n]^+ B(r){\rm d}S_x{\rm d}t
\end{equation}
$$
\le
\int_0^\tau\intO{\Big(B(r)\vu\cdot\Grad\varphi-\varphi(r B'(r)-B(r)){\rm div}\vu\Big)}{\rm d}t 
$$
$$
-\int_0^\tau\int_{\partial\Omega}
\varphi[\vu_B\cdot\vc n]^- B(r_B){\rm d}S_x{\rm d}t
- \ep\int_0^\tau\intO{ B'(r)\Grad r\cdot
\Grad\varphi}{\rm d }t
$$
with any $\tau\in [0,T]$ and any non negative $\varphi\in C^1(\overline\Omega)$.

Inequality (\ref{renoepc}) is generalizable to renormalizing functions of several variables and several parabolic equations (\ref{E1}--\ref{E3}) with the same $\vu$.
In particular, if $B\in C^2([0,\infty)^2)$ is convex  with the  growth
\begin{equation}\label{grB+}
|B(R,Z)|\aleq (1+\sqrt{R^2+Z^2})^{\frac 12\mathfrak{c}},\;
|\nabla_{R,Z}B(R,Z)|\aleq (1+\sqrt{R^2+Z^2})^{\frac 12\mathfrak{c}-1},\;|\nabla^2_{R,Z}(R,Z)|\aleq 1,
\end{equation}
then
\begin{equation}\label{reno+}
\intO{B(R,Z)}\Big|_0^\tau + \int_0^\tau\int_{\partial\Omega}
\varphi[\vu_B\cdot\vc n]^+ B(R,Z){\rm d}S_x{\rm d}t
\end{equation}
$$
\le
\int_0^\tau\intO{\Big[B(R,Z)\vu\cdot\nabla\varphi- \varphi{\Div\vu}\Big(R\partial_R B+Z\partial_Z B-B\Big)(R,Z)\Big]}{\rm d}t
$$
$$
- \int_0^\tau\int_{\partial\Omega}[\vu_B\cdot\vc n]^- B(R_B,Z_B){\rm d}S_x{\rm d}t
-\ep\int_0^\tau\intO{\Big(\partial_R B(R,Z)\Grad R + \partial_Z B(R,Z)\Grad Z\Big)\cdot\Grad\varphi}{\rm d}t
$$
with any  $\tau\in\overline I$ and with any non negative $\varphi \in C^2(\overline\Omega)$.

We shall use (\ref{renoep!}--\ref{reno+}) in the following situations
\begin{enumerate}
\item If $B(r)=r^2$ and $\varphi=1$, we obtain the renoramlized identity
(\ref{eq5.2!!}).
\item If $B(r)=r\log (r+\mathfrak{a})$, $\mathfrak{a}>0$ we obtain the
renormalized inequality (\ref{renoep}). 

\item If $B(R,Z)=\frac {Z^2}{R+\mathfrak{a}}$,
$\mathfrak{a}>0$ in (\ref{reno+}), we obtain (\ref{renofrac}) after letting $\mathfrak{a}\to 0+$.
\end{enumerate}

\subsection{Limit in the momentum equation (end)}

{
Combining (\ref{eq5.2}) { with} (\ref{eq5.2!!}) together with the lower weak semicontinuity of norms, we deduce
$$
\|r(\tau)\|^2_{L^2(\Omega)} + 2\varepsilon \int_0^\tau \|\Grad r\|_{L^2(\Omega)}^2{\rm d}t  
-\int_0^\tau\int_{\Gamma^{\rm in}}\vu_B\cdot\vc n (r-r_B)^2{\rm d}S_x{\rm d}t
+ \int_0^\tau\int_{\Gamma^{\rm out}}r^2\vu_B\cdot\vc n {\rm d}S_x{\rm d}t
$$
$$
\le\liminf_{n\to\infty}\Big(
\|r_n(\tau)\|^2_{L^2(\Omega)} + 2\varepsilon \int_0^\tau \|\Grad r_n\|_{L^2(\Omega)}^2{\rm d}t  
-\int_0^\tau\int_{\Gamma^{\rm in}}\vu_B\cdot\vc n (r-r_B)^2{\rm d}S_x{\rm d}t
+ \int_0^\tau\int_{\Gamma^{\rm out}}r^2\vu_B\cdot\vc n {\rm d}S_x{\rm d}t\Big)
$$
$$
\le
\|r(\tau)\|^2_{L^2(\Omega)} + 2\varepsilon \int_0^\tau \|\Grad r\|_{L^2(\Omega)}^2{\rm d}t  
-\int_0^\tau\int_{\Gamma^{\rm in}}\vu_B\cdot\vc n (r-r_B)^2{\rm d}S_x{\rm d}t
+ \int_0^\tau\int_{\Gamma^{\rm out}}r^2\vu_B\cdot\vc n {\rm d}S_x{\rm d}t,
$$
where $r$ stands for $\vr$, $z$, $R$, $Z$, and any of its linear combinations; 
whence, in particular,
$$
\Grad(\vr_n+z_n)\to\Grad(\vr+z)\;\mbox{in $L^2(I\times\Omega)$}.
$$
This in combination with (\ref{ts8?}) justifies convergence (\ref{epS}) and ends the proof of the convergence
in the momentum equation, yielding (\ref{E5!}).
}

\subsection{Summary of limit passage from level I to level II}

We resume the results of Section \ref{Sn}.

\begin{prop}{\rm [Approximate solutions, level II]} \label{EP2}
Let $\Omega \subset \R^3$ be a bounded Lipschitz domain.  
Let the data $(\vr_B,z_B,R_B,Z_B,\vu_B)$, $(\vr_0,z_0,R_0,Z_0,\vu_0)$  belong to the class (\ref{ruB}--\ref{calO}),
(\ref{initr}). Suppose that $P$ satisfies assumptions (\ref{regP}--\ref{convH}).

Then for each fixed $\delta > 0$, $\ep > 0$,  { there exists a quintet
$(\vr,z,R,Z,\vu=\vv+\vu_B)=(\vr_\ep,z_\ep,R_\ep,Z_\ep,\vu_\ep=\vv_\ep+\vu_B,)$ 
}
with the following properties:
\begin{enumerate}
\item It belongs to the function spaces
\begin{equation}\label{fsw}
r \in C_{\rm weak} (\overline I; L^{\mathfrak{c}}(\Omega))\cap L^2(0,T; W^{1,2}(\Omega))\cap L^{\mathfrak{c}}(I;L^{\mathfrak{c}}(\partial\Omega,
|\vu_B\cdot\vc n|{\rm d}S_x)),
\end{equation}
$$
{ \vv:=\vu-\vu_B \in L^2(0,T;W_0^{1,2}(\Omega; \R^3)),}
$$
$$
(\vr+z)\vv\in C_{\rm weak}(\overline I;L^{\frac{2{\mathfrak{c}}}{\mathfrak{c}+1}}(\Omega)),\;(\vr+z)\vv^2, P(R,Z)\in L^\infty(I;L^1(\Omega)).
$$
\item It obeys the following domination  inequalities
{
\begin{equation} \label{eq5.1!}
\begin{array}{c}
\mbox{$\forall t \in \overline I$},\; R_\ep(t,x)\ge 0,\;Z_\ep(t,x)\ge 0,\;
 \underline a R_\ep(t,x)\le
Z_\ep(t,x)\le\overline a R_\ep(t,x),\\
\underline F R_\ep(t,x)\le
\vr_\ep(t,x)\le\overline F R_\ep(t,x),\; \underline G R_\ep(t,x)\le
z_\ep(t,x)\le \overline F Z_\ep(t,x)\;\mbox{ for a.a. $x\in\Omega$,}
\\ \\
\mbox{for a.a. $(t,x)\in I\times\partial\Omega$},\;
R_\ep(t,x)\geq 0,\;Z_\ep(t,x)\geq 0, \; \underline a R_\ep(t,x) \leq Z_\ep(t,x) \leq \overline a R_\ep(t,x),\\
\underline F R_\ep(t,x)\le
\vr_\ep(t,x)\le\overline F R_\ep(t,x),\; \underline G R_\ep(t,x)\le
z_\ep(t,x)\le \overline F Z_\ep(t,x).
\end{array}
\end{equation}
}
\item Each of 
\begin{equation}\label{E4!+}
\mbox{
$\vr_\ep, z_\ep, Z_\ep, R_\ep$ satisfies the weak formulations 
(\ref{E4}) and (\ref{E4!!})
of parabolic equation with $\vu=\vu_\ep$.
}
\end{equation}
\item { The momentum equation 
\begin{equation} \label{E5!}
 \intO{ (\vr_\ep+z_\ep) \vu_\ep \cdot \bfphi } \Big|_{0}^{\tau} = 
\int_0^\tau \int_\Omega \Big[ (\vr_\ep+z_\ep) \vu_\ep \cdot \partial_t \bfphi + (\vr_\ep+z_\ep) \vu_\ep \otimes \vu_\ep : \Grad \bfphi 
+ P_\delta(R_\ep,Z_\ep) \Div \bfphi
 \end{equation}
$$
- \mathbb{S}(\nabla\vu_\ep): \Grad \bfphi
-\ep\Grad (\vr_\ep+z_\ep) \cdot \Grad \vu_\ep \cdot \bfphi\Big] {\rm d}x \dt
$$
for any $\bfphi \in C_c^1([0,T]\times\Omega)$. 
} 

\item The  energy inequality 
\begin{equation} \label{E7!}
\begin{split}
&\intO{\left[ \frac{1}{2} (\vr_\ep+z_\ep) |\vv_\ep|^2 + { H}_\delta(R_\ep,Z_\ep)  \right] } \Big|_{ 0}^{ \tau} + 
\int_0^\tau \intO{\mathbb{S}(\Grad\vu_\ep):\Grad\vu_\ep } \dt \\
&{ +\int_0^\tau \int_{\Gamma^{\rm out}} { H}_\delta(R_\ep,Z_\ep)  \vu_B \cdot \vc{n} \ {\rm d} S_x \dt}
\leq
- 
\int_0^\tau \intO{ \left[ (\vr_\ep+z_\ep) \vu_\ep \otimes \vu_\ep + P_\delta(R_\ep,Z_\ep) \mathbb{I} \right]  :  \Grad \vu_B } \dt \\
&+ \int_0^\tau { \intO{ (\vr_\ep+z_\ep) \vu_\ep  \cdot\Grad \vu_B \cdot  \vu_B  } }
\dt 
+ \int_0^\tau \intO{ \mathbb{S}(\Grad\vu_\ep) : \Grad \vu_B } \dt 
\\
&- \int_0^\tau \int_{\Gamma^{\rm in}} { H}_\delta(R_B, Z_B)  \vu_B \cdot \vc{n} \ {\rm d} S_x \dt  
\end{split}
\end{equation}
holds for a.a. $\tau\in I$.
\item Renormalized identity
\begin{equation}\label{eq5.2!!}
\|r_\ep(\tau)\|^2_{L^2(\Omega)}-\|r_0\|^2_{L^2(\Omega)} + 2\varepsilon \int_0^\tau \|\Grad r_\ep\|_{L^2(\Omega)}^2{\rm d}t  
-\int_0^\tau\int_{\Gamma^{\rm in}}\vu_B\cdot\vc n (r_\ep-r_B)^2{\rm d}S_x{\rm d}t
\end{equation}
$$
+ \int_0^\tau\int_{\Gamma^{\rm out}}r_\ep^2\vu_B\cdot\vc n {\rm d}S_x{\rm d}t= 
-\int_0^\tau\int_{\Gamma^{\rm in}}r_B^2\vu_B\cdot\vc n {\rm d}S_x{\rm d}t
-\int_\Omega r_\ep^2\Div \vu_\ep\dx,
$$
where $r_\ep$ stands for any linear combination of $\vr_\ep,z_\ep,R_\ep,Z_\ep$
\item Renormalized inequality
\begin{equation}\label{renoep}
\intO{r_\ep\log (r_\ep+\mathfrak{a})(\tau,x)\varphi(x)}
-\intO{r_0\log (r_0+\mathfrak{a})(x)\varphi(x)}
+ \int_0^\tau\int_{\Gamma^{\rm out}}\varphi r_\ep\log (r_\ep+\mathfrak{a})\vu_B\cdot\vc n\varphi {\rm d}S_x{\rm d}t
\end{equation}
$$
\le
\int_0^\tau\intO{\Big(r_\ep\log (r_\ep+\mathfrak{a})\vu_\ep\cdot\Grad\varphi-\varphi \frac{ r^2_\ep}{r_\ep+\mathfrak{a}}{\rm div}\vu_\ep \Big)}{\rm d}t
$$
$$
-\int_0^\tau\int_{\Gamma^{\rm in}} \varphi r_B\log( r_B+\mathfrak{a})\vu_B\cdot\vc n\varphi {\rm d}S_x{\rm d}t
-\ep
\int_0^\tau\intO{\Big(\log(r_\ep+\mathfrak{a})+\frac {r_\ep}{r_\ep+\mathfrak{a}}\Big)
\Grad r_\ep\cdot\Grad\varphi}{\rm d}t
$$
holds for all $\tau\in \overline I$, $\mathfrak{a}>0$, and all
$0\le \varphi\in C^1(\overline\Omega)$, where $r_\ep$ stands for $\vr_\ep,z_\ep, R_\ep, Z_\ep$.
\item Renormalized inequality
\begin{equation}\label{renofrac}
\intO{R_\ep\mathfrak{s}^2_{\ep,\mathfrak{d}}(\tau,x)\varphi(\tau,x)}
+ \int_0^\tau\int_{\Gamma^{\rm out}}R_\ep s^2_{\ep,\mathfrak{d}}\vu_B\cdot\vc n \varphi
{\rm d}S_x{\rm d}t
\end{equation}
$$
\le\intO{R_0(x)\mathfrak{s}^2_{\mathfrak{d}}(0,x)\varphi(x)}-\int_0^\tau\int_{\Gamma^{\rm in}}R_B\mathfrak{s}^2_B
\vu_B\cdot\vc n \varphi{\rm d} S_x{\rm d}t
$$
$$
+\int_0^\tau\intO{R_\ep\mathfrak{s}_{\ep,\mathfrak{d}}\vu_\ep\cdot\Grad\varphi}
{\rm d}t -\ep \int_0^\tau\intO{\Big(2\mathfrak{s}_{\ep,\mathfrak{d}}\Grad Z_\ep-\mathfrak{s}^2_{\ep,\mathfrak{d}}\Grad R_\ep\Big)\cdot\Grad\varphi}
{\rm d}t
$$
holds for all $\tau\in\overline I$ with any non negative $\varphi\in C^1(\overline\Omega))$ and any $\mathfrak{d}\in \R$, where $\mathfrak{s}_B(x)= \frac{Z_B(x)}{R_B(x)}$ and $\mathfrak{s}_{\mathfrak{d}}$ is defined in (\ref{renofrac+}).
\end{enumerate}
\end{prop}

\section{Limit from level II to level III { (limit $\ep\to 0$)}}
\label{Sep}
We deduce from (\ref{eq5.1!}), (\ref{E7!}) and (\ref{eq5.2!!}) readily the following uniform bounds with respect to $\ep\in (0,1)$:
\begin{equation}\label{ts1}
\|\vv_\ep\|_{L^2(I,W^{1,2}(\Omega ))}\aleq C(\delta),\;
\|(\vr_\ep+z_\ep)|\vu_\epsilon|^2\|_{L^\infty(I,L^1(\Omega ))}\aleq
 C(\delta),
\end{equation}
\begin{equation}\label{ts2}
\|r_\ep\|_{L^\infty(I,L^{\mathfrak{c}}(\Omega))}\le
C(\delta),\; \|r_\ep|\vu_B\cdot\vc n|^{1/{\mathfrak{c}}}\|_{L^{\mathfrak{c}}(I,L^{\mathfrak{c}}(\Omega))}\le
C(\delta),
\end{equation}
and $\ep$-dependent bounds
{
\begin{equation}\label{ts3}
\sqrt \ep \|\nabla r_\epsilon\|_{L^2( Q_T)}\aleq
C(\delta).
\end{equation}
}
In the above, $r$ stands for $\vr, z, Z, R$.

{ Estimate (\ref{ts2}) yields the pressure bounded solely in $L^\infty(I;L^1(\Omega))$. we can improve this estimate by taking
{
\begin{equation}\label{tsb}
\phi(t,x) = \psi(t)\mathcal{B} \left[\eta R(t) - \frac{1}{|\Omega|} \intO{\eta R(t)} \right](x),\ 0\le\psi\in C^\infty_c((I)),\ 0\le\eta\in  C^1_c(\Omega)
\end{equation}
}
as test function
the  momentum equation (\ref{E5!}), where ${\cal B}$ is the Bogovskii operator, cf. Lemma \ref{LBog},
\begin{equation}\label{Z4}
\int_0^T \psi\intO{ P_\delta(R,Z) R \eta}\dt = \int_0^T \Big[\psi\frac{1}{|\Omega|} \intO{ R\eta } \intO{ P_\delta(R,Z) }\Big]{ \dt} 
\end{equation}
$$
 -  \int_0^T \psi\intO{  (\vr(t)+z(t)){\vu}( t ) \cdot \mathcal{B} \Big[\partial_t \Big(\eta R(t)- \frac{1}{|\Omega|} \intO{\eta R(t)}\Big) \Big] } \ \dt  
$$
$$
- \int_0^T \psi'\intO{(\vr+z){\vu}( t ) \cdot \mathcal{B} \Big[\eta R(t)- \frac{1}{|\Omega|} \intO{\eta R(t)}\Big) \Big] } \ \dt
$$
$$
- \int_0^T \psi\intO{ ((\vr+z) {\vu} \otimes \vu) : \Grad  \mathcal{B} \left[ R - \frac{1}{|\Omega|} \intO{ R } \right]  } \ \dt 
$$
$$
+ \int_0^T \psi\intO{  \mathbb{S}(\Grad \vu) : \Grad \ \mathcal{B} \left[ \eta R- \frac{1}{|\Omega|} \intO{\eta R} \right]} \dt.
$$ 
}
{
In the above,\footnote{{ Strictly speaking
we should take in (\ref{tsb}) the time (or space) mollification
of $\eta R$, in order to be able to justify the indentity
$\partial_t {\mathcal B} \Big[\Big(\eta R(t)- \frac{1}{|\Omega|} \intO{\eta R(t)}\Big) \Big]= {\cal B}\Big[\partial_t \Big(\eta R(t)- \frac{1}{|\Omega|} \intO{\eta R(t)}\Big) \Big]
$
in $L^p(Q_T)$. This is however standard, see \cite[Section 2.2.5]{FeNoB}.}}
$$
\partial_t( R\eta)={\mathcal F}_\eta\in L^2(I; [W^{1,2}(\Omega)]^*)\ \;
<{\cal F}_\eta;\varphi>=<{\cal F}(R,\vu),\phi\eta>,\ \mbox{cf. (\ref{A0})}.
$$
}
Seeing (\ref{ts1}-\ref{ts3}), we easily verify the bounds { (chosing properly the sequence of $\psi$'s),}
$$
\|\partial_t R_\ep\|_{L^2(I; [W^{1,2}(\Omega)]^*))}\aleq C(\delta,
{\rm supp}\eta)
),\; \mbox{whence}\; \|\mathcal{B} [\partial_t{ R_\ep(t) }] \|_{L^2(Q_T)}
\aleq C(\delta, {\rm supp\eta}),
$$
where we have used (\ref{Z3}) { in Lemma \ref{LBog}}. In view of the last estimate, estimates (\ref{ts1}--\ref{ts3}) and in view of Lemma \ref{LBog},
we easily verify that the right hand side of the identity (\ref{Z4}) is bounded.

Seeing the structural hypotheses (\ref{grP})
we deduce from the above and the domination inequalities (\ref{eq5.1!}) that,
\begin{equation}\label{Bogep}
\| r_\ep\|_{L^{\mathfrak{c}+1}(I\times K)}\aleq C(\delta,K),\;r\;\mbox{stands for}\;\vr, z, R, Z.
\end{equation}
and 
\begin{equation}\label{PBogep}
\| P_\delta(R_\ep,Z_\ep)\|_{L^{\frac{\mathfrak{c}+1}{\mathfrak c}}(I\times K)}\aleq C(\delta,K)
\end{equation}
with any compact $K\subset\Omega$.

{
In absence of improved estimates of the pressure up to the boundary, we shall need at least its  equi-integrability near the boundary, in order
to pass to the limit  in the energy inequality in the term containing $P(R_\ep, Z_\ep){\rm div}\vu_B$. To this end we recall
the following lemma (see \cite[Lemma 6.1]{KwNo}),
\begin{lem}\label{localp}
Let $\Omega$ be a bounded Lipschitz domain. Denote by
$U_{-,h}:=\{x\in\Omega|{\rm dist}(x,\R^3\setminus\Omega\}$ its
inner neighborhood.
Consider
a sequence
$
(p_\ep,\vc z_\ep,\vc F_\ep,\tn G_\ep)_{\ep>0}
$
of functions  which satisfy equation
\begin{equation}\label{distr}
\partial_t\vc z_\ep+\vc F_\ep+{\rm div}\tn G_\ep +\Grad p_\ep=0\;\mbox{in ${\cal D}'(Q_T;\R^d)$}.
\end{equation}
Suppose finally that 
$$
(z_\ep,\vc F_\ep,\tn G_\ep,p_\ep)_{\ep>0}\; \mbox{is bounded in $L^\infty(0,T;L^\alpha(\Omega))\times L^\kappa(Q_T)\times L^\kappa(Q_T)\times L^1(Q_T)$ by $k>0$}
$$
uniformly with respect to $\ep$.

Then there exists $h_0>0$ and $c=c(k,T,\Omega)>0$ such that
\begin{equation}\label{localpe}
\int_0^T\int_{U-,h}{ p_\ep}{\rm d}x{\rm d}t\le ch^{\Gamma},\;\Gamma=\min\{1/\alpha', 1/\kappa'\},
\end{equation}
uniformly with respect to $\ep$. 
\end{lem}

We apply this lemma to the momentum equation (\ref{E5!}) and obtain, \footnote{ Alternatively, this estimate can be achieved directly by testing the momentum equation \eqref{distr} using test function $\psi{\cal B}\Big(1_{\Omega\setminus U_{-,h}}(x)-\frac{|\Omega|-|U_{-,h}|}{|\Omega|}\Big)$ and estimates \eqref{ts1}--\eqref{ts3},
in particular, the bound of $P(R_\ep,Z_\ep)$ in $L^\infty(I;L^1(\Omega))$.}
\begin{equation}\label{Best}
\| P_\delta(R_\ep,Z_\ep)\|_{L^{1}(I\times U_{-,h})}\aleq h^\Gamma,
\ \mbox{with some $\Gamma>0$} 
\end{equation} 
uniformly with respect to $\ep$.
}

The goal now is to pass to the limit $\ep\to 0+$ in Proposition \ref{EP2} and to get
the following result.

\begin{prop}{\rm [Approximate solutions, level III]} \label{EP3}

Let $\Omega \subset \R^3$ be a bounded Lipschitz domain satisfying (\ref{Om}).  
Let the data $\vr_B,z_B,R_B,Z_B,\vu_B)$, $(\vr_0,z_0,R_0,Z_0,\vu_0)$  belong to the class { (\ref{ruB}--\ref{calO})},
(\ref{initr}). Suppose that $P$ satisfies assumptions (\ref{regP}--\ref{convH}).

Then for each fixed $\delta > 0$  there exists a quinted
$(\vr,z,R,Z,\vu=\vv+\vu_B)=(\vr_\delta,z_\delta,R_\delta,Z_\delta,\vu_\delta=\vv_\delta+\vu_B)$ 
with the following properties:
\begin{enumerate}
\item It belongs to the function spaces
\begin{equation}\label{fsw+}
r \in C_{\rm weak} (\overline I; L^{\mathfrak{c}}(\Omega))\cap L^{\mathfrak{c}}(I;L^{\mathfrak{c}}(\Gamma^{\rm out},
|\vu_B\cdot\vc n|{\rm d}S_x)),
\end{equation}
$$
\vv:=\vu-\vu_B \in L^2(0,T;W_0^{1,2}(\Omega; \R^3)),\;
$$
$$
(\vr+z)\vv\in C_{\rm weak}(I;L^{\frac{2{\mathfrak{c}}}{\mathfrak{c}+1}}(\Omega)),\;(\vr+z)\vv^2, P(R,Z)\in L^\infty(I;L^1(\Omega)).
$$
\item Domination  inequalities:
\begin{equation} \label{eq5.1!++} 
{ \mbox{$\vr_\delta,z_\delta,R_\delta,Z_\delta$ satisfy domination inequalities (\ref{eq5.1!})}.}
\end{equation}
\item Each of 
\begin{equation}\label{E4!++}
\mbox{
$\vr_\delta, z_\delta, Z_\delta, R_\delta$ satisfies the weak formulation 
of the continuity equation  (\ref{eq2.7}) with $\vu=\vu_\delta$.
}
\end{equation}
\item The momentum equation 
\begin{equation} \label{E5!+}
 \intO{ (\vr_\delta+z_\delta) \vu_\delta \cdot \bfphi } \Big|_{0}^{\tau} 
 \end{equation}
 $$
= 
\int_0^\tau \intO{ \Big[ (\vr_\delta+z_\delta) \vu_\delta \cdot \partial_t \bfphi + (\vr_\delta+z_\delta) \vu_\delta \otimes \vu_\delta : \Grad \bfphi 
+ P_\delta(R_\delta,Z_\delta) \Div \bfphi - \mathbb{S}(\nabla\vu_\delta): \Grad \bfphi \Big] }{\rm d}t
$$
holds with any $\tau\in\overline I$, and any $\bfphi \in C^1([0,T]\times\Omega)$.
\item The  energy inequality 
\begin{equation} \label{E7+}
\begin{split}
&\intO{\Big( \frac{1}{2} (\vr_\delta+z_\delta) |\vv_\delta|^2 + { H}_\delta(R_\delta,Z_\delta)  \Big) } \Big|_{0}^{\tau} + 
\int_0^\tau \intO{\mathbb{S}(\Grad\vu_\delta):\Grad\vu_\delta } \dt + \\ 
&+\int_0^\tau \int_{\Gamma^{\rm out}} { H}_\delta(R_\delta,Z_\delta)  \vu_B \cdot \vc{n} \ {\rm d} S_x \dt
- 
\int_0^\tau \intO{ \left[ (\vr_\delta+z_\delta) \vu_\delta \otimes \vu_\delta + P_\delta(R_\delta,Z_\delta) \mathbb{I} \right]  :  \Grad \vu_B } \dt \\
&+ \int_0^\tau { \intO{ ({\vr}_\delta +z_\delta)\vu_\delta  \cdot\Grad \vu_B \cdot  \vu_B  } }
\dt 
+ \int_0^\tau \intO{ \mathbb{S}(\Grad\vu_\delta) : \Grad \vu_B } \dt 
\\
&- \int_0^\tau \int_{\Gamma^{\rm in}} { H}_\delta(R_B, Z_B)  \vu_B \cdot \vc{n} \ {\rm d} S_x \dt. 
\end{split}
\end{equation}
\item Renormalized equations
\begin{equation}\label{renoep+}
\intO{B(r_\delta(\tau))\varphi(x)}-\intO{B(r_0)\varphi(x)} + {\int_0^\tau\int_{\Gamma^{\rm out}} B(r_\delta)\vu_B\cdot\vc n\varphi {\rm d}S_x{\rm d}t}
\end{equation}
$$
=
\int_0^\tau\intO{\Big( B(r_\delta)\vu_\delta\cdot\Grad\varphi-{\cal Y}(r_\delta){\rm div}\vu\varphi\Big)}{\rm d}t
-\int_0^\tau\int_{\Gamma^{\rm in}} B(r_B)\vu_B\cdot\vc n\varphi {\rm d}S_x{\rm d}t,
$$
hold for all $\tau\in \overline I$ and all
$ \varphi\in C^1_c(\overline\Omega)$, where $r_\delta$ stands for $\vr_\delta,z_\delta,R_\delta,Z_\delta$.
In the above $(B,{\cal Y})$ are, in particular, the following couples
$$
\Big(B(s)=s^\theta, {\cal Y}(s)=(\theta-1)s^\theta\Big),\;
\Big(B(s)=s\log s, {\cal Y}(s)=s\Big),\;
\Big(B(s)={\cal L}_k(s), {\cal Y}(s)={\cal T}_k(s)\Big),
$$
where $\theta\in (0,\gamma/2]$, $k\ge 1$ and
\begin{equation}\label{Tk}
 {\cal T}_k(s)=k{\cal T}(s/k),\; {\cal T}\in C^1([0,\infty))\;{\cal T}(s)=
 \left\{\begin{array}{c}
         s\,\mbox{if $s\in [0,1]$}\\
         2\,\mbox{if $s\ge 3$}
        \end{array}
        \right\},\;\mbox{concave}
\end{equation}
is a truncation of the map $s\mapsto s$, while
\begin{equation}\label{Lk}
 {\cal L}_k(s)=s\int_1^s\frac{{\cal T}_k(t)}{t^2}{\rm d}t,\;\mbox{in particular, $s{\cal L}_k'(s)-{\cal L}_k(s)={\cal T}_k(s)$}
\end{equation}

\item Renormalized equation
\begin{equation}\label{renofrac!+}
\intO{R_\delta(\tau,x)\mathfrak{s}^2_{\delta,\mathfrak{d}}(\tau,x)\varphi(x)}+ \int_0^\tau\int_{\Gamma^{\rm out}}R\mathfrak{s}_{\mathfrak{d}}^2
\vu_B\cdot\vc n \varphi{\rm d} S_x{\rm d}t
=\intO{R_0(x)\mathfrak{s}^2_{\mathfrak{d}}(0,x)\varphi(x)}
\end{equation}
$$
-\int_0^\tau\int_{\Gamma^{\rm in}}R_B\mathfrak{s}^2_B
\vu_B\cdot\vc n \varphi{\rm d} S_x{\rm d}t
+\int_0^\tau\intO{R_\delta\mathfrak{s}^2_{\delta,\mathfrak{d}}\vu_\delta\cdot\Grad\varphi}
{\rm d}t 
$$
holds for all $\tau\in\overline I$ with any
$\varphi\in C^1(\overline\Omega)$ and any $\mathfrak{d}\in \R$, where $\mathfrak{s}_B(x)= \frac{Z_B(x)}{R_B(x)}$ and 
$\mathfrak{s}_{\mathfrak{d}}=\mathfrak{s}_{\delta,\mathfrak{d}}$ is defined in (\ref{renofrac+}).
\end{enumerate}
\end{prop}

The remaining part of this section is devoted to the proof of Proposition \ref{EP3}.

\subsection{Weak limits}\label{Sep+}
\subsubsection{Limit in the continuity equations}\label{Sep+1}
 Now, let $\varepsilon \to 0^+$ in equations (\ref{E4!+}), cf. (\ref{E4})$_{r=r_\ep,\vu=\vu_\ep}$, and in (\ref{E5!}). We notice that the limit passage in the convective terms can be performed as in the case of the mono-fluid compressible Navier--Stokes equations. Indeed, 
\begin{equation}\label{conv0}
\vu_\ep\rightharpoonup\vu\;\mbox{weakly in $L^2(I;W^{1,2}(\Omega))$}
\end{equation}
and
\begin{equation}\label{conv0+}
r_\ep\rightharpoonup r\;\mbox{in $L^\gamma(I;L^\gamma(\Gamma^{\rm out};|\vu_B\cdot\vc n|{\rm d}S_x))$}.
\end{equation}
Further,
seeing that 
\begin{equation}\label{conv1}
r_\ep\to r \;\mbox{in $C_{\rm weak}(\overline I;L^{\mathfrak{c}}(\Omega))$ }
\end{equation}
(as one can show by means of the  Arzel\`a--Ascoli type argument - where the requested equi-continuity
hypothesis is deduced from equation (\ref{E4})$_{r=r_\ep,\vu=\vu_\ep}$ and the uniform bounds (\ref{ts1}--\ref{ts3}))-- we deduce from the compact embedding $L^\mathfrak{c}(\Omega)\hookrightarrow\hookrightarrow W^{-1,2}(\Omega)$ and from $\vu_\ep\rightharpoonup\vu$ in $L^2(I;W^{1,2}(\Omega))$ the weak-* convergence
\begin{equation}\label{ruinfty}
r_\ep\vu_\ep\rightharpoonup_* r\vu\;\mbox{in $L^\infty(I;L^{\frac{2\mathfrak{c}}{\mathfrak{c}+1}}(\Omega))$}.
\end{equation}
Recalling (\ref{ts3}), we can pass to the limit $\ep\to 0+$ in parabolic equations (\ref{E4!+}) in order to obtain,
\begin{equation}\label{contdelta}
\intO{r\varphi(\tau)} + \int_0^\tau\int_{\Gamma^{\rm out}}\varphi r\vu_B\cdot\vc n{\rm d}S_x{\rm d}t
\end{equation}
$$
=\intO{r(0)\varphi(0)}+\intO{\Big(r\partial_t\varphi+r\vu\cdot\Grad\varphi\Big)}{\rm d}t
-\int_0^\tau\int_{\Gamma^{\rm in}}\varphi r_B\vu_B\cdot\vc n{\rm d}S_x{\rm d}t.
$$
for all $\tau\in[0,T]$ and all $\varphi\in C([0,T]\times\overline\Omega)$.
This yields the statement (\ref{E4!++}) in Proposition \ref{EP3}.

Further, convergence (\ref{conv0+}--\ref{conv1}) and (\ref{eq5.1!}) { yield} domination inequalities (\ref{eq5.1!++}).

Finally (\ref{contdelta}) yields renormalized equations (\ref{renoep+}) by virtue of (\ref{P3}) in Proposition \ref{LP2}, cf. also Remark \ref{rinflow}.
Combining Item 3 of Proposition \ref{LP2} with Corollary \ref{cont-tr} and applying them to
equation (\ref{contdelta}) with $r=R$ and $r=Z$, we get, in particular,
identity (\ref{renofrac!+})

\subsubsection{Weak { limit} in the momentum equation}
If $r_\ep=\vr_\ep+z_\ep$, the convergence (\ref{ruinfty}) can be consequently improved thanks to momentum equation (\ref{E5!})
and estimates (\ref{ts1}--\ref{ts3}), (\ref{Bogep}--\ref{PBogep}) to 
\begin{equation}\label{conv2*}
(\vr_\ep+z_\ep)\vu_\ep\to (\vr+z)\vu\;\mbox{ in $C_{\rm weak}(\overline I; L^{\frac{2\mathfrak{c}}{\mathfrak{c}+1}}(\Omega;\R^3))$}
\end{equation}
 again by the Arzel\`a--Ascoli type argument. 
With this observation at hand,
employing compact embedding  $ L^{\frac{2\mathfrak{c}}{\mathfrak{c}+1}}(\Omega)\hookrightarrow\hookrightarrow W^{-1,2}(\Omega)$ and $\vu_\ep\rightharpoonup\vu$ in $L^2(I;W^{1,2}(\Omega;\R^3))$ we infer   that
$$
(\vr_\ep+z_\ep)\vu_\ep\to (\vr+z)\vu \quad \text{ in } \quad L^2(I;W^{-1,2}(\Omega;\R^3))
$$ 
and consequently 
\begin{equation}\label{conv3}
(\vr_\ep+z_\ep)\vu_\ep\otimes\vu_\ep
 \rightharpoonup (\vr+z)\vu\otimes\vu\;\mbox{weakly e.g. in $L^1(Q_T;\R^{9})$,}
\end{equation}
at least for a chosen subsequence (not relabeled). 

{ Finally,
\begin{equation}\label{conv+}
 P_\delta(R_\ep, Z_\ep)\rightharpoonup\overline{P_\delta(R, Z)}\;\mbox{in
 $L^{\frac{\mathfrak{c}+1}{\mathfrak{c}}}(Q_T)$},
\end{equation}
by virtue of (\ref{PBogep}).\footnote{Starting from now, in general, $\overline{f(t,x, r, (\ldots),\vu)}$ denotes the $L^1-$weak limit of the
sequence $f(t,x,r_n, (\ldots)_n,\vu_n)$ in $L^1(Q_T)$ (provided it exists).}
}

Resuming (and realizing that all terms  in the momentum equation (\ref{E5!}) multiplied by $\ep$ will disappear in the limit again by virtue of (\ref{ts1}--\ref{ts3})) we get
{
\begin{equation} \label{E5!++}
\intO{ (\vr+z) \vu \cdot \bfphi } \Big|_{0}^{\tau} = 
\int_0^\tau \intO{ \Big[ (\vr+z) \vu \cdot \partial_t \bfphi + (\vr+z) \vu \otimes \vu : \Grad \bfphi 
+ \overline{P_\delta(R,Z)} \Div \bfphi - \mathbb{S}(\nabla\vu): \Grad \bfphi \Big] }{\rm d}t
\end{equation}
for all $\bfphi\in C^1([0,T)\times\Omega))$.
}

\subsection{Limit in the momentum equation- continued}\label{S5.2}

\subsubsection{Exploiting the almost compactness}\label{S5.2.1}

We define quantities $\mathfrak{s}_{\ep}=Z_\ep/_{\mathfrak{d}} R_\ep$ and ${ \mathfrak{s}=}\mathfrak{s}_{\mathfrak{d}}=
Z/_{\mathfrak{d}} R$
as in (\ref{renofrac+}). From (\ref{renofrac!+}), we get, in particular,
{
\begin{equation}\label{renofrac*}
\intO{R(\tau,x)\mathfrak{s}^2_{\mathfrak{d}}}+
\int_0^\tau\int_{\Gamma^{\rm out}}R\mathfrak{s}_{\mathfrak{d}}^2
\vu_B\cdot\vc n {\rm d} S_x{\rm d}t
=\intO{R_0(x)\mathfrak{s}^2_{\mathfrak{d}}(0,x)}
-\int_0^\tau\int_{\Gamma^{\rm in}}R_B\mathfrak{s}^2_B
\vu_B\cdot\vc n {\rm d} S_x{\rm d}t
\end{equation}
for all $\tau\in\overline I$. 
}

We have,
$$
\forall \tau\in [0,T],\;
\lim_{\ep \to 0}\intO{R_\ep(\mathfrak{s}_{\ep}-\mathfrak{s})^2}=\lim_{\ep\to 0}
\intO{R_\ep\mathfrak{s}_{\ep}^2} -\intO{R\mathfrak{s}^2},
$$
where we have used the fact that
$$
\lim_{\ep\to 0}\intO{R_\ep\mathfrak{s}_\ep\mathfrak{s}(\tau)}=\lim_{\ep\to 0}\intO{Z_\ep\mathfrak{s}(\tau)}=
\intO{Z\mathfrak{s}(\tau)}= \intO{R\mathfrak{s}^2(\tau)},\;
\tau\in [0,T].
$$

Thus putting together (\ref{renofrac}) and (\ref{renofrac*}), we get
\begin{equation}\label{P2}
\lim_{\ep \to 0}\intO{R_\ep(\mathfrak{s}_{\ep}-\mathfrak{s})^2(\tau)}=0
\end{equation}
for all $\tau\in [0,T]$. 
Consequently, by interpolation, in particular,
\begin{equation}\label{almostcep}
\int_0^T\int_K{R^p_\ep(\tau)(\mathfrak{s}_\ep-\mathfrak{s})^q}{\rm d}t\to 0,\; { 0<p<\frac{\mathfrak{c}+1}{\mathfrak{c}}}, \;0<q<\infty,\
K\ \mbox{any compact in $\Omega$}.
\end{equation}

Now, we may write
$$
P_\delta(R_\varepsilon,Z_\varepsilon) = \Pi_\delta(R_\varepsilon, R_\varepsilon \mathfrak{s}_\varepsilon) 
= P_\delta (R_\varepsilon,R_\varepsilon \mathfrak{s}_\varepsilon) - P_\delta(R_\varepsilon, R_\varepsilon \mathfrak{s}) + \Pi_\delta(R_\varepsilon, t,x),\;\mbox{
$\Pi_\delta(R,t,x):=P_\delta(R,R \mathfrak{s}(t,x)).$}
$$
We have, due to (\ref{almostcep}) and hypothesis (\ref{grP}), 
$$
\lim_{\varepsilon \to 0} \Big|\int_0^T \int_K \big(P_\delta(R_\varepsilon,R_\varepsilon 
{ \mathfrak{s}_\varepsilon}) - P_\delta(\vr_\varepsilon,\vr_\varepsilon \mathfrak{s})\big)\dx\dt\Big|
$$
$$  
\leq c(\delta) \lim_{\varepsilon \to 0} \Big[\int_0^T \int_K\Big(R_\ep^{\underline\gamma}+
R_\ep^{\overline\gamma}\Big)|\mathfrak{s}_\varepsilon -\mathfrak{s}| \dx\dt  + \int_0^T \int_K R_\epsilon^{\mathfrak{c}} |\mathfrak{s}_\varepsilon^{\mathfrak{c}} -\mathfrak{s}^{\mathfrak{c}}| \dx\dt\Big]  = 0;
$$
whence
\begin{equation}\label{eq6.6}
\overline{P_\delta(R,Z)}=\overline{\Pi_\delta(R,t,x)}.
\end{equation}
Consequently, in particular,
\begin{equation} \label{eq6.7}
\int_0^T \int_\Omega \big((\vr+z)\vu \cdot \partial_t \vcg{\varphi} +  \big((\vr+z)\vu\otimes \vu\big): \Grad \vcg{\varphi} + {\Ov{\Pi_\delta (R,t,x)}} \, \Div \vcg{\varphi}\big) \dx \dt \\
= \int_0^T\int_\Omega \mathbb{S}(\Grad\vu):\Grad \vcg{\varphi}  \dx \dt, 
\end{equation}
{ where $\vcg{\varphi}\in C_c^1(I\times\Omega).$}

We may apply to the problem the theory available for the "mono-fluid" case with pressure of one variable from Feireisl et al. \cite{FNP}, slightly
modified in order to accommodate pressure laws allowing the space-time dependence, { similarly as in 
\cite[Chapter 3]{FeNoB}.}

\subsubsection{Effective viscous flux identity} \label{S5.2.2}

To this aim, we first recall the {\it effective viscous flux identity} which in our situation has the form
\begin{prop} \label{p5}
We have, possibly for a subsequence $\varepsilon \to 0^+$, the following identity
\begin{equation} \label{eq6.8}
{\Ov{\Pi_\delta (R,t,x) R}} -(2\mu+\lambda) \Ov{R \Div \vu} = {\Ov{\Pi_\delta (R,t,x)}} \, R -(2\mu+\lambda) R \, \Div \vu
\end{equation}
fulfilled a.a. in { $I\times \Omega$.} 
\end{prop}

\begin{proof}
We denote  by ${\mathfrak{R}}$ the Riesz transform
with Fourier symbol $\frac {\xi\otimes\xi}{|\xi|^2}$.
Following Lions \cite{L4}, we shall use in the approximating momentum equation (\ref{E5!}) test function
\begin{equation}\label{test1}
\varphi(t,x)=\psi(t){ \phi(x)}(\nabla\Delta^{-1}(R_\ep\phi))(t,x),\;\;\psi\in C^1_c(0,T),\;\phi\in C^1_c(\Omega)
\end{equation}
and in the limiting momentum equation (\ref{E5!++}) (resp. (\ref{eq6.7})) test function 
\begin{equation}\label{test2}
\varphi(t,x)=\psi(t){ \phi(x)}(\nabla\Delta^{-1}(R\phi))(t,x),\;\;\psi\in C^1_c(0,T),\;\phi\in C^1_c(\Omega),
\end{equation}
subtract both identities and perform the limit passage $\ep\to 0$. This is a laborious, but nowadays standard calculation (whose details, for "simple" compressible Navier--Stokes equations, can be found e.g. in 
\cite[Lemma 3.2]{FNP}, \cite[Chapter 3]{NoSt}, \cite{EF70} or \cite[Chapter 3]{FeNoB}) leading
 to the identity
\begin{equation}\label{ddd!}
\int_0^T\intO{\psi{ \phi^2}\Big({\overline{\Pi_\delta(R,t,x)}}
-(2\mu +\lambda){\rm div}\,\vu\Big)R}\,{\rm d}t
\end{equation}
$$
-\int_0^T\intO{\psi{ \phi^2}\Big({\overline{\Pi_\delta(R,t,x) R}}
-(2\mu +\lambda)\overline{R\,{\rm div}\,\vu}\Big)}\,{\rm d}t
$$
$$
=
\int_0^T\intO{\psi{ \phi}\vu\cdot\Big(R {\mathfrak{R}}\cdot((\vr+z)\vu\phi)-(\vr+z)\vu\cdot{\mathfrak{R}}(R\phi)\Big)
}\,{\rm d}t
$$
$$
-
\lim_{\ep\to 0}\int_0^T\intO{\psi{ \phi}\vu_\ep\cdot\Big(R_\ep {\mathfrak{R}}\cdot((\vr_\ep+z_\ep)\vu_\ep\phi)-(\vr_\ep+z_\ep)\vu_\ep\cdot{\mathfrak{R}}(R_\ep\phi)\Big)
}\,{\rm d}t.
$$
This process involves several integrations by parts and exploits continuity equations in form (\ref{E4!++}) and the parabolic equations (\ref{E1}--\ref{E3}) in form
 (\ref{E4!+}) in the same way as in the
mono-fluid theory.\footnote{ Due to the presence of $\phi$ in (\ref{test1}--\ref{test2}), the boundary conditions in the proof of the effective viscous flux identity do not play any role.} As in the mono-fluid theory,
{the essential observation for getting (\ref{ddd!}) is the fact that the map $R\mapsto\varphi$ defined above (cf. (\ref{test2}){ )} is a linear and continuous from
$L^p(\Omega)$ to $W^{1,p}(\Omega)$, $1<p<\infty$ as a consequence of classical   H\"ormander--Michlin's multiplier theorem of harmonic analysis, cf. Lemma \ref{LcalA}.
The most non trivial moment in this process is to show that the right-hand side of identity (\ref{ddd!}) is $0$. To see it, we repeat the reasoning \cite{FNP}
adapted to this situation.
We first realize  that the $C_{\rm weak}(I, L^{\mathfrak{c}}(\Omega))$-convergence of $(R_\ep, Z_\ep)$ and 
$C_{\rm weak}([0,T], L^{\frac{2\mathfrak{c}}{\mathfrak{c}+1}}(\Omega;\R^3))$-convergence of $(\vr_\ep+z_\ep)\vu_\ep$ evoked in (\ref{conv1}--\ref{conv2*})
imply, in particular,
\begin{equation}\label{cvep!}
\mbox{for all $t\in [0,T]$},\;(R_\ep, Z_\ep)(t)\rightharpoonup(R, Z)(t)\;\mbox{in e.g. $(L^{\mathfrak{c}}(\Omega))^2$},
\end{equation}
$$
(\vr_\ep+z_\ep)\vu_\ep(t) \rightharpoonup(\vr+z)\vu(t)\;\mbox{in 
$L^{\frac {2\mathfrak{c}}{\mathfrak{c}+1}}(\Omega;\R^3)$}.
$$
Since ${\mathfrak{R}}$ is a continuous operator from $L^p(\R^3)$ to $L^p(\R^3)$, $1<p<\infty$, we  have the same type of convergence for sequences ${\mathfrak{R}}[R_\ep(t)]$,
${\mathfrak{R}}[Z_\ep(t)]$ and ${\mathfrak{R}}[(\vr_\ep+z_\ep)\vu_\ep(t)]$ to their respective limits
 ${\mathfrak{R}}[R(t)]$,
${\mathfrak{R}}[Z(t)]$ and ${\mathfrak{R}}[(\vr+z)\vu(t)]$.

 At this stage we  apply  to the above situation Proposition \ref{rieszcom}, and get
 $$
[R_\ep {\mathfrak{R}}\cdot((\vr_\ep+z_\ep)\vu_\ep\phi)-(\vr_\ep+z_\ep)\vu_\ep\cdot{\mathfrak{R}}(R_\ep\phi)](t)\rightharpoonup
[R {\mathfrak{R}}\cdot((\vr+z)\vu\phi)-(\vr+z)\vu\cdot{\mathfrak{R}}(R\phi)](t)
$$
for all $t\in [0,T]$ (weakly) in $L^{\frac{2\mathfrak{c}}{\mathfrak{c}+3}}(\Omega)$.
In view of compact embedding $L^{\frac{2\mathfrak{c}}{\mathfrak{c}+1}}(\Omega)\hookrightarrow\hookrightarrow W^{-1,2}(\Omega)$, 
and the boundedness of $\|R_\ep {\mathfrak{R}}\cdot((\vr_\ep+z_\ep)\vu_\ep)-(\vr_\ep+z_\ep)\vu_\ep\cdot{\mathfrak{R}}(R_\ep)\|_{W^{-1,2}(\Omega)}$  in $L^\infty(I)$,
we infer, in particular,
$$
R_\ep {\mathfrak{R}}\cdot((\vr_\ep+z_\ep)\vu_\ep\phi)-(\vr_\ep+z_\ep)\vu_\ep\cdot{\mathfrak{R}}(R_\ep\phi)\to R {\mathfrak{R}}\cdot((\vr+z)\vu\phi)-(\vr+z)\vu\cdot{\mathfrak{R}}(R\phi)
$$
in $L^2(0,T;W^{-1,2}(\Omega))$.
Recalling the $L^2(I;W^{1,2}(\Omega))$-weak convergence of $\vu_\ep$ we get  (\ref{ddd!}). This completes the proof of Proposition \ref{p5}.

We realize for the further later reference, that the part of argumentation starting from (\ref{cvep!}) requires $\mathfrak{c}>9/2$. 
}
\end{proof}

\subsubsection{Strong convergence of the dominating density sequence}
\label{S5.2.3}

We are now ready to prove the strong convergence of $R_\varepsilon \to R$,
more exactly
\begin{equation}\label{a.a.ep}
R_\varepsilon \to R\;\mbox{a.a. in $I\times\Omega$},
\end{equation}
$$
R_\varepsilon \to R\;\mbox{ a.a. in $I\times\Gamma^{\rm out}$ with respect to the measure $|\vu_B\cdot\vc n|{\rm d}S_x\dt$}.
$$
We will do it again by mimicking the mono-fluid case, see e.g. \cite[Chapter 7]{NoSt}.

We already know { from Section \ref{Sep+1}} that $R$ satisfies the renormalized continuity equation (\ref{renoep+}), in particular, 
\begin{equation}\label{reno!*-}
\intO{{R\log R}(\tau,x)} + \int_0^\tau\int_{\Gamma^{\rm out}} R\log R\vu_B\cdot\vc n{\rm d}S_x{\rm d}t
\end{equation}
$$
=
\intO{R_0\log R_0(x)}
-\int_0^\tau\int_{\Gamma^{\rm in}}  R_B\log R_B\vu_B\cdot\vc n{\rm d}S_x{\rm d}t-\int_0^\tau\intO{R{\rm div}\vu}{\rm d}t
$$ 
for all $\tau\in \overline I$.

Next, we let $\ep\to 0+$ and then $\mathfrak{a}\to 0$ in (\ref{renoep}) in order to get, in particular,
\begin{equation}\label{reno!*}
\intO{\overline{R\log R}(\tau,x)} + \int_0^\tau\int_{\Gamma^{\rm out}} \overline{R\log R}\vu_B\cdot\vc n {\rm d}S_x{\rm d}t
-\int_0^\tau\intO{\overline{R{\rm div}\vu}}{\rm d}t
\end{equation}
$$
\le
\intO{R_0\log R_0(x)}
-\int_0^\tau\int_{\Gamma^{\rm in}}  R_B\log R_B\vu_B\cdot\vc n {\rm d}S_x{\rm d}t
$$ 
for all $\tau\in \overline I$.

Combining the latter inequality with the identity (\ref{reno!*-}), we
 get
$$
\intO{\Big(\overline{R\log R}-{R\log R}\Big)(\tau,x)}
+  \int_0^\tau\int_{\Gamma^{\rm out}} \Big(\overline{R\log R}-R\log R\Big)\vu_B\cdot\vc n {\rm d}S_x{\rm d}t
$$
$$
\le
\int_0^\tau\intO{\varphi\Big(R{\rm div}\vu-\overline{ R{\rm div}\vu }\Big)}{\rm d}t
$$
for all $\tau\in \overline I$.

{ Now, the crucial point is to observe that for almost all $(t,x)\in Q_T$,
$R\mapsto \Pi_\delta(R,t,x)$ is a non decreasing function, by virtue
of (\ref{dzP}--\ref{drP}), cf. (\ref{Ps1}).}
Consequently, according to Proposition
\ref{p5} and Lemma \ref{Lemma4},
$$
\intO{\Big(\overline{R\log R}-{R\log R}\Big)(\tau,x)\varphi(x)}
+  \int_0^\tau\int_{\Gamma^{\rm out}} \Big(\overline{R\log R}-R\log R\Big)\vu_B\cdot\vc n {\rm d}S_x{\rm d}t\le 0.
$$
Since $\overline{R\log R}-R\log R\ge 0$ a.e. in $Q_T$, cf. Lemma \ref{Lemma2}, this implies $\overline{R\log R}-R\log R=0$ a.e. in $Q_T$, and consequently,
$$
R_\ep\to R,\ Z_\ep \to Z\ \mbox{a.e. in $Q_T$}
$$
cf. Lemma \ref{Lemma3} (and (\ref{P2}) to see the confergence of $Z_\ep$).

This, in combination with (\ref{PBogep}) and (\ref{eq6.6}) allows to show that, in particular,
\begin{equation}\label{dod}
\overline{P_\delta(R,Z)}=P_\delta(R,Z)
\end{equation}
and to deduce from
momentum equation (\ref{E5!++}) (resp. (\ref{eq6.7})) its final form (\ref{E5!+}) at level III.

\subsection{Limit in the energy inequality}\label{S5.3}

In view of the already proved convergence relations and what was said about the passage
from (\ref{E7}) to (\ref{E7!}) in Section \ref{SEi},
seeing that one can use the lower weak semicontinuity of convex functionals at the left hand side, the limit passage $\ep \to 0+$ from the energy inequality (\ref{E7!}) to (\ref{E7+}) is at this stage rudimentary. { We shall give more details in the limit passage $\delta\to 0$, which is similar.}

\section{Limit from level III to the "academic" system { (limit $\delta\to 0$)}}
\label{Sdel}
The goal of this section is to pass to the limit $\delta\to 0+$ in
continuity equations (\ref{E4!++}), in the momentum equation (\ref{E5!+}) and in the energy inequality (\ref{E7+}). Most of the ``rough'' work was already done within the previous limit. The only issue will be to handle the lower
summability of uniform estimates for the density variables. Here, the main
ideas can be in a large extent taken over from the monofluid case (namely
from papers \cite{ChJNo} and \cite{KwNo}). We shall therefore proceed more quickly trying underline only the key points and differences.
\subsection{Estimates and weak limits}
Similarly as in Setion \ref{Sep}, we deduce from (\ref{E7!}) and (\ref{eq5.1!}) readily the following uniform bounds with respect to $\delta\in (0,1)$:
\begin{equation}\label{ts1d}
\|\vv_\delta\|_{L^2(I,W^{1,2}(\Omega ))}\aleq C,\;
\|r_\delta|\vu_\delta|^2\|_{L^\infty(I,L^1(\Omega ))}\aleq
 C,
\end{equation}
\begin{equation}\label{ts2d}
\|r_\delta\|_{L^\infty(I,L^{\mathfrak{\gamma}}(\Omega))}\le
L,\; \|r_\delta|\vu_B\cdot\vc n|^{1/{\mathfrak{\gamma}}}\|_{L^{\mathfrak{\gamma}}(I,L^{\mathfrak{\gamma}}(\Omega))}\le
C,
\end{equation}
and $\delta$-dependent bounds
{
\begin{equation}\label{ts3d}
\delta^{\frac 1{\mathfrak{c}}} \|r_\delta\|_{L^\infty(I;L^\mathfrak{c}(\Omega))}\aleq
C.
\end{equation}
}
In the above, $r$ stands for $\vr,z,Z, R$.

Taking in (\ref{E5!+}) test function 
$$
\phi(t,x)= \psi(t)\mathcal{B} \left[ \eta R^\theta(t) - \frac{1}{|\Omega|} \intO{ \eta R^\theta(t)} \right](x),\; \theta>0\,\mbox{sufficiently small,}
$$
$\psi$, $\eta$ as in (\ref{tsb}),
%
we get by the same reasoning as in (\ref{Z4}--\ref{PBogep}), (\ref{Best}), { estimates
\begin{equation}\label{Bogd}
\| r_\delta\|_{L^{\mathfrak{\gamma}+\gamma_{\rm Bog}}(I\times K)}\aleq C(K),\;r\;\mbox{stands for}\;\vr, z, R, Z.
\end{equation}
and 
\begin{equation}\label{PBogd}
\delta^{\frac 1{\mathfrak{c}+\gamma_{\rm Bog}}}\|R_\delta\|_{L^{\mathfrak{c}+\gamma_{\rm Bog}}(I\times K)}\leq L(K),\;
\| P(R_\delta,Z_\delta)\|_{L^{\frac{{\gamma}+\gamma_{\rm Bog}}{\gamma}}(I\times K)}\aleq C(K),
\end{equation}
with any compact subset $K$ of $\Omega$, provided
$$
\gamma>\frac 32,\;\mbox{where}\; \gamma_{\rm Bog}=\min\{\frac 23\gamma-1,\frac \gamma 2\}.
$$

In order to pass to the limit in the term $\int_0^\tau\intO{P_\delta(R_\delta,Z_\delta){\rm div}\vu_B}{\rm d}t$ of the energy inequality (\ref{E7+}), we however would need at least equi-integrability of $P_\delta(R_\delta,Z_\delta)$ up to the boundary, if we do not want to impose further restrictions on $\vu_B$ and $P$. In order to get such estimate, we shall test momentum equation
(\ref{E5!+}) by
$$
\phi(t,x)=\psi(t) {\cal B}\Big(1_{\Omega\setminus U_{-,h}}-\frac{|\Omega|-|U_{-,h}|}{|\Omega|}\Big)(x),\;
\mbox{$0<h<h_0$, $h_0$ sufficiently small},
$$
where the inner neighborhood $U_{-,h}$ of $\partial\Omega$
is defined in Lemma \ref{localp}.
Employing (\ref{ts1d}--\ref{ts3d}), we get in view of Lemma \ref{localp},
\begin{equation}\label{PBogd+}
\| P_\delta(R_\delta,Z_\delta)\|_{L^1(U_{-,h})}
\aleq h^{\frac{2\gamma-3}{6\gamma}}
\end{equation}
uniformly with respect to $\delta$,
provided $\gamma>3/2$, see \cite[Lemma 6.1]{KwNo} for the details.

In view of estimates (\ref{ts1d}--\ref{ts3d}), a short but detailed inspection of Section \ref{Sep+} confirms that the limits
(\ref{conv0}--\ref{ruinfty}), (\ref{conv2*}--\ref{conv3}) remain in force
also for a subsequence of $(\vr_\delta, z_\delta, R_\delta, Z_\delta,\vu_\delta)$ even if one replace in all exponents of Lebesgue spaces
$\mathfrak{c}$ by $\gamma$ (provided $\gamma>3/2$). Further, estimates
(\ref{Bogd}--\ref{PBogd+}) ensure equi-integrability of the sequence
$P_\delta(R_\delta,Z_\delta)$ in $L^1(Q_T)$. Recalling, in addition, (\ref{ts3d}), we infer
\begin{equation}\label{conv+d}
 P_\delta(R_\delta, Z_\delta)\rightharpoonup\overline{ P(R, Z)}\;\mbox{in $L^1(Q_T)$}.
\end{equation}
We can therefore pass to the limit $\delta\to 0$ in the continuity equations
(\ref{E4!++}) in order to get the final system (\ref{eq2.7}) of continuity
equations required by the  Definition \ref{d1}. Likewise, the limit
in the momentum equation (\ref{E5!+}) is
\begin{equation} \label{E5!++d}
\intO{ (\vr+z) \vu \cdot \bfphi } \Big|_{0}^{\tau} = 
\int_0^\tau \intO{ \Big[ (\vr+z) \vu \cdot \partial_t \bfphi + (\vr+z) \vu \otimes \vu : \Grad \bfphi 
+ \overline{P(R,Z)} \Div \bfphi - \mathbb{S}(\nabla\vu): \Grad \bfphi \Big] }{\rm d}t
\end{equation}
for all $\bfphi\in C_c^1([0,T)\times\Omega))$.
}

\subsection{Exploiting the almost compactness}
Now, we are at the point to exploit the almost compactness, as in Section
\ref{S5.2.1}. { Starting from this point we shall use systematically the
generalization of the Di-Perna, Lions transport theory formulated
in Proposition \ref{LP2}. { Essentially from this reason, we shall
need the density sequence $R_\delta$ uniformly integrable  in {$L^2(I;L^2(\Omega))$}, which amounts, in view of only local improved estimates of density (cf. (\ref{Bogd}), to assume
$\gamma\ge 2$, cf. (\ref{ts2d}).}

We define ${\mathfrak{s}}=\mathfrak{s}_{\mathfrak{d}}$ as in (\ref{renofrac+}) by using weak limits $Z$ and $R$. 
Combining (\ref{renofrac*}) with (\ref{renofrac!+})
one gets,
\begin{equation}\label{P2d}
\forall\tau\in\overline I,\ \lim_{\delta \to 0}\intO{R_\delta(\mathfrak{s}_{\delta}-\mathfrak{s})^2(\tau)}=0,\
\lim_{\delta \to 0}\int_0^\tau\int_{\Gamma^{\rm out}}{R_\delta(\mathfrak{s}_{\delta}-\mathfrak{s})^2(\tau)\vu_B\cdot\vc n{\rm d}S_x{\rm d}t}
\end{equation}
compare with (\ref{P2}). 
{Further, by interpolation,
$$
\forall \tau\in\overline I,\;\intO{R^p_\delta(\tau)(\mathfrak{s}_\delta-\mathfrak{s})^q(\tau)}\to 0,\;0<p<\gamma,\;0<q<\infty,
$$
and 
\begin{equation}\label{almostcd}
\int_0^T\int_K{R^p_\delta(t)(\mathfrak{s}_\delta-\mathfrak{s})^q}{\rm d} x{\rm d}t\to 0,\; 0<p<\gamma+\gamma_{\rm Bog}, \;0<q<\infty
\end{equation}
with any $K$ a compact subset of $\Omega$.\footnote{Relation (\ref{P2d}) implies the same type of convergence on $I\times\Gamma^{\rm out}$, but we do not need to exploit it here.}

Introducing
$$
\Pi(R,t,x)=P(R, R\mathfrak{s}(t,x))
$$
and writing
$$
P(R_\delta,Z_\delta)=[P(R_\delta,R_\delta{\mathfrak{s}}_\delta)-
P(R_\delta,R_\delta{\mathfrak{s}})]+P(R_\delta,R_\delta{\mathfrak{s}}) 
$$
we easily deduce by using hypotheses (\ref{dzP}), (\ref{grP}), (\ref{almostcd}) and available bounds,
$$
\overline{P(R,Z)}=\overline{\Pi(R,t,x)},
$$
cf. (\ref{eq6.6}). Thus, equation (\ref{E5!++d}) rewrites
\begin{equation} \label{E5!++d+}
\intO{ (\vr+z) \vu \cdot \bfphi } \Big|_{0}^{\tau} = 
\int_0^\tau \intO{ \Big[(\vr+z) \vu \cdot \partial_t \bfphi + (\vr+z) \vu \otimes \vu : \Grad \bfphi 
+ \overline{\Pi(R,t,x)} \Div \bfphi - \mathbb{S}(\nabla\vu): \Grad \bfphi \Big] }{\rm d}t
\end{equation}
for all $\bfphi\in C_c^1([0,T)\times\Omega))$.

\subsection{Effective viscous flux identity}
We use the same main idea as in section \ref{S5.2.2}. Nevertheless,
$R_\delta$ does not possess enough summability in order to carry out
the argument in the same way; it is well known from the mono-fluid case, that the way out of this is to use for the construction of test functions a 
convenient truncation ${\cal T}_k(R_\delta)$ of $R_\delta$, cf. (\ref{Tk}).

Using in the momentum equation (\ref{E5!+})--where $P_\delta(R,Z)(t,x)=
\Pi(R(t,x),t,x)$ by virtue of equation (\ref{eq6.6})-- test function
\begin{equation}\label{test1+}
\varphi(t,x)=\psi(t){ \phi(x)}(\nabla\Delta^{-1}({\cal T}_k(R_\delta)\phi))(t,x),\;\;\psi\in C^1_c(0,T),\;\phi\in C^1_c(\Omega)
\end{equation}
and in the momentum equation (\ref{E5!++d+}) test function 
\begin{equation}\label{test2*}
\varphi(t,x)=\psi(t){ \phi(x)}(\nabla\Delta^{-1}(\overline{{\cal T}_k(R})\phi))(t,x),\;\;\psi\in C^1_c(0,T),\;\phi\in C^1_c(\Omega),
\end{equation}
we get by the same reasoning as in the proof of Proposition \ref{p5}\footnote{A short inspection of the sketch of the proof of Proposition
\ref{p5} shows that the argument to get (\ref{eq6.8+}) passes provided
$\gamma>3/2$, cf. (\ref{cvep!}) and the reasoning after. We have, however,
used the treshold $\gamma\ge 2$ in order to prove that $\overline{P(R,Z)}
=\overline{\Pi(R,t,x)}$ in the previous section.}, the
effective viscous flux identity in the following form:
\begin{prop} \label{p5+}
\begin{equation} \label{eq6.8+}
{\Ov{\Pi (R,t,x) {\cal T}_k(R)}} -(2\mu+\lambda) \Ov{{\cal T}_k(R) \Div \vu} = {\Ov{\Pi (R,t,x)}} \,\Ov{{\cal T}_k( R)} -(2\mu+\lambda) \Ov{{\cal T}_k( R)}  \, \Div \vu
\end{equation}
fulfilled a.a. in { $I\times \Omega$.} 
\end{prop}
\subsection{Oscillations defect measure}
\begin{prop}\label{odm}
The sequence $\vr_\delta$ satisfies
\begin{equation}\label{odm1}
{\rm osc}_{\gamma+1}[R_\delta\rightharpoonup R](Q_T):=\sup _{k>1}\limsup_{\delta\to 0}
\int_{Q_T}|{\cal T}_k(R_\delta)-{\cal T}_k(R)|^{\gamma+1}\, {\rm d}x\, {\rm d} t<\infty.
\end{equation}
\end{prop}

Proposition \ref{odm} follows from the effective viscous flux identity derived in Proposition \ref{p5+}. To see this fact, we employ in (\ref{eq6.8+}) decomposition (\ref{Ps1}) in order to get 
\begin{equation}\label{last}
d\int_0^T\int_\Omega \Big(\overline{R^\gamma {\cal T}_k(R)}-\overline{R^\gamma}\;\overline{{\cal T}_k(R)}\Big)\, {\rm d}x\, {\rm d}t
+ \int_0^T\intO{\Big(
{ {\overline{\pi(R,t,x){\cal T}_k(R)}}-{\overline{\pi(R,t,x)}}\;\overline{{\cal T}_k(R)}}\Big)}\, {\rm d}t
\end{equation}
$$
=
(2\mu +\lambda)\limsup_{\delta\to 0} \Big[
\int_0^T\intO{\Big({\cal T}_k(R_\delta)-{{\cal T}_k(R)}\Big)\Div\vu_\delta}{\rm d}t+ \int_0^T\intO{\Big({\cal T}_k(R)-\overline{{\cal T}_k(R)}\Big)\Div\vu_\delta}{\rm d}t\Big].
$$

We first observe that the second integral at the left hand side is non negative according to Lemma \ref{Lemma4}, cf. (\ref{Ps1}). Second, we employ the H\"older inequality and interpolation together with the lower weak semi-continuity of norms in order to estimate the right hand side by 
{
\begin{equation}\label{dod8}
c  \Big[{\rm osc}_{\gamma+1}[R_\delta\rightharpoonup R](Q_T)\Big]^{\frac{1}{2\gamma}}
\end{equation}
with $c>0$ independent of $k$.
}

Concerning the firts term, we write,
$$
\int_0^T\int_\Omega \Big(\overline{R^\gamma {\cal T}_k(R)}-\overline{R^\gamma}\;\overline{{\cal T}_k(R)}\Big)\,{\rm d}x\,{\rm d}t
=\limsup_{\delta\to 0}\int_0^T\intO{ \Big(R_\delta^\gamma-R^\gamma\Big)\Big({\cal T}_k(R_\delta)-T_k(R)\Big)}\,{\rm d}t
$$
$$
+
\int_0^T\intO{\Big(R^\gamma-\overline{R^\gamma}\Big)\Big(\overline{{\cal T}_k(T)}-{\cal T}_k(R)\Big)}\,{\rm d}t
\ge \limsup_{\delta\to 0}\int_0^T\intO{ \Big|{\cal T}_k(R_\delta)-{\cal T}_k(R)\Big|^{\gamma+1}}\,{\rm d}t,
$$
where we have employed convexity of $R\mapsto R^\gamma$ and concavity of $R\mapsto {\cal T}_k(R)$ on $[0,\infty)$, and algebraic { inequalities}
$$
|a-b|^\gamma\le |a^\gamma-b^\gamma|\;\mbox{and}\;|a-b|\ge |T_k(a)-T_k(b)|, \; \ (a,b)\in [0,\infty)^2.
$$
Inserting the last inequality into  (\ref{last}) yields (in combination with estimates of the right hand side (\ref{dod8})) the statement of Proposition \ref{odm}.

\subsection{Strong convergence of density}
We know that continuity equation (\ref{renoep+}) is satisfied, in particular,  in the renormalized sense with renormalizing functions  $B={\cal L}_k$, and that (\ref{P3})$_{M=1,\mathfrak{r}=R}$ is satisfied in the renormalized sense with the same function ${\cal L}_k$, cf. Item 1. in Proposition \ref{LP2} and Item 1. in Remark \ref{rinflow}.
Using these equations with test function $\varphi=1$ (and noticing that
$z {\cal L}_k'(z) -{\cal L}_k(z) = {\cal T}_k(z)$), we get, in particular,
\begin{equation}\label{dod1*}
\intO{\Big({{\cal L}_k(R_\delta)}-{\cal L}_k(R)\Big)(\tau,x)\varphi(x)}
+
\int_0^\tau\int_{\Gamma^{\rm out}}{\Big({{\cal L}_k(R_\delta)}-{{\cal L}_k(R)}\Big)\vu_B\cdot\vc n{\rm d}S_x}{\rm d}t
\end{equation}
$$
= \int_0^\tau\int_{\Omega}({\cal T}_k(R)\Div \vu -\Ov{{\cal T}_k(R)}\Div \vu_\delta)\dx\dt \\
+ \int_0^\tau\int_\Omega (\Ov{{\cal T}_k(R)}-{\cal T}_k(R_\delta))\Div \vu_\delta \dx \dt 
$$
for all $\tau\in \overline I$. 

The { absolute} value of the first term at the right had side of the above identity is\footnote{Indeed, 
$$
\|{\cal T}_k(R)-\overline{{\cal T}_k(R)}\|_{L^1(Q_T)}\le \|{{\cal T}_k(R)}-R\|_{L^1(Q_T)}+\|\overline{{\cal T}_k(R)}-R\|_{L^1(Q_T)}
$$
$$
\le \|{\cal T}_k(R)-R\|_{L^1(Q_T)}+ \liminf_{\delta\to 0}\|{{\cal T}_k(R_\delta)}-R_\delta\|_{L^1(Q_T)}\to 0
$$
as $k\to\infty$.
} 
$$
\aleq 
\|\overline{{\cal T}_k(R)}-{\cal T}_k(R)\|_{L^2(Q_T)}\aleq \Big[{\rm osc}_{\gamma+1}[R_\delta\rightharpoonup R](Q_T)\Big]^{\frac {1}{2\gamma}}
\|\overline{{\cal T}_k(R)}-{\cal T}_k(R)\|_{L^1(Q_T)}\Big]^{\frac{\gamma-1}{2\gamma}}\to 0
$$
as $k\to\infty$,
while the second term can be expressed through Proposition \ref{p5+},
$$
\frac 1{2\mu +\lambda}\Big(\overline{\Pi(R,t,x){\cal T}_k(R)}-
\overline{\Pi(R,t,x)}\,\overline{{\cal T}_k(R)}\Big)
$$
and it is non negative due to Lemma \ref{Lemma4}.

Thus, letting in $\delta\to 0$ in identity (\ref{dod1*}) yields
$$
\intO{\Big(\overline{{\cal L}_k(R)}-{\cal L}_k(R)\Big)(\tau,x)}
+
\int_0^\tau\int_{\Gamma^{\rm out}}{\Big(\overline{{\cal L}_k(R)}-{{\cal L}_k(R)}\Big)\vu_B\cdot\vc n{\rm d}S_x}{\rm d}t\le 0.
$$
Recalling
\begin{equation}\label{LT}
\overline{{\cal L}_k(R)} \to\overline{R\log R},\;{\cal L}_k(R)\to R\log R
,\;\mbox{in $C_{weak}([0,T];L^q(\Omega))$ for any $1\leq q<\gamma$,}
\end{equation}
$$
\overline{{\cal L}_k(R)} \rightharpoonup\overline{R\log R},\;{\cal L}_k(R)\rightharpoonup R\log R
,\;\mbox{in $L^\gamma(0,T;L^q(\Gamma^{\rm out};|\vu_B\cdot\vc n|{\rm d}S_x))$ for any $1\leq q<\gamma$,}
$$
we arrive finally at
$$
\int_\Omega (\overline{R\log R}-R\log R)(\tau,\cdot)\dx
+
\int_0^\tau\int_{\Gamma^{\rm out}}(\overline{R\log R}-R\log R)\vu_B\cdot\vc n{\rm d}S_x{\rm d}t
\le 0
$$
which implies in virtue of the strict convexity of $R\mapsto R\log R$ on $[0,\infty)$
\begin{equation}\label{aed}
R_\delta\to R\;\mbox{a.a. in $Q_T$}\ \mbox{and a.a. in $I\times \Gamma^{\rm out}$}
\end{equation}
{ This relation in combination with (\ref{P2d}) yields also $
Z_\delta\to Z$ a.e. in $Q_T$ and a.e. in $I\times \Gamma^{\rm out}$. This yields, in particular,\footnote{The statement about a.a. convergence in $I\times \Gamma^{\rm out}$ of $R_\delta$ and $Z_\delta$ estblished in (\ref{aed}) is not needed at this place. It is stated for the sake of completeness.}
$$
\overline{P(R,Z)}= P(R,Z)
$$
and, in view of (\ref{E5!++d+}), finishes the proof of the momentum equation (\ref{eq2.8}).
}

\subsection{Energy inequality}

Now, it is rather standard to pass to the limit in the energy inequality (\ref{E7+}) and to obtain energy inequality (\ref{eq2.9}). In this respect, a few observations are in order: 
\begin{enumerate}
\item In order to pass to the limit in the term
$\intO{(\vr_\delta+z_\delta)|\vv_\delta|^2(\tau)}$, we rewrite it in the form $\intO{\mathbb{E}(\vr_\delta+z_\delta,\vc q_\delta)(\tau)}$, $\vc q_\delta=(\vr_\delta+z_\delta)\vv_\delta$, where 
$$
\mathbb{E}(r,\vc q)=
\left\{\begin{array}{c}
            \frac{{\vc q}^2}{r}\;\mbox{if $r> 0$}\\
            0\,\mbox{if $r=0$}\\
            \infty\,\mbox{otherwise}
           \end{array}\right\}
$$
is the lower semicontinuous convex function on $\R\times\R^3$, and use lower semicontinuity of the associated functional, cf. Lemma \ref{Lemma2}.
 \item In the passage in the term $\intO{H(R_\delta,Z_\delta)}$ we consider
 first the limit $\delta\to 0$ in $\frac 1{2\xi}\int_{\tau-\xi}^{\tau+\xi}\int_\Omega H(R_\delta,Z_\delta)$ ${\rm d}x{\rm d}t$, $0<\xi<\tau/2$ by using e.g. 
 Fatou's lemma (or lower weak semicontinuity of convex functional $H$) 
 and then let $\xi\to 0$ employing the Theorem on Lebsegue points.
 \item In the passage in the term $\int_0^\tau\int_{\Gamma^{\rm out}}
 H(R_\delta,Z_\delta)\vu_B\cdot\vc n{\rm d}S_x{\rm d}t$ we use the weak convergence induced by estimate (\ref{ts2d}) and the lower weak semicontinuity of convex functionals, cf. Lemma \ref{Lemma2}.\footnote{The a.a. convergence in $I\times \Gamma^{\rm out}$ established in (\ref{aed}) is not needed in the setting when $H$ is convex. It would however be necessary if $H$ is not convex, cf. Remark 
 \ref{RemH}, namely assumption (\ref{convH+}).}
 \item In the passage in the term $\int_0^\tau\int_{\Omega}
 P(R_\delta,Z_\delta){\rm div}\vu_B{\rm d}x{\rm d}t$ we consider
 a.e. convergence deduced in the previous section in combination
 with estimate (\ref{Bogd}) and complete with estimate
 (\ref{PBogd+}) to pass to the limit near the boundary.
\end{enumerate}

This finishes the proof of Theorem \ref{theorem1}.

\section{From the "academic" to the "realistic" bifluid system}\label{Realistic}

We set
$$
F(\alpha)=\frac 1 {f(\alpha)},\; G(\alpha)=\frac 1 {g(\alpha)},
$$
where clearly $F$, $G$ are strictly monotone, { strictly positive} functions on interval (0,1), and denote
$$
\underline F={\rm min}\{F(\underline\alpha), F(\overline\alpha)\},\; \overline F= {\rm max}\{F(\underline\alpha), F(\overline\alpha)\}.
$$
Similarly, we define numbers $\underline G$, $\overline G$ as above replacing function $F$ by $G$. Then, in particular,
$F: [\underline\alpha,\overline\alpha]\to [\underline F,\overline F]$, 
$G: [\underline\alpha,\overline\alpha]\to [\underline G,\overline G]$ are  $C^1$-diffeomorphisms.

We will use Theorem \ref{theorem1} with initial and boundary conditions
$$
\vr_0,\; z_0,\; R_0=f(\alpha_0)\vr_0,\; Z_0=g(\alpha_0)z_0,\; \vu_0,
$$
$$
\vr_B,\; z_B,\; R_B=f(\alpha_B)\vr_B,\; Z_B=g(\alpha_B)z_B,\; \vu_B.
$$
such that 
$$
\underline a R_0\le Z_0\le \overline a R_0.
$$
We easily verify that
$$
\underline F R_0\le \vr_0\le \overline F R_0,\; \underline G Z_0\le z_0\le \overline G Z_0.
$$

Theorem \ref{theorem1} guarantees existence of  a bounded energy weak solution 
$(\vr,z,R,Z,\vu)$ in the corresponding regularity class described in that theorem which satisfies, in particular, the domination relations
$$
\forall t\in\overline I,\;\underline F R(t)\le \vr(t)\le \overline F R(t),\;
\underline G Z(t)\le z(t)\le \overline G Z(t)\;\mbox{a.e. in $\Omega$},
$$
$$
\mbox{for a.a. $t\in I$},\;\underline F R(t)\le \vr(t)\le \overline F R(t),\;
\underline G Z(t)\le z(t)\le \overline G Z(t)\;\mbox{a.e. in $\partial\Omega$}.
$$
We set 
\begin{equation}\label{talpha}
\forall t\in\overline I,\;\alpha(t)=F^{-1}\Big(\vr(t)/_{\mathfrak{d}} R(t)\Big),\; \tilde\alpha(t)={ G^{-1}}\Big(z(t)/_{\mathfrak{d}} Z(t)\Big)\;\mbox{a.e. in $\Omega$},
\end{equation}
$$
\mbox{for a.a. $(t,x)\in I\times\Gamma^{\rm out}$},\ 
\alpha(t,x)=F^{-1}\Big(\vr(t,x)/_{\mathfrak{d}} R(t,x)\Big),\; \tilde\alpha(t,x)={ G^{-1}}\Big(z(t,x)/_{\mathfrak{d}} Z(t,x).
$$
Each of the quantities $\alpha$ and $\tilde\alpha$ satisfies the pure transport equation (\ref{eq2.7+}) with transporting velocity $\vu$ and with the same  initial and boundary  conditions. By Item 2. of Proposition
\ref{LP2}, the same is true for $f(\alpha)$ and $g(\alpha)$. Therefore,
$$
\forall t\in\overline I,\;R(t)=(f(\alpha)\vr)(t)=(f(\tilde\alpha)\vr)(t),\;
Z(t)=(g(\tilde\alpha)z)(t)=g(\alpha)z(t),
$$
$$
\mbox{for a.a. $(t,x)\in I\times\Gamma^{\rm out}$},\ \;R(t,x)=(f(\alpha)\vr)(t,x)=(f(\tilde\alpha)\vr)(t,x),\;
Z(t,x)=(g(\tilde\alpha)z)(t,x)=g(\alpha)z(t,x),
$$
where we have used the almost uniqueness established in Corollary \ref{cont-tr}. 

Consequently:
\begin{enumerate}
 \item In the momentum equation (\ref{eq2.8}),
 $$
 \int_0^\tau\intO{P(R,Z){\rm div}\varphi}{\rm d}t=
 \int_0^\tau\intO{P(f(\alpha)\vr,g(\alpha)z){\rm div}\varphi}{\rm d}t.
 $$
 This yields momentum equation (\ref{eq2.8-}).
 \item In the energy inequality (\ref{eq2.9}),
 $$
 \mbox{for all $\tau\in\overline I$},\;
 [\intO{H(R,Z)}](\tau)=
 [\intO{H(f(\alpha)\vr,g(\alpha)z)}](\tau),
 $$
 $$
 \int_0^\tau\intO{P(R,Z){\rm div}\vu_B}{\rm d}t=
 \int_0^\tau\intO{P(f(\alpha)\vr,g(\alpha)z){\rm div}\vu_B}{\rm d}t
 $$
 while
 $$
  \int_0^\tau\int_{\Gamma^{\rm out}}H(R,Z)\vu_B\cdot\vc n{\rm d}S_x{\rm d}t= \int_0^\tau\int_{\Gamma^{\rm out}}
  H(f(\alpha)\vr,g(\alpha)z)\vu_B\cdot\vc n{\rm d}S_x{\rm d}t.
 $$
 This yields energy inequality (\ref{eq2.9-}).
 \end{enumerate}
 Theorem \ref{theorem2} is thus proved.

\section{Appendix}\label{Appx}

\subsection{Some elements of functional, convex and harmonic analysis}\label{Appx1}

We recall some properties
of the pseudodifferential operator
\begin{equation}\label{calA}
{\cal A}[u]=\Grad\Delta^{-1}[u] := \mathcal{F}^{-1}_{\xi \to x} \left[
\frac{-{\rm i}\xi}{|\xi|^2 } \mathcal{F}_{x \to \xi} [u] \right],\;u\in C^\infty_c(\Omega)\subset C^\infty_c (\R^3),
\end{equation}
where
${\cal F}_{x\to\xi}$ denotes the Fourier transform}.

We have the following lemma (a consequence of the H\"ormander-Michlin multiplier theorem and Sobolev imbeddings):
\begin{lem}\label{LcalA}
For all $u\in C^{\infty}_c(\Omega)$ there holds,
$$
\|{\cal A}[u]\|_{W^{1,p}(\Omega)}\aleq \|u\|_{L^p(\Omega)},\; 1<p<\infty.
$$
Consequently,
the operator ${\cal A}$ admits an extension (denoted by the same symbol) which is a continuous linear operator from
$$
L^p(\Omega)\;\mbox{to $W^{1,p}(\Omega)$}, \ 1<p<\infty
$$
\end{lem}

The next lemma deals with a particular solution of the equation ${\rm div}\vc w=r$ in  $\Omega$,  $\vc w|_{\partial\Omega}=0$ 
called the Bogovskii solution, \cite{Bog}.

\begin{lem}\label{LBog}
Let $\Omega$ be a bounded Lipschitz domain. There exists a linear operator ${\cal B}$ defined on ${C}^\infty_c(\Omega)$
with the following properties
\begin{equation}\label{Z0}
{\cal B}(C^\infty_c(\Omega))\subset  C^\infty_c(\Omega)
\end{equation}
\begin{equation}\label{Z2}
\forall r\in {C^\infty_c(\Omega)},\quad\left\| \mathcal{B}[r] \right\|_{W^{1,p}(\Omega;\R^3)} \aleq \| r\|_{L^p(\Omega)},\ 1 < p < \infty,
\end{equation}
\begin{equation}\label{Z3}
\forall {r}\in {C^\infty_c(\Omega;\R^4)},\quad
\left\| \mathcal{B}[{r}] \right\|_{L^q(\Omega;\R^3)} \aleq \| {r} \|_{[W^{1,q'}(\Omega)]^*},\ 1 < q < \infty.
\end{equation}
\begin{equation}\label{Z1}
\Div \mathcal {B}[r] = r-\frac 1{|\Omega|}\intO{r}. 
\end{equation}
Consequently, the operator ${\cal B}$ admits an extension (denoted by the same symbol) which is a contiuous linear operator
from
$$
L^p(\Omega)\mapsto W_0^{1,p}(\Omega), \ 1<p<\infty,
$$
and from
$$
[W^{1,q'}]^*(\Omega)\mapsto L^q(\Omega),\ \mbox{where }\
q\in \left\{
\begin{array}{c}
 (1, \frac{3p}{3-p})\;\mbox{if $1<p<3$}\\
 (1,\infty)\; \mbox{if $p\ge 3$}
\end{array}
\right\}.
$$
Moreover, ${\cal B}$ satisfies (\ref{Z1}) for any $r\in L^p(\Omega)$.
\end{lem}

On Lipschitz domains, the Bogovskii solution is given by an explicit formula involving a singular kernel which is particularly "accessible" if the domain is star-shaped and which allows to provide the proof of Lemma \ref{LBog} via an explicit
(but involved) calculation. We refer to Galdi \cite[Chapter 3]{Ga} for a detailed proof of the properties (\ref{Z0}), (\ref{Z2}), (\ref{Z1}), and to Geissert, Heck and Hieber \cite{GeHeHi} for (\ref{Z3}).  The admissible values $p,q$ follow
from (\ref{Z0}--\ref{Z3}), the  density of $C^\infty_c(\Omega)$
in $[W^{1,q'}]^*(\Omega)$ and Sobolev imbeddings.

The next theorem involving commutator of Riesz operators may be seen as a consequence
of the celebrated Div-Curl lemma above, see Murat, Tartar \cite{Mu} and
\cite[Section 6]{EF70} or \cite[Theorem 10.27]{FeNoB} for the below adapted formulation

\begin{lem}\label{rieszcom} Let
$$
\vc{ V}_n \rightharpoonup \vc{ V} \ \mbox{ in}\ L^p(\R^3; \R^3),
$$
$$
\vc{ U}_n \rightharpoonup \vc{ U} \ \mbox{ in}\ L^q(\R^3; \R^3),
$$
where $ \frac{1}{p} + \frac{1}{q} = \frac{1}{s} < 1$. Then
$$
\vc{ U}_n \cdot \Grad\Delta^{-1}\Div[\vc{V}_n ] -\vc V_n\cdot \Grad\Delta^{-1}\Div
[\vc{
U}_n]\cdot \vc{ V}_n \rightharpoonup \vc{ U}\cdot \Grad\Delta^{-1}\Div[\vc{ V}] -
\vc V\cdot\Grad\Delta^{-1}\Div[\vc{ U}] \ \mbox{ in}\ L^s(\R^3).
$$
\end{lem}

Finally, the last two lemmas are well known results from convex analysis, see e.g. Lemma 2.11 and Corollary 2.2 in
Feireisl \cite{EF70}.

\begin{lem}\label{Lemma2}
Let $O \subset \R^d$, $d\ge 2$, be a measurable set and $\{ \vc{v}_n
\}_{n=1}^{\infty}$ a sequence of functions in $L^1(O; \R^M)$ such
that
$$
\vc{v}_n \rightharpoonup \vc{v} \ \mbox{ in}\ L^1(O; \R^M).
$$
Let $\Phi: R^M \to (-\infty, \infty]$ be a lower semi-continuous
convex function such that $\Phi(\vc v_n)$ is bounded in  $L^1({ O})$.

Then $\Phi(\vc v):O\mapsto R$ is integrable and
$$
\int_{O} \Phi(\vc v){\rm d} x\le \liminf_{n\to\infty} \int_{O} \Phi(\vc v_n){\rm d} x.
$$
\end{lem}

\begin{lem}\label{Lemma3}

Let $O \subset \R^d$, $d\ge 2$ be a measurable set and $\{ \vc{v}_n
\}_{n=1}^{\infty}$ a sequence of functions in $L^1(O; \R^M)$ such
that
$$
\vc{v}_n \to \vc{v} \ \mbox{weakly in}\ L^1(O; \R^M).
$$
Let $\Phi: \R^M \to (-\infty, \infty]$ be a lower semi-continuous
convex function such that $\Phi (\vc{v}_n ) \in L^1(O)$ for
any $n$, and
$$
\Phi (\vc{v}_n) \to \Ov{\Phi(\vc{v})}
\ \mbox{weakly in}\ L^1(O).
$$

Then

\begin{equation}\label{FNeb4}
\Phi (\vc{v})  \leq \Ov{\Phi (\vc{v})} \ \mbox{a.e. in}\ O.
\end{equation}

If, moreover, $\Phi$ is strictly convex on an open convex set
$U \subset \R^M $, and
$$
\Phi(\vc{v}) = \Ov{\Phi (\vc{v})} \ \mbox{a.e. on}\ O,
$$
then

\begin{equation}\label{FNeb5}
\vc{v}_n  (\vc{y}) \to \vc{v} (\vc{y})
\ \mbox{for a.a.}\ \vc{y} \in \{
\vc{y} \in O \ | \ \vc{v}(\vc{y}) \in U \}
\end{equation}
extracting a subsequence as the case may be.
\end{lem}

\begin{lem} \label{Lemma4}
Let $O$ be a domain in $\R^d$, $P,G:O\times[0,\infty)\mapsto [0,\infty)$ be a couple of functions such that
for almost all $y\in O$,  $\varrho\mapsto P(y,\varrho)$ and $\varrho\mapsto G(y,\varrho)$
are both non decreasing and continuous on $[0,\infty)$. Assume that  $\varrho_n\in
L^1(O;[0,\infty))$ is a sequence such that
$$
\left.\begin{array}{c}
P(\cdot,\vr_n) \rightharpoonup \overline{P(\cdot,\vr)}, \\
G(\cdot,\vr_n) \rightharpoonup \overline{G(\cdot,\vr)}, \\
P(\cdot,\vr_n)G(\cdot,\vr_n) \rightharpoonup \overline{P(\cdot,\,\vr)G(\cdot\vr)}
\end{array} \right\} \mbox{ in } L^1(O).
$$
Then
$$
\overline{P(\cdot,\vr)}\, \, \overline{G(\cdot,\vr)} \leq
\overline{P(\cdot,\vr)G(\cdot,\vr)}
$$
a.e. in $O$.
\end{lem}

The last lemma we wish to recall is the Friedrichs lemma on commutators, see e.g; Di-Perna, Lions \cite{DL}.

\begin{lem}[Friedrichs commutator lemma] \label{Friedrichs}
Let $I\subset R$ be an open bounded interval and $f \in L^\alpha(I;$ $L^\beta_{{\rm loc}}(\R^d))$, $\vu \in L^p(I;W^{1,q}_{{\rm loc}}(\R^d;\R^d))$. Let $1\leq q,\beta\leq \infty$, $(q,\beta)  \neq (1,\infty)$, $\frac 1q + \frac 1{\beta} \leq 1$, $1\leq \alpha \leq \infty$ and $\frac 1\alpha + \frac 1p \leq 1$.  Then
$$
{\rm div}([\vu f]_\ep) - {\rm div}(\vu [f]_\ep) \to 0
$$
strongly in $L^t(I;L^r_{{\rm loc}} (\R^d))$, where 
$$
\frac 1t \geq \frac 1\alpha + \frac 1p,\;  t\in [1,\infty)
$$
and 
$$
r\in [1,q) \text{ for } \beta = \infty, \ q \in (1,\infty],
$$ 
while $\frac 1\beta + \frac 1q \leq \frac 1r \leq 1$ otherwise. 
In the above $[f]_\ep$ denotes the mollifications of $f$ over the space
variables via the convolution of $f$ with the standard regularizing kernel.
\end{lem}

\end{document}